\documentclass[reqno,12pt,letterpaper]{amsart}

\usepackage[mathcal]{euscript}
\usepackage{mathrsfs, amssymb, latexsym, amsmath, graphicx, stmaryrd}
\usepackage{psfrag,units,bbm,xcolor,enumitem} 
\input diagxy
\usepackage[OMLmathsfit]{isomath}

\SetSymbolFont{stmry}{bold}{U}{stmry}{m}{n}

\allowdisplaybreaks[4]

\usepackage[T1]{fontenc}

\newcommand{\parag}{\S}

\newcommand{\roots}{\Phi}
\newcommand{\sroots}{\boldsymbol{\triangle}}
\newcommand{\rootl}{Q}

\DeclareMathAlphabet{\mathpzc}{OT1}{pzc}{mx}{it}

\newtheorem{thm}{Theorem}[section]
\newtheorem{prop}[thm]{Proposition}
\newtheorem{cor}[thm]{Corollary}
\newtheorem{lemma}[thm]{Lemma}
\newtheorem{definition}[thm]{Definition}

\theoremstyle{remark}

\newcounter{thmcount}

\newenvironment{thmlist}{
\begin{list}
{{\rm (\alph{thmcount})}} {
\usecounter{thmcount}
\setcounter{thmcount}{0}
\setlength{\topsep}{0bp} 
\setlength{\leftmargin}{16bp} 
\setlength{\itemindent}{-12bp}
\setlength{\itemsep}{4bp}
\setlength{\parsep}{0bp}
}
}{\end{list}}

\renewcommand{\le}{\leqslant}

\newcommand{\fsl}[1]{\mathfrak{s}\mathfrak{l}_{#1}}
\newcommand{\Ufsl}[2]{U_{#1}(\fsl{#2})}

\newcommand{\pwrt}{\tau^{SO(3)}_\zeta}

\renewcommand{\H}{{\mathcal H}}
\newcommand{\A}{{\mathscr A}}

\newcommand{\bbz}{\mathbb{Z}}

\newcommand{\zz}{\bbz[\zeta]}
\newcommand{\Zz}{\mathscr{Z}_\zeta}

\newcommand{\bbd}{\mathbbm{k}}
\newcommand{\D}{{\mathcal D}}
\renewcommand{\dh}{\D(H)}

\newcommand{\Y}{\nabla}
\newcommand{\F}{{\textsf{\textbf F}}}
\newcommand{\GF}{{\textsf{\textbf G}}}
 
\newcommand{\fupsilon}{\mbox{\boldmath$\mathsf \Upsilon$\unboldmath}}
\newcommand{\Fa}[1]{ \fupsilon_{\!\![#1]}}
\newcommand{\Fao}[1]{ \overline{\fupsilon}_{\!\![#1]}}
\newcommand{\Fae}[1]{ \overline{{\mathsf \Upsilon}}_{\!#1}}

\newcommand{\xt}{{\textsf{\textbf x}}}
\newcommand{\zt}{{\textsf{\textbf z}}}

\newcommand{\Jr}{\mathcal J}
\newcommand{\Jl}{\overline{\mathcal J}}
\newcommand{\JJ}{\mathcal T}

\newcommand{\byeq}[1]{\stackrel{\mathrm{by (\ref{#1})}}{=\!=\!=}}

\newcommand{\commlabel}[1]{\mbox{{\LARGE$\circlearrowleft$}\hspace*{-3.9mm}\raise 0.16 ex\hbox{\small \bf #1}}}

\newcommand{\hlk}{{\nicefrac{1}{2}}}

\newcommand{\isto}{\rightarrowtail\mkern -18mu \rightarrowtail}
\newcommand{\listo}{\mbox{\large $\rightarrowtail\mkern -18mu \rightarrowtail$}}

\newcommand{\mathsmall}[2]{\mbox{\fontsize{#1}{1}$#2$}}

\newcommand{\qchoose}[2] {\Bigl[{\begin{matrix}\raise 0.1ex \hbox{\mathsmall{10}{#1}}\vspace*{-1.8mm}  \\ \mathsmall{10}{#2}\end{matrix}}\Bigr]}

\newcommand{\cobB}{\mathcal{C}\hspace{-.2ex}\mathit{o\hspace{-.2ex}b}} 
\newcommand{\cobn}{\mbox{\boldmath $\cobB\hspace{-.1ex}^{\mbox{\small\boldmath$\scriptstyle\emptyset$}}$}}
\newcommand{\cobc}{\mbox{\boldmath $\cobB\hspace{-.1ex}^{\bullet}$}}

\newcommand{\Sigmab}{\Sigma^{\scriptscriptstyle\bullet}}

\newcommand{\cyconC}{{\mkern -0.5mu \succ\mkern -14.3mu \rule[-0.5pt]{0.27pt}{6.9pt} \mkern 2mu \rule[-0.5pt]{0.25pt}{6.9pt}\mkern 6mu}}

\newcommand{\fillf}{{\mathscr F}\!_{\partial}}

\newcommand{\tang}{\mathcal{TC}}

\newcommand{\Tang}{\mbox{\boldmath $\tang\hspace{-.1ex}^{\bullet}$}}
\newcommand{\Tangn}{\mbox{\boldmath $\tang\hspace{-.1ex}^{\emptyset}$}}

\newcommand{\Diag}{\mbox{\boldmath{$\mathcal{D}\!c$}}}
\newcommand{\sDiag}{\mbox{\boldmath{$\scriptstyle \mathcal{D}\!c$}}}

\newcommand{\Surg}{\mathscr{K}}
\newcommand{\Surgf}{\mathscr{K}^*}

\newcommand{\eval}{\mathcal{E}}

\newcommand{\decor}[1]{\mathcal{Z}_{#1}}

\newcommand{\Ins}{\mathcal I}

\newcommand{\tangNM}{\mathcal{T}\hspace{-.45ex}\mathit{g}\hspace{-.25ex}\mathit{l}}

\newcommand{\TangNM}{\mbox{\boldmath $\tangNM$}}
\newcommand{\sTangNM}{\mbox{\boldmath $\scriptstyle\tangNM$}}

\newcommand{\Tmv}{\mathscr M\hspace{-.1ex}^{\bullet}}
\newcommand{\Tmvn}{\mathscr M\hspace{-.04ex}^{\emptyset}}

\newcommand{\Sm}{\mathcal S}
\newcommand{\Gm}{\Gamma}

\newcommand{\invf}[1]{\mathrm{Inv}_{#1}}

\newcommand{\HpDcat}[2]{#1 \!\vDash \!\mathfrak{proj}(#2)}
\newcommand{\pDcat}[1]{\mathfrak{proj}(#1)}
\newcommand{\HfDcat}[2]{#1  \!\vDash  \!\mathfrak{free}(#2)}
\newcommand{\fDcat}[1]{\mathfrak{free}(#1)}

\renewcommand{\to}{\rightarrow}

\newcommand{\MM}{{\mathfrak m}(\bbd)}

\newcommand{\ld}{\lambda_\D}

\newcommand{\Vf}{\mbox{\large$\mathscr V$}}
\newcommand{\tqftVc}[1]{\Vf_{\!#1}^{\bullet}}
\newcommand{\tqftVn}[1]{\Vf_{\!#1}^{\emptyset}}
\newcommand{\Phc}{\tqftVc{\dh}}
\newcommand{\Ph}{\tqftVn{\dh}}
\newcommand{\PH}{\tqftVn{\H}}
\newcommand{\PHc}{\tqftVc{\!\!\H}}
\newcommand{\PHFc}{\tqftVc{\!\!\H_{\F}}}
\newcommand{\PBz}{\tqftVc{\dBz}}

\newcommand{\natto}{\stackrel{\bullet}{\longrightarrow}}

\newcommand{\GTnatiso}{\mbox{\boldmath$\gamma$}}

\newcommand{\HI}[1]{\mbox{\large \raise 0.3ex \hbox{$\varphi$}}_{#1}}
\newcommand{\ph}{\HI{\dh}}
\newcommand{\pH}{\HI{\H}}

\def\parag{\mathhexbox278}

\newcommand{\inv}{\mathrm{Inv}}
\newcommand{\Bz}{B_\zeta^{\circ}(\fsl2)}
\newcommand{\dBz}{\D(\Bz)}
\newcommand{\leg}[2]{\left(\frac {#1}{#2}\right)}
\newcommand{\h}{\mathrm{h}}

\newcommand{\pslz}{\HI{\dBz}}

\renewcommand{\r}{\mathbf{r}}

\newcommand{\lar}{\leftharpoonup}
\newcommand{\rar}{\rightharpoonup}
\renewcommand{\S}{S_{\mathcal D}} 
\newcommand{\DLambda}{\Lambda_{\mathcal D}} 
\newcommand{\R}{{\mathcal R}}
\newcommand{\M}{{\mathcal M}}
\newcommand{\st}{\ |\ }
\newcommand{\bbf}{\mathbb{F}}

\newcommand{\db}{{\D(B)}}

\newcommand{\tz}{\tau_\zeta}
\newcommand{\uz}{U_\zeta}
\newcommand{\into}{\lambda_\db}
\newcommand{\ia}{\lambda_\A}
\newcommand{\intp}{\lambda_{\uz}}
\newcommand{\hdb}{\D_e(B)}
\newcommand{\rd}{\r_{\db}}

\newcommand{\kd}{\kappa_{\db}}

\newcommand{\Ru}{\R_{\uz}}
\newcommand{\Rd}{\R_{\db}}
\newcommand{\Dd}{D_{\db}}
\newcommand{\Du}{D_{\uz}}

\newcommand{\textunknot}[2]{\begin{boldmath}$\,_{#1}\bigcirc_{#2}$\end{boldmath}}
\newcommand{\texthopf}[2]{\begin{boldmath}$\,_{#1}\mkern -5mu \bigcirc \mkern -13mu\bigcirc_{#2}$\end{boldmath}}
\newcommand{\tenhom}{\mbox{\large$\boldsymbol{\tau}$}}
 
\newcommand{\ndiv}{\not\hspace{-3pt}|\hspace{2pt}}

\newcommand{\BDelta}{\Delta_*}

\newcommand{\sfbib}[2]{\bibitem[{\sffamily{\footnotesize #1}}]{#2}}

 \newcommand{\VcDBUA}{\mbox{\boldmath$\eta$}}

\newcommand{\faciso}{\varpi}
\newcommand{\facisoNat}{\mbox{\boldmath$\faciso$}}

\newcommand{\nicehopf}{topogenic\ }

\newcommand{\Bbox}{\raise -2.5pt \hbox{\Large\boldmath{$\Box$}}}

\newcommand{\ndeg}{\boldsymbol{\nu}}

\newcommand{\smallBz}{B_{\zeta}'(\mathfrak g)}

\newcommand{\tops}{\boldsymbol{\star}}

\newcommand{\crass}{\boldsymbol{\kappa}}
 
\newcommand{\grint}{\mathbfit L}

\newcommand{\comod}{\alpha_{\zeta}}
\newcommand{\hcomod}{\beta_{\zeta}}

\newcommand{\sqw}{{\mathbfit w}}

\newcommand{\ztheta}{\theta_{\zeta}}

\newcommand{\dOne}{\hat d}
\newcommand{\ellOne}{\hat \ell}
\newcommand{\ellTwo}{\ellOne'}
\newcommand{\ellThree}{\ellOne''}
 
\definecolor{commcol}{rgb}{0.7,0, 0}

\newcommand{\lint}[1]{\int^{\ell}_{#1}}
\newcommand{\rint}[1]{\int^{\r}_{#1}}

\newcommand{\lbl}{\mathsfit{b}}
\newcommand{\lgl}{\mathsfit{g}}
\newcommand{\lblz}{\lbl_{\zeta}}

\newcommand{\ttau}{{\vartheta}_{\ell}}

\def\cprime{$'$}

\title[Integrality and Gauge Dependence of Hennings TQFTs]{Integrality and Gauge Dependence \\ of Hennings TQFTs}
\author{Qi Chen\ \  and\ \  Thomas Kerler}

\address{Department of Mathematics, Winston-Salem State University\newline\indent Winston Salem, NC 27110, USA}
\email{chenqi@wssu.edu}
\address{Department of Mathematics, The Ohio State University\newline\indent Columbus, OH 43210, USA}
\email{kerler.2@osu.edu}

\begin{document}

\begin{abstract}
We provide a general construction of integral TQFTs over a general commutative ring, $\bbd$,
starting from a finite Hopf algebra over $\bbd$ which is Frobenius and double
balanced. These TQFTs specialize to the Hennings invariants of the  respective doubles 
on closed 3-manifolds.

We show the construction applies to index 2 extensions of  the Borel parts of Lusztig's 
small quantum groups for all simple Lie types, yielding integral TQFTs over the cyclotoic integers
for surfaces with boundary. 

We further establish and compute isomorphisms of TQFT functors constructed from Hopf algebras that 
are related by a strict gauge transformation in the sense of Drinfeld. Formulas for the natural
isomorphisms are given in terms of the gauge twist element.

These results are combined and applied to show that the Hennings invariant associated to 
quantum-$\fsl2$ takes values in the cyclotomic integers. Using prior results of Chen et al 
we infer integrality also of the Witten-Reshetikhin-Turaev $SO(3)$ invariant for rational
homology spheres.  

As opposed to most other approaches the methods described in this article do not 
invoke   calculations of skeins, knots polynomials, 
or representation theory, but follow a combinatorial construction that uses only 
the elements and operations of the underlying Hopf algebras. 
\end{abstract}

\maketitle

\begin{center}
 \today

\begin{minipage}[t]{5in}
 \small 
\tableofcontents
\end{minipage} 

\end{center}

\section{Introduction and Main Results}\label{s1}
 
The mathematical axiomatization of  2+1-dimensional {\em Topological Quantum Field Theory} (TQFT) due 
to Atiyah \cite{at} has been refined in numerous ways since its original formulation. Among the most 
interesting generalizations is the extension of the definition of a TQFT to a functor from a particular
cobordism category to the category of $\bbd$-modules for some commutative ring $\bbd$, generalizing the  
category of complex vector spaces. A sufficiently rich ideal structure of $\bbd$ will then allow to 
capture more  subtle topological information.

The most prominent example are TQFTs based on the famous Witten-Reshetikhin-Turaev (WRT) construction 
\cite{rt,tu10}, which are a-priori formulated over the  complex numbers. 
It was soon realized in a series of
articles \cite{mur,mr,le93,mw,g2,cl05,habiro4} that the WRT-type invariants and TQFTs can be constructed,
for adequately restricted cobordisms, 
over the cyclotomic integers $\bbd=\bbz[\zeta]$ for a given root of unity $\zeta\,$. The ideal and 
integrality structure of $\bbz[\zeta]$ has
had useful topological applications as described, for example, in \cite{fk01,gkp02,cl04}.

\smallskip

In this article we establish analogous integrality results for the family of so called the Hennings TQFTs 
associated to a ribbon Hopf algebra $\H$, as developed and studied in  \cite{hennings,kr,kerler3,o3}.
Their construction starts, as for the WRT invariants, from surgery presentations but uses 
elements of $\H$ directly in its algebraic assignments and computations. The Hennings construction thus 
circumvents the use of the representation theory of $\H$ as in 
\cite{rt,tu10} or the respective (essentially equivalent) combinatorial skein theory as in \cite{bhmv}.

The Hennings invariants (and TQFTs) coincide with WRT-type theories in the case of semisimple algebras
$\H$ as both approaches may be understood as special cases of the same universal theory \cite{kl}. 
For non-semisimple ribbon Hopf
algebras, however, the two flavors of TQFTs exhibit manifestly different behaviors for manifolds with
non-trivial rational homology. Nevertheless, in the case of  quantum $\mathfrak{sl}_2\,$ at a root of unity,
close relations between the associated Hennings and WRT TQFTs can still be established. See, for example, 
\cite{cks,ke94,fgst} as well as Theorem~\ref{thm2} below. It is thus to be expected that the TQFTs 
constructed here also for higher rank Lie types are closely related to the respective higher rank 
WRT theories as well. 

\smallskip

\subsection{Some Basic Terminology}\label{ss-intro-term}

The primary cobordism category we consider is denoted $\cobc$ and has connected, compact, oriented 
surfaces with one boundary  component as objects and classes of relative, 2-framed cobordisms as morphisms.
The 2-framing may be more conveniently thought of as the signature of a bounding 4-manifold 
\cite{atf} or also the signature of a certain closure 
of the framed tangle representing the cobordism \cite{ke99}. The respective category $\cobn$  with closed
surfaces is related by an obvious fill functor $\;\fillf :\;\cobc \rightarrow \cobn\;$ 
which pastes a disk into the boundary 
component of a surface. See \cite{kl} and Section~\ref{s2} for more detailed definitions.

The target categories of the TQFT functors are defined with respect to a ribbon Hopf algebra 
$\H$ over a commutative ring $\bbd$, such that
$\H$ is projective and finitely generated as a $\bbd$-module.  
Throughout this article we also assume all algebras and rings to be unital and 
Hopf algebras to have antipodes. 

We denote by $\pDcat{\bbd}$ the category of finitely generated 
projective $\bbd$-modules,
and by  $\,\HpDcat{\H}{\bbd}\,$ the category of $\H$-modules which are finitely generated and projective 
as $\bbd$-modules (but not necessarily as
$\H$-modules). Besides the forgetful functor these categories are also related by the invariance functor 
\begin{equation}\label{eq-inv}
\invf{\H}:\HpDcat{\H}{\bbd}\rightarrow\pDcat{\bbd}:\,X\mapsto \invf{\H}(X)=\mathrm{Hom}_{\H}(1,X)\,.
\end{equation}
We also denote by $\,\HfDcat{\H}{\bbd}\,$ and $\,\fDcat{\bbd}\,$ the respective subcategories of free 
$\bbd$-modules (for which the invariance functor is generally not well defined). The unit 1 of 
$\HpDcat{\H}{\bbd}$ is, as usual, given by $\bbd$ with $\H$-action defined by the counit.

In this article we will focus on Hopf algebras given as the Drinfeld double 
$\H=\dh$ of a Hopf algebra $H$ \cite{dr87,ka94} or occurring as tensor 
factors of such doubles. The Hopf algebra $H$ over a unital commutative
ring $\bbd$ will need to satisfy two technical conditions used in the TQFT
constructions, namely, that $H$ is {\em finite}, {\em Frobenius} and {\em double balanced}. 
We review the definition of these terms next. 

To begin with, an algebra $H$ is said to be {\em finite} over a commutative ring $\bbd$ if it is finitely
generated and projective as a $\bbd$-module. The Frobenius condition for an algebra $H$ is given 
as follows. 

\begin{definition}[\cite{par, kaso1}]\label{def-Frobenius}
 An algebra $H$ over a commutative ring $\bbd$ is {\em Frobenius} if $H$ is finite 
 over $\bbd$, and if there is an isomorphism
 of right $H$ modules
 \begin{equation}\label{eq-Frobenius}
   H_H\,\stackrel{\cong}\longrightarrow\, H^*_H\;.
 \end{equation} 
\end{definition}
The condition in (\ref{eq-Frobenius}) is equivalent to the existence of a {\em Frobenius homomorphism}
$\phi\in H^*$ characterized by the condition that $a\mapsto (\phi\lar a)=\phi(a\_):H\rightarrow H^*$ is a
bijection.

The Frobenius condition, when applied to finite Hopf algebras over a (unital) commutative ring $\bbd$,
implies the existence results of integrals such as that the space of left integrals 
$\lint{H}=\{\Lambda\,:\,x\Lambda=\epsilon(x)\Lambda\;\forall x\in H\}$, as a $\bbd$-module,  is a free, rank one, 
direct summand of $H$. See Section~\ref{sec-HAovDed} for details. 

The existence of integrals also implies the existence of moduli, which are distinguished group-like elements 
$\lgl\in H$ and $\alpha\in H^*$. These famously enter Radford's formula, $S^4=ad(\lgl)\circ ad^*(\alpha)$, for the
fourth order of  the anitpode.

As detailed in Definition~\ref{def-doubal}, we say that a finite Frobenius Hopf is double balanced if these
moduli admit well behaved group like  square roots $\lbl\in H$ and $\beta\in H^*$ implementing $S^2$. 
In this case the balancing element 
\begin{equation}\label{eq-bal-elem}
\theta=\beta(\lbl)\in \bbd^{\times}
\end{equation}
 turns out to be a root of unity.

\subsection{Statements of Main Results}
With the definitions above we can now state the first of our main results.

\medskip

\begin{thm}\label{thm1}
Suppose $H$ is a finite double balanced Frobenius Hopf algebra over a unital commutative ring $\bbd$ .
\begin{thmlist}
\item
Then there are TQFT functors $\Ph$ and $\Phc$ from cobordism categories to 
categories of $\bbd$-modules as indicated in the horizontal arrows in Diagram~(\ref{eq-thm-diag})
below. Moreover, this diagram of functors commutes.
\begin{equation}\label{eq-thm-diag}
\bfig
\square<1150,500>[\cobc`\HpDcat{\dh}{\bbd}`\cobn`\pDcat{\bbd};\mbox{$\Phc$}`
\mbox{$\fillf$}`\mbox{$\invf{\dh}$}`\mbox{$\Ph$}] 
\efig
\end{equation}
\item For  $\theta$ as defined in (\ref{eq-bal-elem}) or (\ref{eq-theta})  
and a closed 2-framed 3-manifold $M^*$, the value 
of $\Ph$ is related to the Hennings invariant $\ph$ for $\dh\,$ as follows
\begin{equation}
\ph(M)\;=\;\theta^{3\sigma(M^*)}\cdot \Ph(M^*)\,,
\end{equation}
where $M$ is the underlying 3-manifold of $M^*$ and $\sigma(M^*)$ is the signature corresponding to the
2-framing of $M^*$. 

\item If  $H$ is a free $\bbd$-module of finite rank then $\Phc$ restricts to a TQFT functor
\begin{equation}
\Phc : \cobc \to \HfDcat{\dh}{\bbd}\,.
\end{equation}
\end{thmlist}
\end{thm}

\medskip

Theorem~\ref{thm1} is based on the more general Theorem~\ref{thm-Htqft}, which
asserts  the existence of a TQFT functor, constructed explicitly in  Section~\ref{s2},
\begin{equation}\label{eq-tqft-genH}
\PHc : \cobc \to \HpDcat{\H}{\bbd}\quad(\mbox{\rm or }\HfDcat{\H}{\bbd})
\end{equation}
for a {\em \nicehopf} Hopf algebra, that is, a ribbon Hopf algebra $\H$
satisfying a list of technical assumptions described
in detail in Section~\ref{s2.2}. The upshot of Theorem~\ref{thm1} is
that all these technical assumptions are automatically fulfilled for the
quantum double of a double balanced Hopf algebra. 

The Frobenius condition is essentially equivalent to the existence of free and complemented 
integral spaces  and may be inferred in several situations from other criteria. 
The first, due to Pareigis, uses an
invariant of only the ground ring $\bbd$, namely, its Picard group $\mathrm{Pic}(\bbd)\,$.

\begin{lemma}[Corollary~1 in \cite{par}]\label{lm-picard}
 Suppose $H$ is finite Hopf algebra over a commutative ring $\bbd$.
 
 If $\mathrm{Pic}(\bbd)=0$ then $H$ is Frobenius.
\end{lemma}

If $\bbd$ is a ring of integers or a Dedekind domain the Picard group coincides with the 
ideal class group of $\bbd$. Thus the condition in Lemma~\ref{lm-picard} is equivalent to 
assuming that $\bbd$ is a PID. See, for example, \cite{lam,lang,langA}. In the case  of 
quantum algebras the ground ring is typically given by cyclotomic integers $\mathbb Z[\zeta]$
with $\zeta$ an $\ell$-th root of unity. In this situation Lemma~\ref{lm-picard} can be used to infer the 
Frobenius condition only for finite number of values of $\ell$. Specifically, as noted in  \cite{wa97},
$\mathbb Z[\zeta]$ is a PID if and only if   $\ell=$3, 4, 5, 7, 8, 9, 11, 12, 13, 15, 16, 17, 19, 20, 
21, 24, 25, 27, 28, 32, 33, 35, 36, 40, 44, 45, 48, 60, 84.

In other situations the explicit structure of the quantum algebra, rather than just its ground ring, 
needs to be put to use. A class of algebras $H$ for which the functors from Theorem~\ref{thm1}
are most closely related to the usual Reshetikhin-Turaev TQFTs is given by the Borel parts
of 
Lusztig's {\em small quantum groups} as described in   \cite{lu90b} and also \cite{lu90a}.

Specifically, for any simple Lie type $\mathfrak g$ and any integer $\ell\geq 2$ Lusztig defines
a small quantum Borel algebra $B_{\zeta}(\mathfrak g)$ over the domain $\bbd=\mathbb Z[\zeta]$
of cyclotomic integers where $\zeta$ is a primitive $\ell$-th root of unity.
In Section~\ref{s9} we further define an index two  extension $\smallBz$ of $B_{\zeta}(\mathfrak g)$
by a group like element.
For a given Lie algebra $\mathfrak{g}$ and integer $\ell$ we denote by $\rho$ half of the sum of all 
positive roots and
by $\ttau$ the highest element in the root lattice of $\mathfrak g$ 
that appears in the grading of $B_{\zeta}(\mathfrak g)$, see (\ref{eq-gradBz}).

\begin{thm}\label{thm-Bzg=Frob+Bal} For any simple Lie type $\mathfrak g$, any integer $\ell\geq 2$,
and primitive $\ell$-th root of unity $\zeta$ we have that 
the algebra $\smallBz$ as constructed in Section~\ref{ss-defSBZ} is a finite Frobenius Hopf algebra over 
 $\mathbb Z[\zeta]$.
 
 Moreover, $\smallBz$ is double balanced if   $(\rho,\ttau)\equiv 2 \sqw \mod \ell$ for some integer 
$\sqw\in\mathbb Z$. In this case the  balancing parameter is given by $\theta=\zeta^{\sqw}$.
\end{thm}

The values of the Cartan inner product $(\rho,\ttau)$ are provided in Table~\ref{tb-lie} and
various criteria for the condition above are listed in Corollary~\ref{cor-baldiv2}. We also
note that the existence of a PBW basis assures that $\smallBz$ is in fact a free module
over $\mathbb Z[\zeta]$. 
 
\begin{cor}\label{cor-SBZ} 
Let $\smallBz$ be as in Theorem~\ref{thm-Bzg=Frob+Bal} above. 

Then there exist associated Hennings-TQFT functors 
$\tqftVc{\mathcal D(\smallBz)}$ and $\tqftVn{\mathcal D(\smallBz)}$
in the sense of Theorem~\ref{thm1} with ground ring $\bbd=\mathbb Z[\zeta]$ and 
balancing parameter $\theta=\zeta^{\sqw}$. 
The target for $\tqftVc{\mathcal D(\smallBz)}$  is the free module category 
$\HfDcat{\mathcal D(\smallBz)}{\mathbb Z[\zeta]}$.
\end{cor}

In the second major goal of this paper is to describe the behavior of Hennings
TQFT functors with respect to {\em (strict) gauge transformations} of the coalgebra
and quasi-triangular structure of $\H$. The notion of gauge transformations was
introduced by Drinfeld in \cite{dr89} for quasi Hopf algebras, see also 
Section~XV.3 of \cite{ka94}. Here an invertible element $\F\in\H\otimes\H$ 
 is used to define
gauge transformed coproducts and R-matrices by $\Delta_{\F}(x)=\F\Delta(x)\F^{-1}$
and $\R_{\F}=\F_{21}\R\F^{-1}$. We denote by $\H_{\F}$ the ribbon
Hopf algebra with the so transformed  coproduct and $R$-matrix.

In this article we confine ourselves to strict quasi-triangular
Hopf algebras with trivial associators, which imposes an additional
cocycle condition given in (\ref{eq-Fcocyle}).

\begin{thm}\label{thm-GTtqft}
 Suppose $\H$ is a ribbon Hopf algebra fulfilling the prerequisites for the
Hennings TQFT construction and {\rm $\F\in\H\otimes\H$} fulfills the cocycle conditions
from (\ref{eq-Fcocyle}). 
Then there exists a natural isomorphism of TQFT functors
\rm{
\begin{equation}\label{eq-GTtqft}
\GTnatiso_{\F}\,:\,\PHc\;\natto\;\PHFc\;.
\end{equation}}
\end{thm}
We will give an explicit formula for $\GTnatiso_{\F}$ in (\ref{eq-gaugenattrsf}).

\medskip

Of particular interest is the case when $\,H\,$ is  the Borel subalgebra $\,\Bz\,$ of the 
quantum group $\,\Ufsl{\zeta}{2}\,$ where $\zeta$  is a root of unity of odd order. 
Theorem~\ref{thm1}  can now be used to infer integrality of the associated TQFT functor
and Hennings invariant as stated in Corollary~\ref{cor-sl2-TQFT} below.

The Hennings invariant for $\dBz$  is
closely related to the WRT invariant $\pwrt$, which is also constructed via the same surgery presentations
from categories obtained 
from $\,\Ufsl{\zeta}{2}\,$. In order to state the precise relation we introduce the 
following semi-classical invariants. The first is
\begin{equation}\label{eq-defhM}
\h(M)\;=\;\left\{
\begin{array}{cc}
|H_1(M,\mathbb Z)| & \mbox{ for } \beta_1(M)=0\\
0 & \mbox{ for } \beta_1(M)>0\rule{0mm}{6mm}\\
\end{array}\right.\quad.
\end{equation}
The second is the MOO invariant $\Zz(M)$ introduced in \cite{mh3} (see also Section~\ref{MOO}),
which is computed from only the linking matrix of a representing framed link.

\begin{thm}\label{thm2}
Let $M$ be a closed oriented 3-manifold and $\zeta$ be a root of unity of odd order $\ell$.
Let $\h(M)$ and $\Zz(M)$ be as above. Then the Hennings invariant 
$\pslz$ and the WRT $SO(3)$ invariant $\pwrt$ are related as follows:
\begin{equation}\label{e1}
\pslz(M) =  \h(M)\, \Zz(M)\, \pwrt(M)\,.
\end{equation}
\end{thm}

Theorem~\ref{thm2} is obtained from the more general Theorem~\ref{thm-TQFT-DB=UA} asserting an analogous
factorization of TQFTs, which, in turn, is based on the almost factorization of $\dBz$ into a version $\uz$ of
quantum-$\fsl2$ and and algebra $\A$ underlying the MOO invariant. The subtlety that prevents a strict 
factorization  is that $R$-matrix and coproduct of $\dBz$ differs from that of $\uz\otimes\A$ by a 
Drinfeld gauge twist as described above, see Proposition~\ref{lm-DBFfact}.

Thus Theorem~\ref{thm-GTtqft} needs to be invoked and 
enters the proofs of Theorem~\ref{thm2} and Theorem~\ref{thm-TQFT-DB=UA}. 
The gauge twist $\F$ and isomorphism $\GTnatiso_{\F}$ can be defined over $\mathbb Z[\zeta,\frac1\ell]$.

An immediate corollary and application of  Theorem~\ref{thm2} is the rederivation of the integrality result given 
by Le in \cite{le08}. Unlike the original proof no reference to the colored Jones Polynomial or representations of 
quantum algebras is made here.

\begin{cor}[\cite{le08}]\label{cor}
Suppose $\zeta$ is a root of unity of odd order $\ell>1$ and $M$ is a rational homology sphere. 
If $\h(M)$ and $\ell$ are coprime then $\pwrt(M)\in\zz$.
\end{cor}

\begin{proof}
Since $\pslz$ extends to an integral TQFT $\PBz$ as in Theorem~\ref{thm1} we find $\pslz(M)\in\zz$. 
Given $(\h(M),\ell)=1$ we have by Lemma~\ref{lm-MOO=Jacobi} that $\Zz(M)=\pm 1$
so that Equation (\ref{e1}) implies $\h(M)\,\pwrt(M)\in \zz$. As remarked in the end of
Section~1  of \cite{km}  we also have $\pwrt(M)\in\mathbb Z[\zeta,\frac1\ell]$ and hence $\ell^m\,\pwrt(M)\in \zz$
for some $m\in\mathbb N\,$. Using again $(\h(M),\ell^m)=1$ this immediately implies $\pwrt(M)\in \zz\,$.
\end{proof}

We note that the result in \cite{le08} has been generalized in \cite{bl}. In particular, 
the fact that $M$ is a rational homology sphere and that $(\h(M),\ell)=1$ are no longer 
required to ensure integrality. An argument similar to one used in the above proof 
can be found in \cite{cyz}.

\smallskip

\subsection{Further Remarks and Directions}\label{ss-open}

The approach via Hopf algebras to the construction of integral TQFTs developed in this 
article fundamentally differs from the other constructions based on WRT-type TQFTs. The 
existence of such TQFTs is established in \cite{g2} and \cite{cl05} indirectly from the 
known existence \cite{mur,mr,le93} of integral invariants for closed 3-manifolds using 
properties of Dedekind domains or the Kontsevich integral respectively.

Explicit integral bases for the $SO(3)$-theory are obtained \cite{gm07} carefully and 
intricately chosen combinations of skeins that span an invariant $\mathbb Z[\zeta]$-lattice. 

In our construction natural integral bases for $\mathbb Z[\zeta]$-lattices associated to 
bounded surfaces are readily obtained from a PBW basis of $\H$ and thus exist for all Lie 
types $\mathfrak g$. The more difficult question is integrality and the construction of bases 
in the restriction the TQFTs in Corollary~\ref{cor-SBZ} to closed surfaces.

Integrality will still be given in the finite number of cases where $\mathbb Z[\zeta]$ is 
a PID as in the list following Lemma~\ref{lm-picard}. Particularly, Theorem~\ref{thm1} 
implies that the $\mathbb Z[\zeta]$-modules associated to closed surfaces as  finitely 
generated projective modules and thus automatically free. 

For other roots of unity the general structure of finitely generated projective modules 
of Dedekind domains implies that the $\mathbb Z[\zeta]$-modules assigned to closed surfaces 
are of the form $M\cong F\oplus I$ where $F$ is a free $\mathbb Z[\zeta]$-module and $I$ an 
ideal of $\mathbb Z[\zeta]$. Moreover, $M$ depends up to isomorphism only on the rank of $F$
and the class $[I]$ in the ideal class group $C_{\zeta}$ of $\mathbb Z[\zeta]$
(see for example, \cite{langA} Ch.III, Exercises 11-13). In abstract terms the possible 
deviation from traditional integrality is thus given by only a single element in $C_{\zeta}$. 

In order to determine the exact structure and generators of the modules associated to a 
closed surface the subspace in $\dh^{\otimes g}$ invariant under the adjoint action of $\dh$
has to be computed over $\mathbb Z[\zeta]$, which is technically more involved and not 
considered in this paper.

\smallskip

\subsection{Overview and Organization of Article}

 Section~\ref{s2}  reviews definitions of cobordism categories, their tangle presentations,
and properties of Hopf algebras that will be required for the subsequent TQFT constructions. 
In Section~\ref{sec-tqft} we describe in more formal terms  the TQFT construction for surfaces 
with one boundary component that 
generalizes the Hennings invariant for closed 3-manifolds, introducing auxiliary categories
of $\H$-labeled planar curves that will play an important technical role throughout the paper.
Several functors involving these categories are constructed, assuming a list of properties of
the underlying Hopf algebra from which the TQFT is assembled.The extension of the TQFT functor 
to closed surfaces as well as the specialization to closed manifolds, resulting in the original
Hennings invariant are discussed in Section~\ref{sec-tqft-invar}. There the existence of 
TQFT functors for general \nicehopf Hopf algebras established throughout these sections in 
Corollary~\ref{cor-fact}, Corollary~\ref{cor-eqvTQFT}, and Theorem~\ref{thm-Htqft}.

Section~\ref{s2.4} begins with a review of Drinfeld's notion of gauge transformations of 
quasi-triangular Hopf algebras in the strict case where the transformation element has to 
fulfill a cocycle condition. We extract canonical elements associated to a gauge transformation 
that are used to provide explicit formulae for the change of the antipode, balancing and ribbon
elements as well as the integrals under the gauge transformations. This is applied in Section~\ref{sec-gaugeequiv}
to describe the effect of gauge transformations on the Hennings TQFT construction, 
resulting in the proof of Theorem~\ref{thm-GTtqft} above together with an explicit formula for
the natural isomorphism. 
 
In Section~\ref{s3} we develop the structure theory for quantum doubles of finite Frobenius Hopf 
algebras over commutative rings 
as it pertains to the TQFT construction given in Section~\ref{sec-tqft}. This involves the existence 
of integrals, moduli, double balancing structures, and ribbon elements, as well as their proper 
normalizations and projective phases. This analysis, summarized in Proposition~\ref{prop-intD},
entails the proof of  Theorem~\ref{thm1} given in Section~\ref{s3.0}.

The general theory is illustrated in greater detail in Section~\ref{s4} using example of 
the Borel subalgebra  $\,\Bz\,$ of 
quantum-$\fsl2$ at an odd root of unity. In Proposition~\ref{lm-DBFfact} and Theorem~\ref{thm-TQFT-DB=UA}
we provide the factorization up to gauge isomorphism of  $\dBz$ and the associated TQFT respectively. 
This implies then the proof of Theorem~\ref{thm2} given  in Section~\ref{s:relation}. 

Finally, in Section~\ref{s9} we review the definition and constructions of the Borel part of 
Lusztig's small quantum group mainly following
\cite{lu90b} with emphasis on its natural grading in the root lattice and introducing the above mentioned index 2
extension of the Cartan torus. Integrals and balancing structures are discussed in detail in Section~\ref{ss-balSBZ},
which is eventually used to prove  Theorem~\ref{thm-Bzg=Frob+Bal}.

\section{Topological and Algebraic Prerequisites}\label{s2}

In this section we provide some technical background required for the  construction of Hennings TQFTs .
Particularly, we describe the definition of cobordism categories, their tangle presentations, and various
relevant definitions and properties of Hopf algebras. In particular, we introduce the notion of \nicehopf
Hopf algebra which absorbs the requirements for the TQFT construction given in the subsequent section.
 
\subsection{The Cobordism Category}\label{ss-cobs}
We  summarize here the definition of the cobordism category $\cobc$ from \cite{kl}.  
The set of objects $\,Obj\left(\cobc\right)$ consists of compact, connected, oriented surfaces with one boundary component.
In addition, each surface $\,\Sigmab\in Obj\left(\cobc\right)$ comes equipped with a fixed orientation preserving
homeomorphism $\,\partial \Sigmab \listo S^1\,$. We also assume that for each integer $\,g\in\{0,1,2,\ldots\}\,$ the
set $\,Obj\left(\cobc\right)$ contains exactly one surface of genus $g\,$. 
For two surfaces $\Sigmab_1$ and $\Sigmab_2$ consider the closed surface obtained by sewing a cylinder
$\,C=S^1\times [0,1]$ between the two surfaces using the boundary homeomorphisms. We denote the resulting closed surface
as follows:
\begin{equation}\label{eq-sewsurf}
\Sigmab_1\cyconC\Sigmab_2\;:=\;-\Sigmab_1\bigsqcup_{\partial \Sigmab_1\sim S^1\times 0}
S^1\times [0,1]\bigsqcup_{\partial \Sigmab_2\sim S^1\times 1}\Sigmab_2 
\end{equation}
Here $(-\Sigmab_1)$ denotes the surface with opposite orientation. Thus, using the isomorphisms 
$\partial\Sigmab_i\cong S^1$ and  $\partial C\cong -S^1\sqcup S^1$ the combined surface $\Sigmab_1\cyconC\Sigmab_2$ 
will admit an orientation compatible with its pieces $-\Sigmab_1$ and $\Sigmab_2$.
See Figure~\ref{f2} for two equivalent depictions of $\Sigmab_1\cyconC\Sigmab_2$  in the case where $g_1=2$ and 
$g_2=3$.

\begin{figure}[ht] 
\includegraphics[height=23ex]{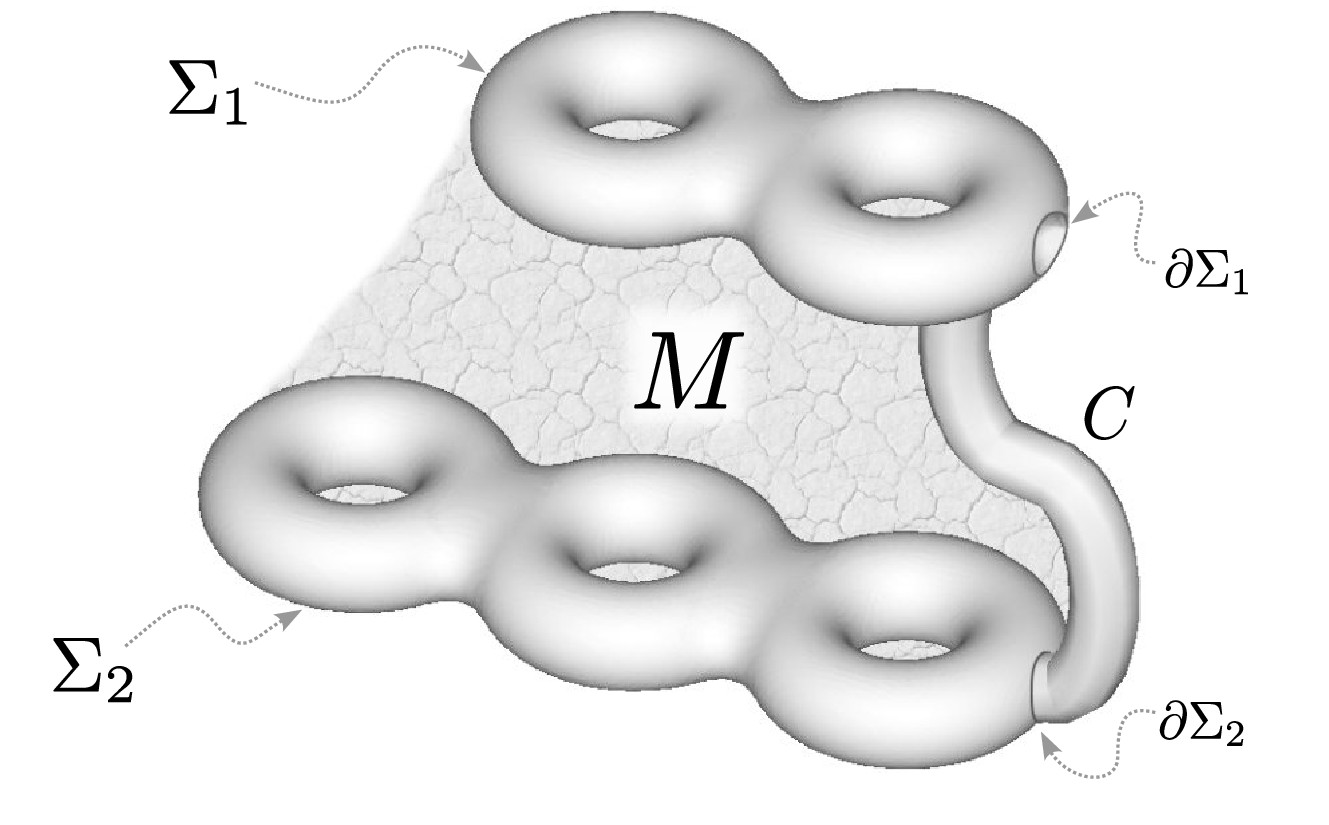}\qquad\raise 12 ex \hbox{\LARGE$\cong$}\qquad\includegraphics[height=25ex]{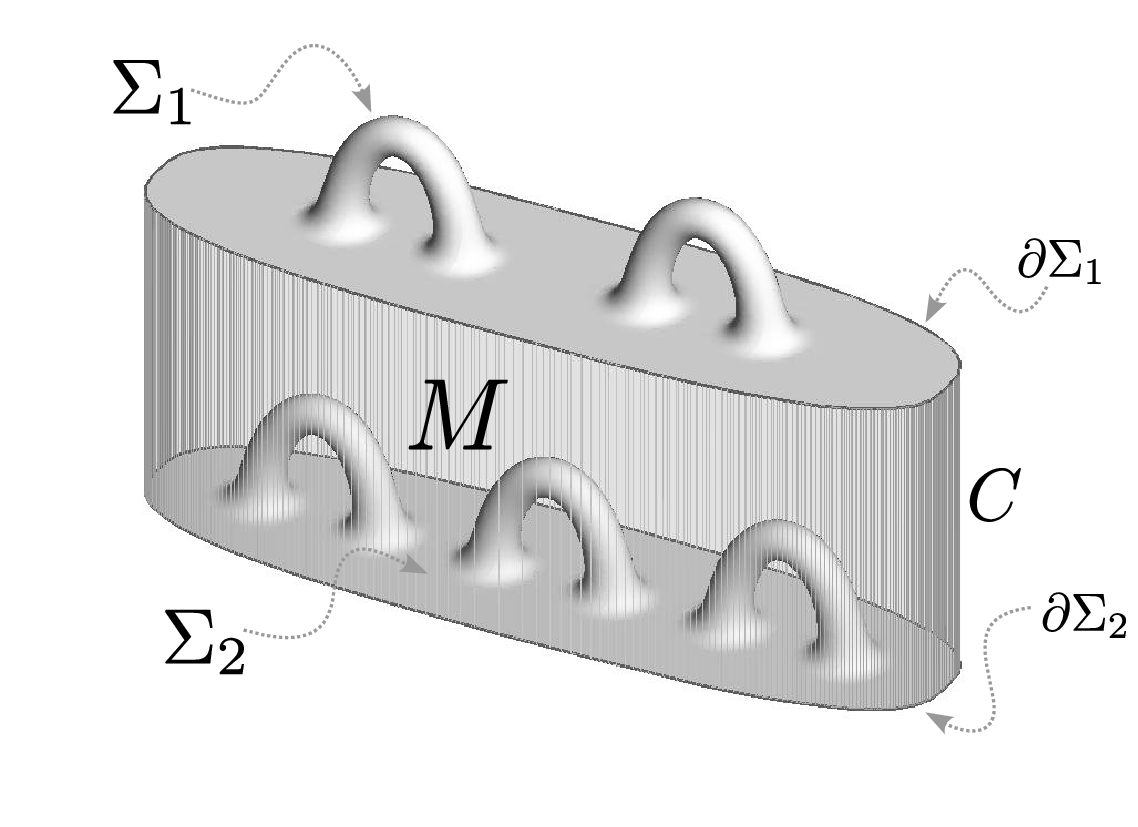}
\caption{A morphism $\,M:\Sigmab_1\to\Sigmab_2\,$ in $\cobc$}\label{f2}
\end{figure}

A cobordism is represented by a compact oriented 3-manifold $M$ with corners, together with a homeomorphism
 $\,\xi:\,\Sigmab_1\cyconC\Sigmab_2\isto \partial M\,$, mapping
the $S^1$-strata in (\ref{eq-sewsurf}) to the  1-dimensional corners of $M\,$. As usual,
we consider cobordisms $(M,\xi)$ and $(M',\xi')$ between the same surface to be equivalent if there is a
homeomorphism $\,\eta:M\isto M'\,$ such that $\,\eta\circ\xi=\xi'\,$.

A morphism in $\cobc$ is now an equivalence class of cobordisms $[M,\xi]$ together with a 2-framing of $M$, 
or, equivalently, the signature of a 4-manifold bounding a standard closure of $M\,$. 
For simplicity we will occasionally abuse notation and write $M$ for a morphism instead of $[M, \xi]$. 
Composition in $\cobc$ is defined by gluing over the respective surface pieces and rescaling of the cylindrical 
pieces. The composition is extended to include the signature information by gluing together representing 4-manifolds. 
See \cite{kl} for detailed definitions and constructions.

The objects of the category $\cobn$ are the same surfaces from $\,Obj(\cobc)\,$ but with a standard disc $D^2$  glued in
along each boundary $S^1\cong\partial\Sigmab$ yielding a closed surface $\Sigma=\Sigmab\sqcup D^2$. Cobordisms
are defined in exactly the same way as for $\cobc$. Moreover, given a cobordism $M:\Sigmab_1\to\Sigmab_2\,$ in
$\cobc$ we can obtain a cobordism $M^{o}:\Sigma_1\to\Sigma_2\,$ in $\cobn$ by gluing in a full cylinder $D^2\times[0,1]$ along
the boundary piece $C=S^1\times [0,1]\subset \partial M$. This filling is consistent with the standard closure from which
the signature extension is constructed (see \cite{kl}). Consequently, we have a well defined surjective, 
cylinder-filling functor:
\begin{equation}\label{eq-fillfunct}
\bfig
\morphism(0,0)/>>/<600,0>[{\fillf\,:\;\cobc\,}`{\,\cobn};]\;.
\efig
\end{equation}
 
\subsection{Algebraic Prerequisites for Hennings TQFTs}\label{s2.2}
Throughout this article $\H$ is a  Hopf algebra over a commutative ring $\bbd$ with invertible antipode. 
We denote by $\H^*={\rm Hom}_{\bbd}(\H,\bbd)$ 
the dual space of $\H$.
An element $\lambda\in \H^*$ is called a {\em right integral} on $\H$ if
\begin{equation}\label{eq-def-rinton}
\lambda f = f(1)\lambda, \quad \forall f\in \H^*.
\end{equation}
An element $\Lambda\in \H$ is called a {\em left (resp. right) integral} in $\H$  if
$$
x\Lambda = \epsilon(x)\Lambda \quad (\text{resp. } \Lambda x = \epsilon(x)\Lambda), \quad \forall x\in \H,
$$
where $\epsilon$ is the counit of $\H$. Note that the definition in (\ref{eq-def-rinton}) can be rewritten as
$$
\sum \lambda(x') x'' \,=\,(\lambda\otimes id)(\Delta(x)) \,=\,\lambda(x) 1 \quad \forall x\in \H\,,
$$
where we use Sweedler's notation for the coproduct
$$
\Delta(x) = \sum  x'\otimes x'', \quad \forall x\in \H.
$$
 
Let us also review a few standard notations for actions of  $\H$ and $\H^*$
on each other.  $\H^*$ carries a natural $\H$-bimodule structure given by the formulas and notation
\begin{equation}\label{e:co-arrows}
(f\lar a) (b) = f(ab) = (b\rar f)(a),
\end{equation}
for all $f\in \H^*$, $a, b \in \H$. Similarly, $\H$ can be endowed with an $\H^*$-bimodule structure via the formulas
\begin{equation}\label{e:arrows}
\textstyle
f\rar a = \sum f(a'')a' \quad \text{and} \quad a\lar f = \sum f(a')a'',
\end{equation}
A Hopf algebra $\H$ is said to be {\em unimodular} if there exists a non-zero left integral in $\H$ that is also a right integral in $\H$. 
Given a right integral $\lambda\in\H^*$ and a two-sided integral $\Lambda\in\H$ with $\lambda(\Lambda)=1$ define now maps
\begin{equation}\label{eq-beta-def}
\begin{split}
\beta\,&:\,\H\,\to\,\H^*\qquad \mbox{with} \qquad\beta(a) = a\rar\lambda\\
\overline\beta\,&:\,\H^*\,\to\,\H \qquad \mbox{with} \qquad\overline\beta(f) =\Lambda\lar f
\end{split}
\end{equation}
It is well known that these are isomorphisms. Particularly, the following relation is easily verified.
\begin{equation}\label{eq-bbS}
\overline\beta\circ \beta \,=\,S
\end{equation}

For $\H$ over a field it follows in \cite{rad1} that unimodularity implies (and in fact is equivalent) to the following identities.
\begin{equation}\label{eq-unimod}
 S(\Lambda)=\Lambda \qquad \quad \mbox{and}\qquad \quad \lambda(xy)=\lambda(S^2(y)x) \,.
\end{equation}

The relations in (\ref{eq-unimod}) are essentially the duals of Proposition 5 and Lemma 3 in \cite{rad1} respectively using that the comodulus $\alpha=\epsilon$
in the unimodular case. If $\H$ is finitely generated over some domain $\bbd$ these equations still hold in $\H\otimes K$ where $K$ is the 
field of factions of $\bbd$. Thus equations (\ref{eq-unimod}) also hold in $\H$ provided that  $\Lambda$ and $\lambda$ are also elements of $\H$ and $\H^*$
and provided that $\bbd$ is a domain.

Integrals in $\H$ and $\H^*$ exist and are unique up to scalars if $\bbd$ is a principal ideal domain and $\H$ is a free $\bbd$-module of finite rank
\cite{larson}. Conversely, Sweedler also showed in \cite{sw69} that for $\H$ over field this implies that $\H$ is finite dimensional. In \cite{lo04}
this implication is generalized, namely, that if $\bbd$ is an integral domain and $\H$ has an integral then $\H$ is finitely generated  
over $\bbd$. Thus we will always assume or imply that $\H$ is finitely generated over $\bbd$.

The next required ingredient for $\H$ is quasi-triangularity as defined by Drinfeld \cite{dr87}, which stipulates the 
existence of an R-matrix 
$\R\in\H^{\otimes 2}$ with functorial properties as follows.
\begin{equation}\label{eq-quasi-tri-def}
\begin{split}
 \Delta\otimes id(\R)=\R_{13}\R_{23} \qquad \qquad id\otimes \Delta (\R)=\R_{13}\R_{12}\\
\R\Delta(x)=\Delta'(x)\R\qquad\forall x\in \H\;.\rule{0mm}{7mm}
\end{split}
\end{equation}
Here $\Delta'$ denotes the opposite coproduct. For other
notations and more details see Section~VIII.2 in  \cite{ka94} or Section~10 in \cite{dr87}.

 A quasi-triangular Hopf algebra $\H$ is called {\em ribbon} if it also contains an element
$\r\in\H$ with the following properties.
\begin{equation}\label{eq-def-ribbon}
\r \mbox{ is central},\qquad S(\r)=\r, \qquad \mbox{and} \qquad \R_{21}\R \,=\,(\r\otimes \r)\Delta(\r^{-1})\;.
\end{equation}
See also equation (2.48) in \cite{ke94}. 
Here $\,\R_{21}=\sum_jf_j\otimes e_j\,$ denotes the element $\,\R=\sum_ie_i\otimes f_i\,$ with 
transposed tensor factors. Note that $\r^2$  is already determined by
the quasi-triangular structure alone via $m(id \otimes S)(\R_{21}\R)=\r^2$ so that the ribbon condition is really
about the existence of compatible square roots.
 
The ribbon structure for a quasi-triangular Hopf algebra may, alternatively, be described as a 
{\em balancing}.
To this end consider the canonical element 
\begin{equation}\label{eq-def-u}
u =\sum_i S(f_i)e_i\,.
\end{equation} 

It is well known \cite{dr89a} that this element
satisfies $uxu^{-1}=S^2(x)\,\forall x\in\H\,$ and that $uS(u)^{-1}$ is group like. A 
{\em balancing element} $\kappa$ 
is defined to be a group like element for which the following hold.
\begin{equation}\label{kappa}
\kappa^2 = uS(u)^{-1}= S(u)^{-1}u, \qquad
S^2(x) = \kappa x\kappa^{-1}, \quad\forall x \in \H.
\end{equation}
The existence of a balancing element is equivalent to the existence of a ribbon element $\r$ and the two are 
related by 
\begin{equation}\label{eq-rel-ribbon}
\r\,=\,u\kappa^{-1}\;.
\end{equation}

Yet another standard condition for the construction of TQFTs is {\em modularity}, which, in the original categorical
framework, has been 
 formulated as the invertibility of the so called ``S-matrix''. In the setting of Hennings TQFTs this may be rephrased
to require that the element
\begin{equation}\label{eq-def-Omega}
\M =  \R_{21}\R=\sum_{ij}f_je_i\otimes e_jf_i
\end{equation} 
 is (left) non-degenerate in the sense that the map
\begin{equation}\label{eq-def-Omega-map}
\overline\M:\;\H^*\to\H\,:\;l\,\mapsto\, l\otimes id(\M)
\end{equation}
is an injection. It would, in indeed, be more accurate to speak of
left modularity and consider as well the notion right modularity given by
injectivity of $l\,\mapsto\, id\otimes l(\M)$. These two conditions will turn out to be equivalent
to each other and to a number of other conditions which will be discussed in
greater detail in   Lemma~\ref{lm-modular-crit} below.

 At this point let us formalize 
the requirements on $\H$ needed for the construction of a Hennings TQFT from $\H\,$. 
\begin{definition}\label{def-cobmor}
 Let $\H$ be a quasi-triangular Hopf algebra, finitely generated over a commutative ground ring $\bbd$. 
We say $\H$ is {\em \nicehopf}  if it satisfies the following additional conditions: 
\begin{enumerate}
 \item $\H$ is ribbon or, equivalently, balanced with elements $\r$ and $\kappa$.
\item $\H$ is modular in the sense that (\ref{eq-def-Omega-map}) is an isomorphism.
\item There is a right integral $\lambda\in\H^*$ such that $\lambda(\r)\lambda(\r^{-1})=1$.
\item $\H$ admits a two-sided integral $\Lambda\in\H$ with $\lambda(\Lambda)=\lambda(S(\Lambda))=1\,$.
\end{enumerate}
\end{definition}

In Sections~\ref{sec-tqft} and \ref{sec-tqft-invar} 
we will show that any \nicehopf  Hopf algebra 
gives rise to associated TQFT functors
$ \PHc$ and $\PH$, as described in Theorem~\ref{thm1}, 
following the methods and constructions in \cite{kl} and \cite{kerler4}.

\subsection{Tangle Presentations}\label{sec-tang}

The first key ingredient in the TQFT construction is a surgery presentation of cobordisms
extending Kirby's calculus of links for closed 3-manifolds. Instead of links we consider a category
of admissible tangles $\TangNM$. Its set of objects is the set of 
non-negative integers $\mathbb Z^+=\mathbb N\cup\{0\}$.
For any pair $n,m \in \mathbb Z^+$ the set of morphisms ${\rm Hom}_{\sTangNM}(n,m)=\{[T]:n\to m: T=\mbox{admissible}\}$ 
is given by  equivalence classes of generic diagrams of  
admissible planar framed tangles in the strip 
$\mathbb R\times [0,1]$. Each tangle consists of $n$ {\em top} components, $m$ {\em bottom}
components, (each $\cong [0,1]$) and any number of {\em closed} components (each $\cong S^1$). The
end points of the $j$-th top component connects the integers $2j-1$ and $2j$ in the upper boundary
$\mathbb R\times \{1\}$ of the strip, and the $k$-th bottom component connects 
$2k-1$ and $2k$ in the lower boundary $\mathbb R\times \{0\}$. Figure~\ref{f3} depicts an example
of an admissible tangle $T:2\to 1$.
\begin{figure}[ht]
\centering
\psfrag{tm}{\hspace*{-12mm}$T:2\to 1\,\;=$} 
\psfrag{b0}{\footnotesize $1$}  
\psfrag{b1}{\footnotesize $2$}  
\psfrag{t0}{\footnotesize $1$}  
\psfrag{t1}{\footnotesize $2$}  
\psfrag{t2}{\footnotesize $3$}  
\psfrag{t3}{\footnotesize $4$}  
\includegraphics[height=17ex]{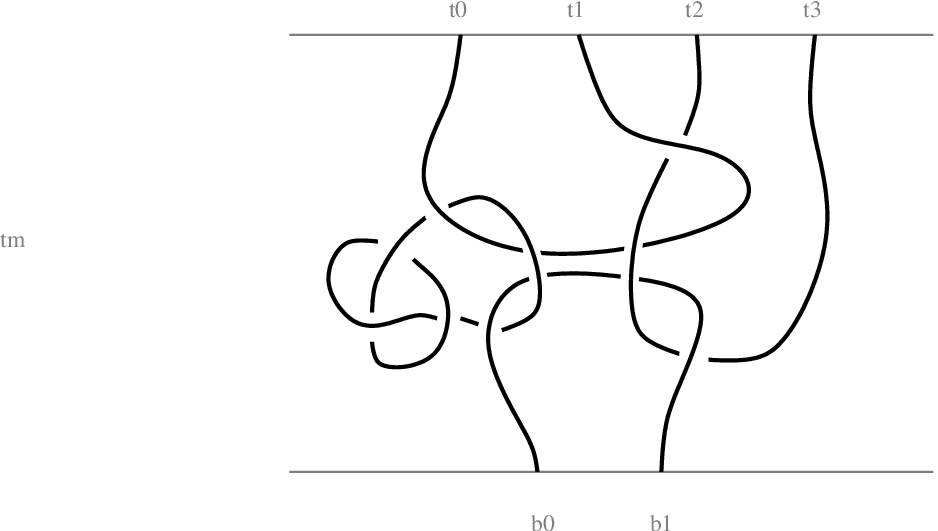}
\caption{A tangle presentation of a morphism in $\cobc$}\label{f3}
\end{figure}
The equivalences of tangle diagrams, defining the classes $[T]\in {\rm Hom}_{\sTangNM}(n,m)$, 
are given by isotopies of the diagrams in the plane and the usual Reidemeister moves for framed tangle diagrams.
Consequently, we can think of the sets of  morphisms
in  $\TangNM$ also as isotopy classes of admissible, framed tangles in $\mathbb  R^2\times [0,1]$.
As usual when drawing diagrams we assume that the framings of tangle components are give by
the blackboard framing.  Composition is defined
my stacking diagrams on top of each other.

From any admissible tangle in $\mathbb  R^2\times [0,1]$
 we obtain a cobordisms by adding 1-handles to the $\mathbb R^2\times\{1\}$ boundary 
of the 3-dimensional slice and continuing the top components through them. Moreover, we drill out holes
along the bottom components of the tangle. This turns the bottom boundary of the slice into an open genus $m$ 
surface and the top part into a genus $n$ surface, which are canonically compactified as such that their boundary 
is a circle. The remaining closed as well as closed-off top components thus constitute  a framed link $\mathcal L$.
The cobordism is finally obtained by performing surgery 
in the usual fashion along $\mathcal L$. 

This process thus describes a functor $\Surgf$ from the category of tangles to the category of cobordisms
defined thus far, which assigns to each integer
$n$ a surface $\Sigmab_n$ of genus $n$, and to each tangle $T:n\to m$ a cobordism 
$M=\Surgf(T):\Sigmab_n\to \Sigmab_m\,$. This functor is surjective on morphisms -- that is, any cobordism
can be obtained by this type of surgery. However, generalizing the ordinary Kirby calculus for closed 3-manifolds, 
different tangles will
yield the same cobordisms if and only if they are related by a sequence of the following moves  \cite{kl}:
\begin{enumerate}
\item[(T0)] Isotopies
\item[(T1)] Addition and removal of a pair of isolated unknots with framings +1 and -1.
\item[(T2)] $\mathcal O_2$-slides of any component over a closed component.
\item[(T3)] The $\sigma$-Move at any pair, see \cite{kl} or Section~\ref{sss-ModulT3} below.
\end{enumerate}
We denote by $\Tang$ the category obtained from $\TangNM$ by quotienting the morphism sets
by the additional equivalences (T0)-(T3), and by
$\Tmv:\TangNM\to\Tang$ the functor that assigns a tangle its equivalence class with respect
to these moves. We obtain the following commutative diagram.
\begin{equation}\label{eq-tanglefunct}
\bfig
\node a(0,250)[{\, \TangNM \,}]
\node b(700,-250)[{\, \Tang\; \; }]
\node c(2000,0)[{\,\cobc \,\;}]
\arrow|b|/->>/[a`b;\mbox{$\Tmv$}]
\arrow|b|/>->>/[b`c;\mbox{$\Surg$}]
\arrow|a|/>->>/[b`c;\mbox{$\cong$}]
\arrow/->/[a`c;\mbox{$\Surgf$}]
\efig
\end{equation}
It follows from the calculus described in Section~2.5.2.A of \cite{kl} that the functor $\Surg$ is indeed an
isomorphism of categories. Let us add a few more comments on the moves above:

The $\sigma$-Move (T3) is given by replacing the tangle configuration $\Pi$ depicted in Figure~\ref{fig-unimod}
by two parallel strands connecting respective points at the top and bottom line.

Instead of (T1) the move described in \cite{kl} actually involve both the addition of
Hopf links \texthopf{0}{0} and  \texthopf{0}{1} for which one component is 0-framed and the other can have
framing either 0 or 1. The diagrams  \texthopf{0}{1}  and \textunknot{-1}{}$\sqcup\!\!$ \textunknot{}{1} differ
only by a 2-handle slide.  
\begin{figure}[ht]
\centering
\psfrag{f1}{\small $\mathbf{0}$}
\psfrag{f2}{\small $\mathbf{0}$}
\psfrag{f3}{\small $\mathbf{1}$}
\psfrag{f4}{\small $\mathbf{1}$}
\psfrag{f5}{\small $\mathbf{1}$}
\psfrag{f6}{\small $\mathbf{1}$}
\psfrag{f7}{\small $\mathbf{1}$}
\psfrag{f8}{\small $\mathbf{0}$}
\psfrag{f9}{\small $\mathbf{1}$}
\psfrag{ar1}{\small $\mathcal{O}_2$}
\psfrag{ar2}{\small $\mathcal{O}_2$}
\includegraphics[height=4.7ex]{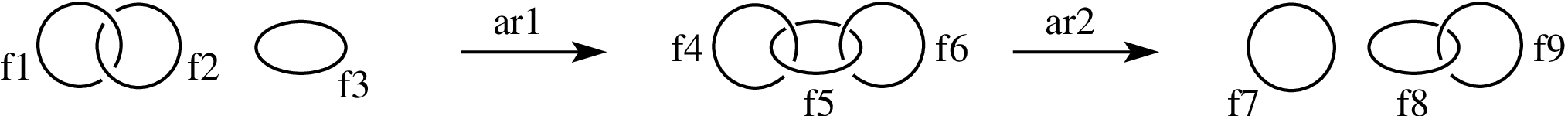}
\caption{Equivalence of Hopf link inclusions}\label{fig-hopf-calc}
\end{figure} 
The calculation in Figure~\ref{fig-hopf-calc}, which is essentially the same for the corollary following 
Proposition~2 in \cite{kirby}, shows that  \texthopf{0}{0} and  \texthopf{0}{1} are equivalent by 2-handle slides 
in the presence of 
an isolated \textunknot{}{1}. Thus, if addition or removal of the pair  \textunknot{1}{}$\sqcup$\textunknot{}{-1}
is considered an equivalence besides the 2-handle slides it follows from  \texthopf{0}{0} $\sim$ 
\texthopf{0}{0}$\,\sqcup\,$\textunknot{1}{}$\sqcup$\textunknot{}{-1}$\sim$ 
\texthopf{0}{1}$\,\sqcup\,$\textunknot{1}{}$\sqcup$\textunknot{}{-1}$\sim$ 
\textunknot{1}{}$\sqcup$\textunknot{}{-1}$\,\sqcup\,$\textunknot{1}{}$\sqcup$\textunknot{}{-1}$\sim\emptyset\,$ 
that also the second Hopf link move is already implied. Thus (T1) suffices as an equivalence.

\section{The Hennings TQFT Construction}\label{sec-tqft}

The Reshetikhin-Turaev construction of 3-manifold invariants and TQFTs \cite{rt} uses certain
semisimple braided tensor categories. These categories are typically obtained as subquotients
of representation categories of quantum groups or similar Hopf algebras.

Subsequently, Hennings formulated  in \cite{hennings} an invariant for closed 3-manifolds 
which entirely circumvents representation theory and, instead, computes the invariant directly 
from the elements of a quasi-triangular Hopf algebra. This simplified construction was further
developed by  Kauffman and Radford in  \cite{kr}, where also the role of the right integral is
clarified. 

The Hennings invariant has been extended to an algorithm for constructing TQFTs in \cite{kerler3}.
Furthermore, the more abstract constructions in \cite{kl} provide a unifying framework in which
both the WRT and Hennings TQFTs occur as  special cases.

In this section we will review and further develop the Hennings TQFT constructions described in 
\cite{kl, kerler3, kerler4}. Particularly, it will be useful to break the construction of the TQFT functor
$\PHc$ 
into that of two functors that are composed over an intermediate category of $\H$-labeled planar
curves $\Diag(\H)$ as indicated in the following commutative diagram.

\begin{equation}\label{eq-compfunct}
\bfig
\node a(1000,-400)[{\,\Tang \,\;}]
\node b(100,0)[{\, \TangNM \,}]
\node c(1100,0)[{\,\Diag(\H) \,}]
\node d(2300,0)[{\, \H\!\vDash \!\MM}]
\arrow|b|/>->>/[b`a;\mbox{$\Tmv$}]
\arrow/->/[b`c;\mbox{$\decor{\H}$}]
\arrow/->/[c`d;\mbox{$\eval_{\H}$}]
\arrow|b|/->/[a`d;\mbox{$\PHc \circ \Surg$} ]
\efig
\end{equation}
Here we consider both the projective and free case by setting 
\begin{equation}\label{eq-defbbM}
 \MM \;=\;\left\{
{\begin{array}{ll}
  \pDcat{\bbd} & \text{  whenever  } \H\in\pDcat{\bbd}\\
 \fDcat{\bbd} & \text{  whenever  } \H\in\fDcat{\bbd}\\
 \end{array}}
\right.\;.
\end{equation}

We begin with the description of $\Diag(\H)$  in Section~\ref{sss-CHLPC} followed by the
construction of the functors $\decor{\H}$ and $\eval_{\H}$ in Section~\ref{sss-FunctDH}.
In Sections~\ref{sss-InvarT0T2} and \ref{sss-ModulT3} we show that the composite
$\eval_{\H}\circ \decor{\H}$ is invariant under the moves (T0)-(T3), which implies 
that the functor factors into a functor $\Tang\to  \H\!\vDash \!\MM$ as indicated. 
Composing this with the inverse of the presentation isomorphism from (\ref{eq-tanglefunct})
 yields the desired TQFT functor.

\subsection{The Category of $\H$-labeled Planar Curves}\label{sss-CHLPC}
 We start with the definition
of the category $\,\Diag(\H) $ of  $\H$-labeled planar curves.
The objects of the category $\,\Diag(\H)\,$ are integers as for $\,\TangNM\,$.

The morphisms of  $\,\Diag(\H)\,$ are equivalence classes of {\em$\H$-labeled} 
planar curves with transverse double 
points and the same component and boundary structure as required for admissible 
tangles in $\,\TangNM\,$.

Formally, we can define an $\H$-labeled planar curve as a pair
$(D,a)$, where $D$ is a planar immersed curve in general position with $N$ ordered 
markings and $a$ is an element in  $\H^{\otimes N}$. If $a=\sum_{\nu}a_1^{\nu}\otimes\ldots\otimes a_N^{\nu}$
we also write $(D,a)$ as a formal sum, also with summation index $\nu$, of the same planar marked curve
with elements in $\H$ associated to each marking. The label at the $j$-th marking of the $\nu$-th diagram would 
thus be $a_j^{\nu}\,$.

We consider $D$ to be in general position if it locally looks like either 
non-horizontal smooth intervals, such intervals with a marking, crossings without horizontal pieces, 
or non-degenerate maxima or minima. Moreover, we may require all
markings, crossings,  and extrema to  occur at different heights. 

The labeled curves $(D,a)$ are subject to equivalence relations as depicted
in the Figure~\ref{f4-1}, Figure~\ref{f5-1}, and   Figure~\ref{f-rm} as well as their various reflections and general
planar isotopies that preserve the extrema of the diagram.  
Here $\kappa$ is the balancing element from (\ref{kappa}). The formal meaning of the equivalence 
on the right side  of Figure~\ref{f4-1}, for example, is that $(D,a)\sim(D',a')$ where $D'$ is obtained from $D$ 
by combining two markings, and $a'\in\H^{\otimes (N-1)}$ is obtained by 
applying the  multiplication map $m:\H^{\otimes 2}\to \H$ 
to the two tensor positions of $a\in\H^{\otimes N}$ (after suitable permutations of factors). The second equivalence
$(D,a)\sim(D',a')$ indicates that $D'$ is obtained from $D$ by moving the marking over an extremum and $a'\in\H^{\otimes N}$
 is found
from $a\in\H^{\otimes N}$ by applying $S$ to the tensor position in $\H^{\otimes N}$ corresponding to  this marking.
The remaining equivalences are described analogously. We denote the equivalence class for $(D,a)$ by
$[D,a]\in\mathrm{Hom}_{\sDiag(\H)}(n,m)\,$.

Composition of two planar curves, $D_1:n\to p$ and $D_2:p\to m$, with $N_1$ and $N_2$ markings respectively
is given by stacking the two diagram and shifting the  numbering of the marking in $D_2$ by $N_1$ denoted
by $D_2^{*+N_1}\,$. The composition of $\H$-labeled diagrams is thus given by 
$(D_2,a_2)\circ (D_1,a_1)=(D_2^{*+N_1}\circ D_1,a_1\otimes a_2)$, which is easily shown to be also well
defined on the equivalence classes defining the morphisms in  $\Diag(\H)\,$. The tensor product of diagrams
is defined analogously using juxtapositions.

\begin{figure}[ht]
\centering
\psfrag{1}{$x$}
\psfrag{2}{$y$}
\psfrag{3}{$yx$}
\psfrag{4}{$S(x)$}
\psfrag{5}{$x$}
\psfrag{6}{$=$}
\includegraphics[width=.87\textwidth]{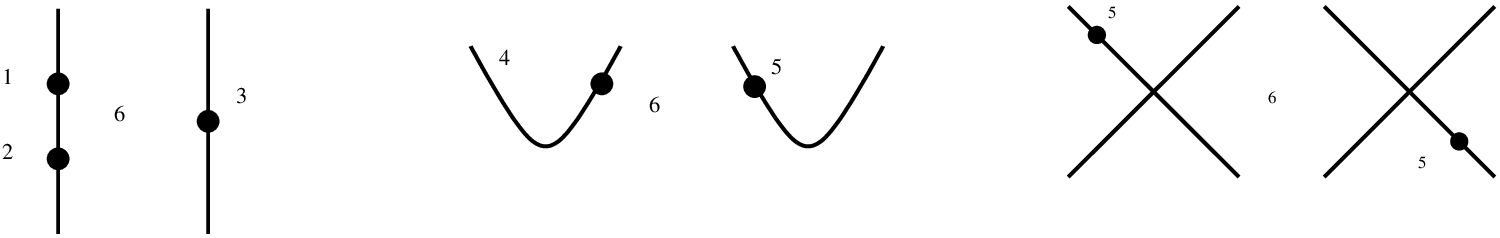}
\caption{Label sliding}\label{f4-1}
\end{figure}

\begin{figure}[ht]
\centering
\psfrag{1}{$\kappa$}
\psfrag{2}{$=$}
\includegraphics[width=.67\textwidth]{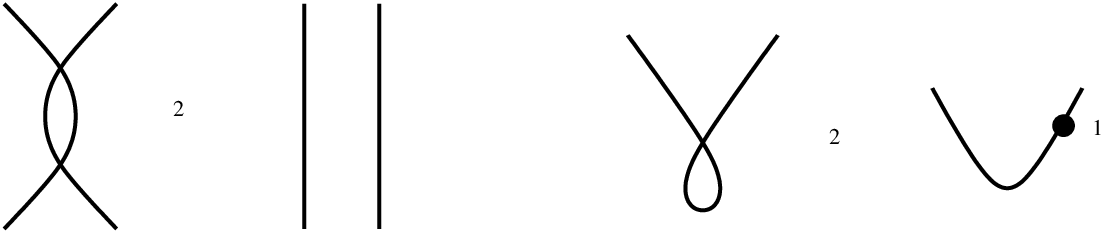}
\caption{Crossing removal}\label{f5-1}
\end{figure}

\begin{figure}[ht]
\centering
\psfrag{e}{{\small $=$}}
\includegraphics[width=.95\textwidth]{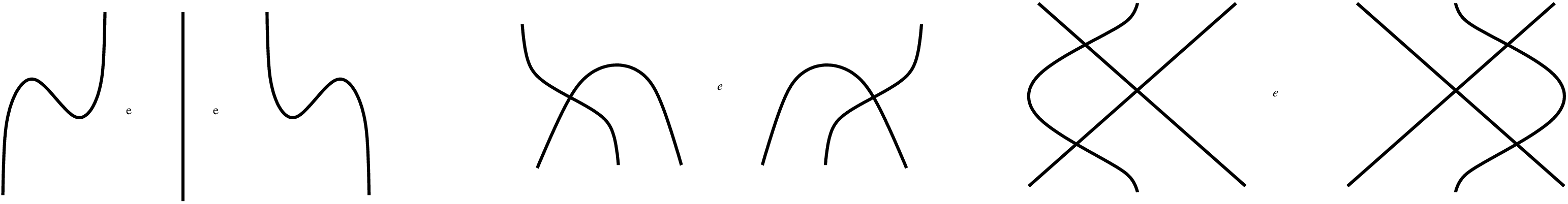}
\caption{Passing extrema and crossings}\label{f-rm}
\end{figure}

\subsection{Functors on $ \Diag(\H)$}\label{sss-FunctDH}

The functor
$\,\decor{\H}:\TangNM\to \Diag(\H)\,$ is identity on the set of objects. For a tangle $T$
representing a morphism in
$\TangNM$ the diagram for $\,\decor{\H}(T)\,$ is obtained by replacing positive or negative crossings
of $T$ by a diagram with double point with a (sum of) labeled markings as depicted in Figure~\ref{f1}.
The elements $e_i,f_i\in\H$ are the ones appearing in the expression for the universal $R$-matrix 
$\,\R = \sum_i e_i\otimes f_i \,\in\,\H\otimes\H\,$, and $S$ is the antipode of $\H$, suppressing
summation signs over  diagrams depending on a summation index $i\,$.

More formally, a tangle diagram $T$ is assigned to an $\H$-labeled diagram $(D,a)$, where $D$ is obtained 
from $T$ by flattening each crossing and two adding markings just above or below the crossing depending
on the orientation of the crossing in $T$. If $T$ has $p$ positive crossings (as defined by the left side
of Figure~\ref{f1}) and $q$ negative crossings then $a=(\R)^{\otimes p}\otimes (\R^{-1})^{\otimes q}\in\H^{\otimes 2(p+q)}$,
assuming an appropriate numbering of the markings and using $\R^{-1}=S\otimes id (\R)\,$.

\begin{figure}[ht]
\centering
\psfrag{1}{$e_i$}
\psfrag{2}{$f_i$}
\psfrag{3}{$S(e_i)$}
\psfrag{4}{$f_i$}
\psfrag{5}{\Large $\mapsto$}
\psfrag{6}{\Large $\mapsto$}
\includegraphics[width=.8\textwidth]{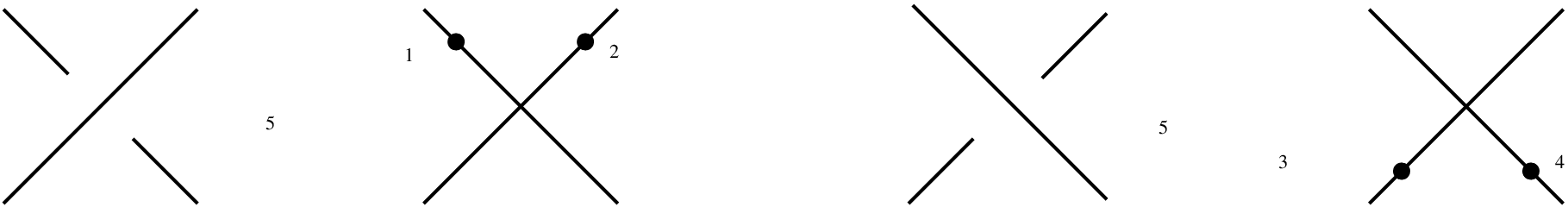}
\caption{Definition of $\,\decor{\H}\,$ at crossings}\label{f1}
\end{figure}

The fact that $\,\decor{\H}\,$ is well defined on isotopy classes of framed tangle diagrams
follows in standard fashion from the usual relations for $\,\R\,$ that can be derived from
the axioms in (\ref{eq-quasi-tri-def}), see \cite{hennings}.

The functor $\,\eval_{\H}:\,\Diag(\H)\to \H\!\vDash \!\MM\,$ is constructed as follows.
On the level of objects we set
\begin{equation}\label{eq-eval-obj}
 \eval_{\H}(n)\;=\;\H^{\otimes n} \;\;\in\,\H\!\vDash \!\MM\;.
\end{equation}

Here $\MM$ is as in (\ref{eq-defbbM}) since the tensor product of projective $\bbd$-modules
is again projective and similarly for free modules. 
The $\H$-action on $\H^{\otimes n}$ is given by the $n$-fold tensor product of
the adjoint action. This will be discussed in more detail in Section~\ref{sec-tqft-invar}.

In order to define the linear map $\,\eval_{\H}([D,a]):\,\H^{\otimes n}\to \H^{\otimes m}\,$
for an $\H$-labeled planar curve $\,(D,a):\,n\to m\,$ note that the moves in Figure~\ref{f4-1}
and Figure~\ref{f5-1} can be used to remove all intersections from the planar diagram
and move and combine the elements along a particular component to a single element 
on this component. We thus obtain an equivalent $\H$-labeled planar curve $(G,b)$,
where $G$ 
consists of $n$
top arcs, $m$ bottom arcs, and some number $p$ of isolated circles. The $p+m+n$
components are disjoint and carry one marking as depicted in Figure~\ref{f6-1}
for each type.
Thus, for an adequate numbering of these markings,
we have 
$\,b\in \H^{\otimes p}\otimes \H^{\otimes m}\otimes \H^{\otimes n}\,$.

Requiring that $\H$ is \nicehopf we may now assume the existence
of a right integral $\lambda$ with all of the listed properties listed
in  Definition~\ref{def-cobmor}. Considering the integral as  a map
$\lambda:\H\to\bbd$, the linear
morphism assigned by $\eval_{\H}$ is now given by the 
following formula:
\begin{equation}\label{eq-Emorph}
\begin{split} 
 \,\eval_{\H}([D,a])(x_1\otimes \ldots \otimes x_n)\,&=\,\eval_{\H}([G,b])(x_1\otimes \ldots \otimes x_n)\\
  \qquad =\;\lambda^{\otimes p}\otimes id^{\otimes m}\otimes \lambda^{\otimes n}&
\left(
(1^{\otimes p+m}\otimes S(x_1)\otimes\ldots\otimes S(x_n))b\right)\;.
\end{split}
\end{equation}

In the more combinatorial language indicated above this can be rephrased as an assignment of 
rank one  linear maps to pure tensors as follows. We write the element 
$b=\sum_{\nu}r_1^{\nu}\otimes\ldots \otimes r_p^{\nu}
\otimes
s_1^{\nu}\otimes\ldots \otimes s_m^{\nu}
\otimes
t_1^{\nu}\otimes\ldots \otimes t_n^{\nu}
\,$ so that the $\H$-labeled planar curve $(G,b)$ can be expressed by the 
union over indices $i,j,k$ and summation over the index $\nu$ of the pictures in Figure~\ref{f6-1}.

The diagrammatic rules to assign linear morphisms to building blocks of $G$ are now as indicated 
by the mappings in Figure~\ref{f6-1}. Specifically, 
we assign to a top arc with marking labeled by an element in $t\in \H$ 
the element $S^*(t \rar \lambda)\in \mathrm{Hom}_{\bbd}(\H,\bbd)=\H^*$, to a bottom arc 
the label $s\in \H$ considered as an element in $\mathrm{Hom}_{\bbd}(\bbd,\H)=\H$, 
and to each isolated circle with label $r\in \h$ the element in $\lambda(r)\in\bbd\,$.

\begin{figure}[ht]
\centering
\psfrag{1}{$\mapsto$}
\psfrag{2}{$\lambda(r_i^{\nu}),$}
\psfrag{3}{$1\mapsto s_j^{\nu},$}
\psfrag{4}{$x\mapsto \lambda(S(x)t_k^{\nu})$}
\psfrag{5}{\hspace*{-1mm}$r_i^{\nu}$}
\psfrag{6}{$s_j^{\nu}$}
\psfrag{7}{$t_k^{\nu}$}
\hspace{-14ex}\includegraphics[width=.75\textwidth]{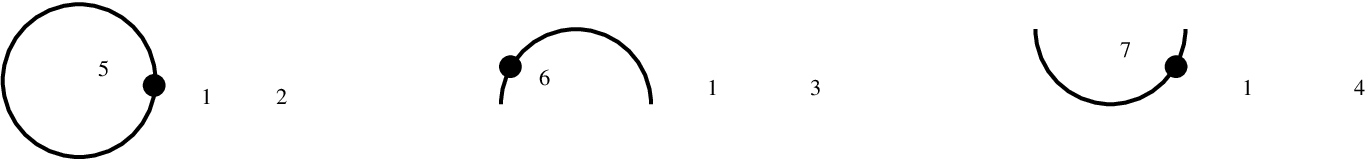}
\caption{Maps constructed from diagrams}\label{f6-1}
\end{figure}

We thus take an ordered tensor product along the indices $i,j,k$ of these factors 
we obtain, for fixed $\nu$,  an element in 
$\mathrm{Hom}_{\bbd}(\H,\bbd)^{\otimes n}\otimes_{\bbd} \bbd^{\otimes p}
\otimes_{\bbd}\mathrm{Hom}_{\bbd}(\bbd,\H)^{\otimes m}
\subseteq \mathrm{Hom}_{\bbd}(\H^{\otimes n},\H^{\otimes m})\,$. 
Summing over all indices $\nu$  yields the desired morphism
$\,\eval_{\H}([D,a])\,$ more formally described in (\ref{eq-Emorph}).
 
We conclude this section with a minor but often useful extension of the tangle
presentations of cobordisms described in Section~\ref{sec-tang} in which the 
notion of admissible tangles is generalized to include one more type. 
In addition to top, bottom, and closed components
we also allow pairs of strands that connect a pair of markings $\{2k-1,2k\}$ to another pair of
markings $\{2l-1,2l\}$. The surgery functor $\Surgf$ extends to such a configuration in the obvious manner. 
The 1-handle at the markings $\{2k-1,2k\}$ is attached first and an arc along the core of the handle
connecting the two markings is joined with the two strands. Along the resulting interval connecting the
bottom markings $\{2l-1,2l\}$ a hole is drilled out as for bottom components.

Although in this setting the  category $\TangNM$ contains additional morphisms, dividing by the
equivalences (T0)-(T3) above still yields the same category $\Tang\cong\cobc$.

The $\Diag(\H)$ are extended analogously by allowing pairs of strands with the same connectivities
as admissible $\H$-labeled planar curves, allowing a corresponding extension of the functor 
$\decor{\H}:\TangNM\to \Diag(\H)\,$.

The functor  $\,\eval_{\H}:\,\Diag(\H)\to \H\!\vDash \!\MM$ defined on the additional types of 
pairs of strands by 
 assigning the map $f:\H\to\H$, to be inserted in the respective factors, can be obtained
by applying substitutions of Figure~\ref{f1} and relations of Figures~\ref{f4-1} and \ref{f5-1}.
The value $f(x)$ is obtained by adding a cap as in the middle picture of Figure~\ref{f6-1} but
with $x$ instead of $s_j^{\nu}$ and collect elements  along the resulting bottom arc. The resulting
rules are as follows:

\begin{figure}[ht]
\centering
\psfrag{A}{$a$}
\psfrag{B}{$b$}
\psfrag{a}{\footnotesize $2k-1$}
\psfrag{b}{\footnotesize $2k$}
\psfrag{c}{\footnotesize $2l-1$}
\psfrag{d}{\footnotesize $2l$}
\psfrag{f}{{\Large $\mapsto$} $\quad x\mapsto axS(b)$}
\psfrag{u}{\footnotesize $2k-1$}
\psfrag{v}{\footnotesize $2k$}
\psfrag{z}{{\Large $\mapsto$} $\quad x\mapsto \kappa^{-1}S(x)$}
\hspace{-17ex}\includegraphics[width=.68\textwidth]{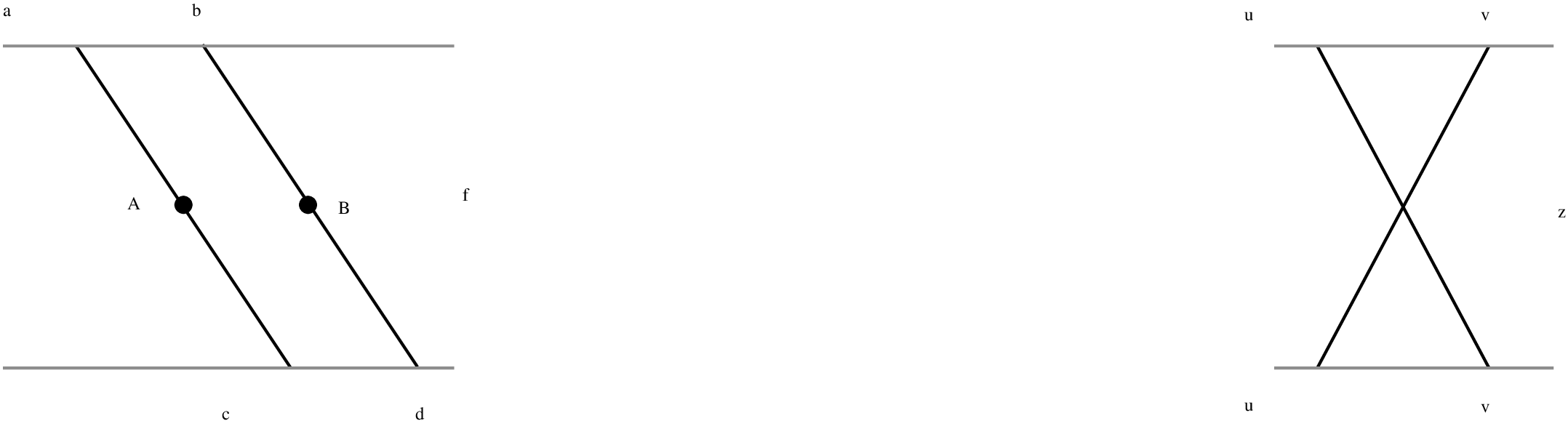}
\caption{Maps constructed from diagrams}\label{f-tt}
\end{figure}

\subsection{Invariance under T0-T2 Moves }\label{sss-InvarT0T2}

In order to establish the functors in (\ref{eq-compfunct}) one has to check that the composite 
$\,\eval_{\H}\circ\decor{\H}\,$ indeed factors through $\Tmv\,$. This means we need to verify
that two morphisms given by admissible framed tangles related by the moves (T0)-(T3) above are mapped to the same
linear map by $\,\eval_{\H}\circ\decor{\H}\,$ given that $\H$ is a \nicehopf Hopf algebra. 

Invariance under (T0) is implied by isotopy invariance on tangles in $\TangNM$. In order to 
see invariance under (T1) observe that $\,\eval_{\H}\circ\decor{\H} 
\left(\mbox{\small \textunknot{}{\scriptscriptstyle -1}}\right)=\lambda(\r)\,$. This follows
from the application of relations in $\Diag(\H)\,$ as depicted in  Figure~\ref{fr}. 
For the recombination of elements in the last step we compute
$$
\sum_iS(e_i)\kappa f_i=\sum_iS(e_i)S^2(f_i)\kappa =S(\sum_iS(f_i)e_i)\kappa=S(u)\kappa=u\kappa^{-1}=\r\;.
$$
where we make use of the  relations from (\ref{kappa}).

\begin{figure}[ht]
\centering
\psfrag{1}{$\scriptstyle -1$}
\psfrag{3}{$\scriptstyle S(e_i)\,$}
\psfrag{4}{$\,\scriptstyle f_i$}
\psfrag{5}{$\scriptstyle \kappa$}
\psfrag{6}{$\scriptstyle \r$}
\includegraphics[height=9ex]{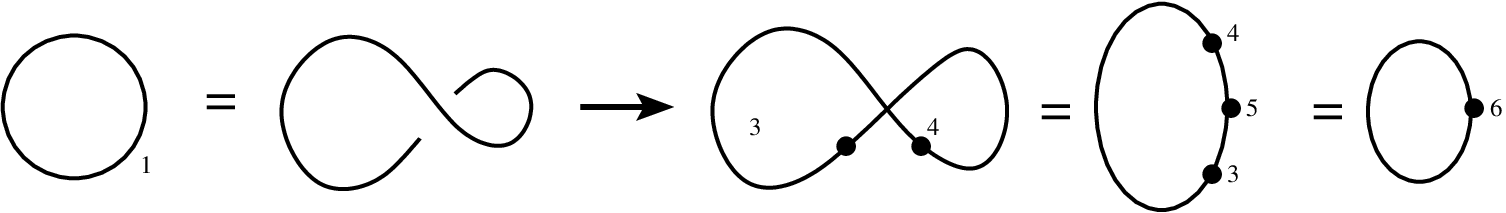}
\caption{The ribbon element}\label{fr}
\end{figure}
A similar calculation for \textunknot{}{\scriptscriptstyle +1} shows more generally that
\begin{equation}\label{eq-EZ-frshift}
 \,\eval_{\H}\circ\decor{\H} 
\left(\mbox{\textunknot{}{\scriptscriptstyle \mp 1}}\right)=\lambda(\r^{\pm 1})\;.
\end{equation}
Since all invariants factor over disjoint diagrams the addition of an isolated pair of
$\pm 1$-framed unknots \textunknot{}{\scriptscriptstyle -1}$\sqcup$\textunknot{}{\scriptscriptstyle +1}
results in an extra factor $\lambda(\r)\lambda(\r^{-1})$. Property (3) of Definition~\ref{def-cobmor}, 
however, implies that this factor is one so that the assignment of linear morphisms is indeed
invariant under the (T1) move.

Invariance under the (T2) move is a direct consequence of the defining properties of 
integrals in (\ref{eq-def-rinton}). See, for example, \cite{kr}.

\subsection{Modularity and T3 Move }\label{sss-ModulT3}

The (T3) or $\sigma$-Move (see also Section~3.1.3 of \cite{ke99}) is closely related to the modularity 
condition (2) required in Definition~\ref{def-cobmor} and
defined  via injectivity of the mapping in (\ref{eq-def-Omega-map}). We discuss next 
the equivalence of several 
conditions for modularity that arise in the context of the Hennings formalism. To this end we 
introduce in Figure~\ref{fig-unimod} special morphisms in $\TangNM$.

\begin{figure}[ht]
\centering
\psfrag{Q}{  $Q$}
\psfrag{is}{  $=$}
\psfrag{Pi}{  $\Pi = $} 
\psfrag{Sp}{  $\Sm_+=$}
\psfrag{Sn}{  $\Sm_-=$}
\psfrag{Gp}{  $\Gm_-=$}
\psfrag{Gn}{  $\Gm_+=$}
\includegraphics[height=29ex]{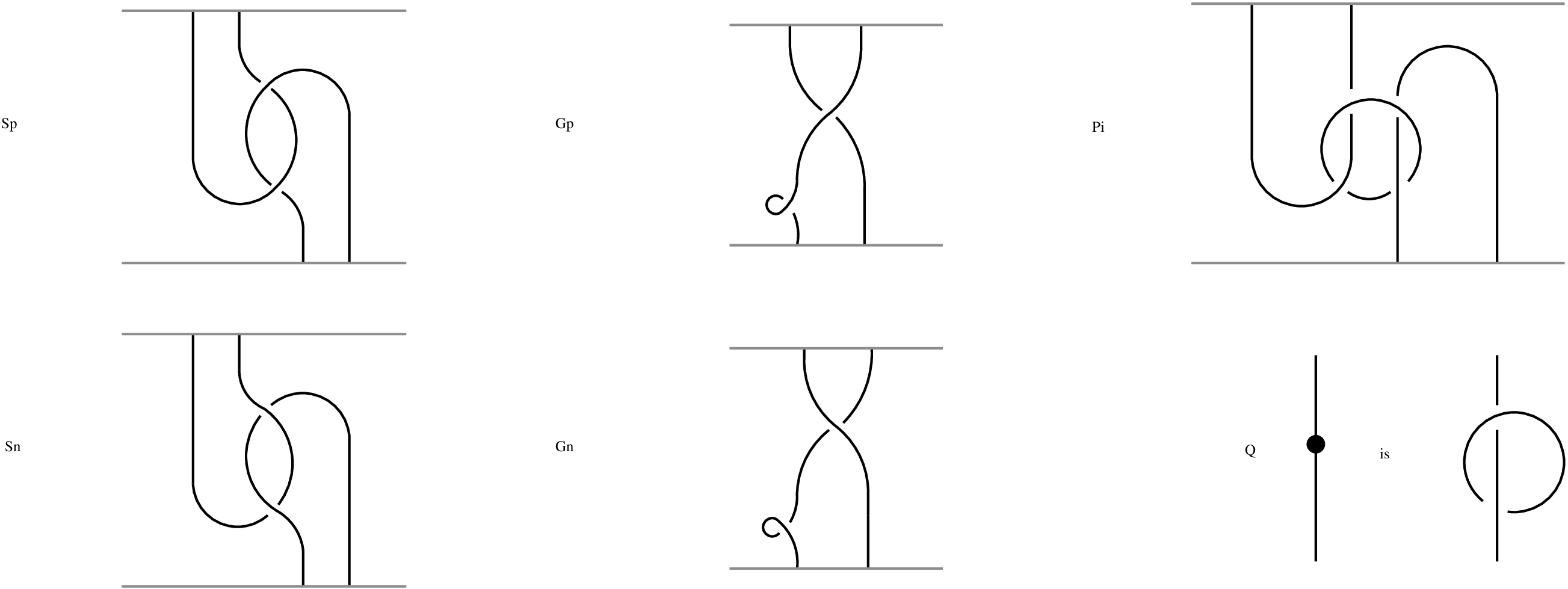}
\caption{Tangles related to modularity}\label{fig-unimod}
\end{figure}

The tangles $\Sm_{\pm}$ and $\Pi$ are admissible in the conventional sense and $\Gm_{\pm}$ are
admissible in the extended sense explained at the end of Section~\ref{sss-FunctDH} assuming
suitable labels at the top and bottom lines of the diagrams. 
The following relations are readily verified by composition of tangles and applications of isotopies. 
\begin{eqnarray}\label{eq-mod-tgl-rels}
\Gm_-&=&(\Gm_+)^{-1}\label{eq-tangle-modrel-Gm}\nonumber\\
\Gm_-\circ\Sm_+&=&\Sm_-\label{eq-tangle-modrel-Sm}\\
 \Pi&=&\Sm_+\circ\Sm_-=\Sm_-\circ\Sm_+\label{eq-tangle-modrel-Pi}\nonumber
\end{eqnarray} 
For convenience let us denote by $\overline T=\eval_{\H}\circ\decor{\H}(T)\,
\in {\rm Hom}_{\bbd}(\H^{\otimes n},\H^{\otimes m})\,$  the image of a tangle 
under the composite functor. It follows
readily from the rules laid out in Section~\ref{sss-FunctDH} that
$
\overline\Sm_+(x)=\sum_{ij}\lambda(S(x)f_je_i)e_jf_i\;,
$
which can be expressed as 
\begin{equation}\label{eq-SmMbS}
\overline\Sm_+=\overline\M\circ\beta^*\circ S\;.
\end{equation}

with $\overline\M$ as in (\ref{eq-def-Omega-map}) and $\beta$ as in (\ref{eq-beta-def}).

Note also that, by functoriality, the relations (\ref{eq-mod-tgl-rels}) also hold for the 
respective linear maps such as, for example, $\overline\Pi=\overline\Sm_+\circ\overline\Sm_-$.
Indeed this morphism can be computed also by  using the fact that doubling a strand is the same as taking a 
coproduct. That is, we have 
\begin{equation}\label{eq-Pi-form}
\overline\Pi(x)\,=\,\sum \lambda(S(x)Q')Q''\;,
\end{equation}
where $Q$ is the element  defined by the last tangle
in Figure~\ref{fig-unimod} and give explicitly by
\begin{equation}\label{eq-Q-def}
Q\,=\,id\otimes (\lambda\circ S)(\R_{21}\R)\,=\,\sum_{ij}e_if_j\cdot \lambda(S(f_ie_j))\,.
\end{equation}
Using the rules depicted in Figure~\ref{f-tt} we can also evaluate the morphisms for
the diagrams for $\Gm_{\pm}$ as
\begin{equation}
\overline\Gm_-(x)=u^{-1}\sum_i  S(f_i)S(x)e_i\qquad 
\mbox{and}\qquad \overline\Gm_+(x)=S(u)\sum_i  S^2(e_i)S(x)f_i\,.
\end{equation}
The squares of these are given by the following well known identity.
\begin{equation}
 {\overline\Gamma_{\pm}}^2(x)=ad(\r^{\pm 1})(x)\;.
\end{equation}

We are now in a position to prove the following lemma relating conditions of modularity.

\begin{lemma}\label{lm-modular-crit}
 Suppose $\H$ is a quasi-triangular ribbon Hopf algebra with a right integral 
$\lambda\in \H^*$ and a left integral $\Lambda\in \H$
such that $\lambda(\Lambda)=\lambda(S(\Lambda))=1$. Then following are equivalent.
\begin{center}
\begin{tabular}{r@{ }l@{\hspace*{26mm}}r@{ }l}
\rm{i)} & $\H$ is left modular. & \rm{ii)} & $\H$ is right modular.\\
\rm{iii)} & $\overline \Sm_+$ is injective. & \rm{iv)} & $\overline \Sm_-$ is injective.\\
\rm{v)} & $\overline \Sm_+$ is invertible. & \rm{vi)} & $\overline \Sm_-$ is invertible.\\
\rm{vii)} & $\overline \Pi$ is injective. & \rm{viii)} & $\overline \Pi$ is invertible.\\
\rm{ix)} & $\overline \Pi=id $. & \rm{x)} & $Q=\Lambda$.\\
\end{tabular}
\end{center}
\end{lemma}

\begin{proof}
i)$\Leftrightarrow$ii): Note that by (\ref{eq-def-ribbon}) and (\ref{eq-def-Omega}) we have 
$\M=(\r\otimes \r)\Delta(\r^{-1})$. If we denote $\rho_L(l)=l\otimes id(\Delta(\r^{-1}))$
the map in (\ref{eq-def-Omega-map}) can be written as $l\mapsto \r\cdot \rho_L(l\lar \r  )$.
Since $\r$ is invertible left modularity is thus equivalent to injectivity of $\rho_L$.
Similarly right modularity is equivalent to injectivity of 
$l\mapsto \rho_R(l)=id\otimes l(\Delta(\r^{-1}))\,$. Using $S\otimes S(\Delta(\r^{-1}))=
\Delta'(S(\r^{-1}))=\Delta'(\r^{-1})$ we obtain the identity $S\circ \rho_L\circ S^*=\rho_R$
so that injectivity of $\rho_L$ and $\rho_R$ are obviously equivalent.
The equivalence i)$\Leftrightarrow$iii) follows from (\ref{eq-SmMbS}) and invertibility
of $\beta$ given that $\lambda(S(\Lambda))=1$.

iii)$\Leftrightarrow$iv): Since $\,\decor{\H}\,$ is already invariant under isotopy we 
have by (\ref{eq-tangle-modrel-Gm}) that $\overline\Gm_{\pm}$ is invertible and by 
(\ref{eq-tangle-modrel-Sm}) that $\overline\Gm_-\circ\overline\Sm_+=\overline\Sm_-$
and $\overline\Gm_+\circ\overline\Sm_-=\overline\Sm_+\,$.
Given this equivalence we also have that both iii)$\Rightarrow$vii) and iv)$\Rightarrow$vii),
since by (\ref{eq-tangle-modrel-Pi}) we have   
\begin{equation}\label{eq-barPiSS}
\overline\Pi=\overline\Sm_+\circ\overline\Sm_-=\overline\Sm_-\circ\overline\Sm_+\;.
\end{equation}

vii)$\Leftrightarrow$viii)  and vii)$\Leftrightarrow$ix): These equivalences follow from
the tangle identity $\Pi^2=\Pi$, as computed, for example, in the proof of Lemma~4 of
\cite{kerler4} using only isotopy (T0), handle slides over circles (T2) and the cancellation of
\texthopf{0}{0} implied by (T1). We already verified that the functor 
$\eval_{\H}\circ\decor{\H}$ is already invariant also under these moves so that 
$\overline\Pi$ is also an idempotent, that is, $\overline\Pi^2=\overline\Pi$.
The claimed equivalences are now immediate.

The implications ix)$\Rightarrow$v) and ix)$\Rightarrow$vi) are obvious from (\ref{eq-barPiSS}),
and the implications v)$\Rightarrow$iii) and vi)$\Rightarrow$iv) are trivial. 
Hence i) through ix) are all equivalent and it remains to show ix)$\Leftrightarrow$x): 

If $\Lambda$ is a left integral with $\lambda(S(\Lambda))=1$
$S(\Lambda)$ is a right integral so that by (\ref{eq-Pi-form}) 
$\overline\Pi(\Lambda)= \sum \lambda(S(\Lambda)Q')Q''
=\lambda(S(\Lambda))Q=Q$ so that $\overline\Pi=id$ implies $\,Q=\Lambda\,$.
Conversely, if $Q=\Lambda$ is a left integral it follows by combining 
(\ref{eq-bbS}) and (\ref{eq-unimod})   that 
$ \,\overline\Pi(x)=\sum \lambda(S(x)\Lambda')\Lambda''\,=\,\lambda(\Lambda)\cdot x\,=\,x\,$.  
\end{proof}

The conclusion from Lemma~\ref{lm-modular-crit} most relevant to our construction is that modularity
implies $\overline\Pi=id$, and hence also that $\eval_{\H}\circ\decor{\H}$ is 
invariant under the required $\sigma$-move in (T3). Except for equivariance with
respect to the $\H$-action we have proved the following.

\begin{cor}\label{cor-fact} 
Give a \nicehopf Hopf algebra $\H$ and functors $\eval_{\H}$ and $\decor{\H}$ as above.

Then
the composite $\eval_{\H}\circ\decor{\H}: \TangNM \longrightarrow \MM\,$
is invariant under equivalences (T0) through (T3) and hence factors in a functor
$\,\PHc \circ \Surg:\Tang \longrightarrow \MM\,$ as indicated in
Diagram~\ref{eq-compfunct}.
\end{cor}
 
In this corollary the target of functors is still only the category of free or projective $\bbd$-modules
without consideration of the $\H$-action. The latter will be described in the following section, 
proving that the image of these functors lies indeed in 
$\H\!\vDash \!\MM\,$.

\section{Equivariance, Closed Surfaces, and Closed Manifolds}\label{sec-tqft-invar}

In this section we treat various aspects of equivariance of Hennings TQFTs with respect to 
the adjoint $\H$-action, starting with the outline of a diagrammatic proof of the same in 
Sections~\ref{sss-DiagrHact} and   \ref{sss-EquivarTQFT}. In Sections~\ref{sss-InvarTQFT}
and \ref{sss-ClosedSurfTQFT} we discuss the invariance functor and use it to construct
a TQFT functor for closed surfaces. Finally, we will define in Section~\ref{s2.3}
an invariant with values in the underlying ring $\bbd$ for closed 3-manifolds without framings 
via the usual framing renormalization and identify it with Hennings' original invariant.

\subsection{$\H$-Action and Equivariance in  $\Diag(\H)$}\label{sss-DiagrHact}

The fact that for any admissible tangle class 
$T:n\to m$ the morphism $\eval_{\H}\circ\decor{\H}(T)$ intertwines the 
adjoint actions on the tensor products of $\H$ can be inferred in an abstract manner
from the construction in \cite{kl}  that assigns a TQFT functor $\mathcal V_{\mathcal C}$ 
to a suitable category $\mathcal C\,$, and the fact that Hennings TQFT is equivalent to 
$\mathcal V_{\mathcal C}$  given that $\mathcal C=\H{\rm-mod}\,$. 

It is useful, however, to describe the $\H$-action and equivariance also concretely 
within the formalism of the combinatorial Hennings TQFT construction described thus far. 
For an element $x\in\H$ and an integer $g$ we assign a morphism in the extended version of
$\Diag(\H)$ by
\begin{equation}\label{eq-bbox}
 \Bbox^{(2g)}(x)=\left[{{|\mkern -7mu \bullet_1} \,\ldots \,{|\mkern -7mu \bullet_{2g}},\Delta^{(2g-1)}(x)}\right]\,:\;g\,\longrightarrow\,g\,,
\end{equation}
where the underlying labeled planar curve is given by $2g$ parallel strands, each with one marking in consecutive labels, and 
$\Delta^{(2g-1)}(x)=\sum x^{(1)}\otimes\ldots\otimes x^{(2g)}\in \H^{\otimes 2g}$ is the $2g$-fold coproduct of $x\,$.
Figure~\ref{fig-Hact} shows the respective picture of this morphism is the  previously described diagrammatic language.

\begin{figure}[ht]
\centering
\psfrag{X}{$\Delta^{(2g-1)}(x)$}
\psfrag{e}{$=$}
\psfrag{a}{\footnotesize $x^{(1)}$}
\psfrag{b}{\footnotesize $x^{(2)}$}
\psfrag{c}{\footnotesize $x^{(2g-1)}$}
\psfrag{d}{\footnotesize $x^{(2g)}$}
\hspace{-1ex}\includegraphics[width=.94\textwidth]{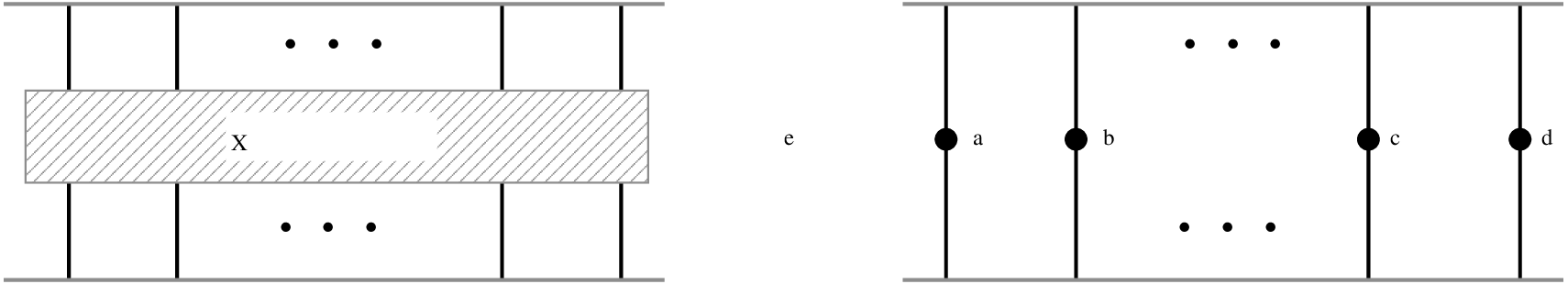}
\caption{Diagrammatic $\H$-action}\label{fig-Hact}
\end{figure}

Equivariance in the diagrammatic context is now given by the following statement. 

\begin{lemma} \label{lm-dec-intw}
Suppose $T:g\to h$ is a tangle in $\TangNM$ and $x\in\H$ then
\begin{equation}
  \decor{\H}(T)\Bbox^{(2g)}(x)=\Bbox^{(2h)}(x) \decor{\H}(T)\;.
\end{equation}
\end{lemma}

\begin{proof}

Note that the categories $\Diag(\H)$ and $\TangNM$ can be extended to categories 
$\Diag^0(\H)$ and $\TangNM^0$ respectively in which all of the connectivity requirements
for admissible tangles or diagrams are dropped. The larger morphism sets of $\TangNM^0$
have a simple set of generators given by over- and under-crossings as well as maxima and 
minima.

Using the same assignments given in Section~\ref{sss-FunctDH} we obtain a functor
$\decor{\H}^0:\,\TangNM^0\to \Diag^0(\H)$ such that $\decor{\H}$ is given precisely as the 
restriction of $\decor{\H}^0$ to $\TangNM$.  Moreover, the construction in (\ref{eq-bbox})
 obviously generalizes to morphisms $\Bbox^{(N)}(x)$ in $\Diag^0(\H)$
with an odd number of parallel strands. For a generating tangle $T$ of $\TangNM^0$
with  $N$ top  and $M$ bottom end points we can now verify the intertwining 
relation
\begin{equation}\label{eq-intwdec}
 \decor{\H}^0(T)\Bbox^{(N)}(x)=\Bbox^{(M)}(x)\decor{\H}^0(T)\;.
\end{equation}
Particularly, $T$ is a crossing (with possibly more parallel strands to the right and left)
this relation follows directly from coassociativity
as well as the second line in (\ref{eq-quasi-tri-def}) after multiplying elements 
along strands. If $T$ is a maximum or minimum (\ref{eq-intwdec}) is a 
consequence of the antipode axiom and the relations in Figure~\ref{f4-1}.

Since every tangle is a composite of these generators
(\ref{eq-intwdec}) holds for all morphisms  in $\TangNM^0$. The claim now
follows by restriction to $\TangNM$.
\end{proof}

An immediate consequence of (\ref{eq-intwdec})  is that string links $N=M=1$ yield 
elements in the center of $\H$, see for example \cite{re90,KRS98}.

\subsection{$\H$-Equivariance of the Hennings TQFT}\label{sss-EquivarTQFT}

As described at the end of Section~\ref{sss-FunctDH} the fiber functor $\eval_{\H}$ 
is defined on the extended version of $\Diag(\H)$ and thus also on the morphisms
$\Bbox^{(2g)}(x)$ for any $x\in\H\,$. The action of $\Bbox^{(2g)}(x)$
is readily found from the rule in Figure~\ref{f-tt} to be just the
$g$-fold tensor of the adjoint action of $\H$ on $\H^{\otimes g}$. 
\begin{equation}\label{eq-adver}
 \begin{split}
 \eval_{\H}(\Bbox^{(2g)}(x))(b_1\otimes \ldots \otimes b_g)
& =
\sum x^{(1)}b_1S(x^{(2)})\otimes\ldots\otimes x^{(2g-1)}b_gS(x^{(2g)})\\
& =
\sum ad(x^{(1)})(b_1)\otimes\ldots\otimes ad(x^{(g)})(b_g)\\
& =
ad(x)^{\otimes g}(b_1\otimes \ldots \otimes b_g)\;.
\end{split}
\end{equation}

Applying the functor  $\eval_{\H}$ to Lemma~\ref{lm-dec-intw} and using (\ref{eq-adver})
we thus arrive 
for any tangle $T:n\to m$ at the
desired $\H$-equivariance relation 
\begin{equation}\label{eq-adverT}
\overline T \cdot ad(x)^{\otimes n} = ad(x)^{\otimes m} \cdot \overline T\;,
\end{equation} 
where we denote with $\overline T= \eval_{\H}\circ\decor{\H}(T)\,$.
This also implies that the target category of the functor $ \eval_{\H}\circ\decor{\H}$ 
is the one of $\H$-modules. Since $\eval_{\H}\circ\decor{\H}=\PHc \circ \Surg\circ\Tmv$
and $\Surg\circ\Tmv$ is a full functor we can now refine the statement of  Corollary~\ref{cor-fact}
as follows. 
\begin{cor}\label{cor-eqvTQFT}
Suppose $\H$ is a \nicehopf Hopf algebra. With notations as above we have a TQFT functor
\begin{equation}\label{eq-funct2Hmod}
\PHc  :\,\cobc \longrightarrow \H\!\vDash \!\MM\,.
\end{equation}
\end{cor}

\subsection{The $\tau$-Move for  Closed Surfaces and the Invariance Functor}\label{sss-InvarTQFT}
Let us now discuss the two main ingredients for extending the TQFT from Corollary~\ref{cor-eqvTQFT} above
to the category $\cobn$ of cobordisms between closed surfaces.   Like $\cobc$ 
the category $\cobn$ is  represented by a tangle category $\Tangn$ which is obtained from
the category $\TangNM$ by dividing out relations (T0)-(T3) as for $\Tang$ but in addition also the
\begin{itemize}
 \item[(T4)] $\tau$-Move 
\end{itemize}
defined in \cite{ke99} (Figure~3.32, see also TS4 of  Section~2.3.3 of \cite{kl}).
This move allows passing a piece of a strand $\mathcal L$ through
the entire collection of $2g$ parallel strands near the top (or bottom) end of
the diagram as depicted in Figure~\ref{fig-tauMove}.
\begin{figure}[ht]
\centering
\psfrag{L}{$\mathcal L$}
\psfrag{e}{\LARGE $\leftrightarrow$}
\psfrag{1}{\footnotesize $1$}
\psfrag{2}{\footnotesize $2$}
\psfrag{a}{\footnotesize $2g-1$}
\psfrag{b}{\footnotesize $2g$}
\hspace{-1ex}\includegraphics[width=.8\textwidth]{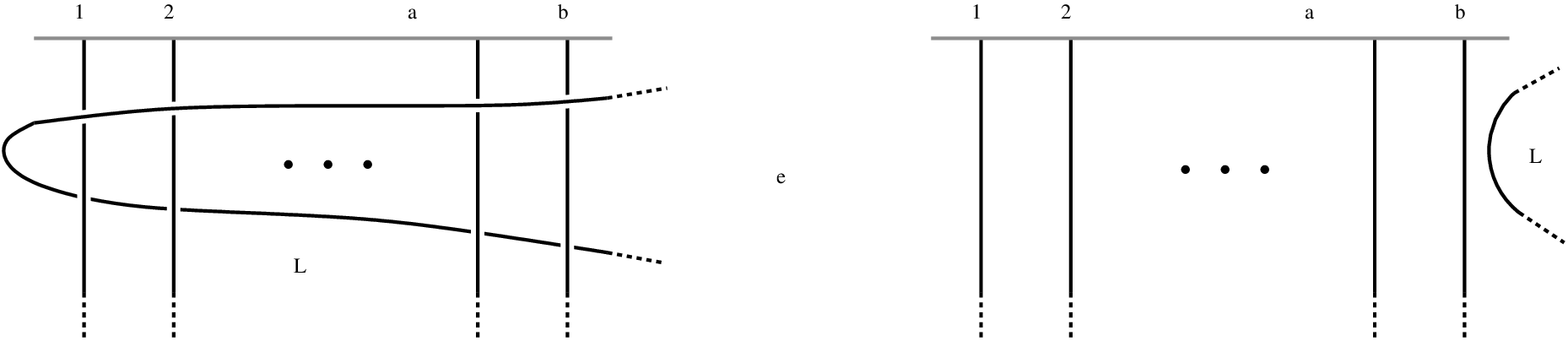}
\caption{$\tau$-Move}\label{fig-tauMove}
\end{figure}

In order to describe the difference between these diagrams in the Hennings formalism
we apply $\decor{\H}$ to the left tangle piece with $4g$ crossings. Denoting
by $\M=\R'\R=\sum m_j\otimes n_j$ and repeatedly applying the relations in (\ref{eq-quasi-tri-def})
we find that the   resulting $\H$-labeled planar curve piece in  $\Diag(\H)$ 
can be written as 
\begin{equation}\label{eq-tau-box}
\left({{|\mkern -7mu \bullet_1} \,\ldots \,{|\mkern -7mu \bullet_{2g}} \,{|\mkern -7mu \bullet_{2g+1}},
\Delta^{(2g-1)}\otimes id(\M)}\right)\,.
\end{equation}
where the last strand with marking numbered $(2g+1)$ belongs to $\mathcal L$. That is, in diagrammatic
terms we have a marking on $\mathcal L$ labeled by elements $n_j$ 
and elements  $\Delta^{(2g-1)}(m_j)$
distributed over the $2g$ parallel top strands with markings as in Figure~\ref{fig-Hact}. 
We use this observation now as follows. 

\begin{lemma}\label{lm-tau-invar}
Suppose $v\in\H^{\otimes g}$ such that $ad(x)^{\otimes g}(v)=\epsilon(x)v$ for all $x\in\H$.
 Then  $T\mapsto \eval_{\H}\circ\decor{\H}(T).v$ is invariant under the $\tau$-Move given in Figure~\ref{fig-tauMove}.
\end{lemma}

\begin{proof} 
Let $T_l$ and $T_r$ be the tangles depicted on the left and right side of Figure~\ref{fig-tauMove},
and let $(D,a)$ be the $\H$-labeled planar curve obtained from $T_r$ by application of the rules  in
Section~\ref{sss-FunctDH}. Moreover, let $D'$ be the diagram $D$ but with one additional marking 
on the $\mathcal L$ piece and numbering such that $[D,a]=[D',a\otimes 1]\,$. 

The diagram assigned to $T_l$ consists of $4g$ additional crossing and thus $8g$ additional markings,
of which each vertical strand  carries two and the piece $\mathcal L$ has $4g$. Multiplying the elements
along each of these strand pieces using equivalences in Figure~\ref{f4-1} leaves $(2g+1)$ markings
(one for each piece).  As argued above the iteration of relations in (\ref{eq-quasi-tri-def}) shows
that the tensor assigned  is as in (\ref{eq-tau-box}). Writing this labeled curve piece as a sum
$\sum_j \left({{|\mkern -7mu \bullet_1} \,\ldots \,{|\mkern -7mu \bullet_{2g}} \,{|\mkern -7mu \bullet_{2g+1}},
\Delta^{(2g-1)}(m_j)\otimes n_j}\right)
$,
we can further move all of the first $2g$ markings along the parallel strands upwards so that they are 
separated by horizontal line from the rest of the diagram. For fixed $j$ the resulting diagram can consequently be
expressed as a composition of the piece $\left({{|\mkern -7mu \bullet_1} \,\ldots \,{|\mkern -7mu \bullet_{2g}} ,
\Delta^{(2g-1)}(m_j)}\right)$ with the remainder of the diagram, which, in turn, represents the morphism 
from (\ref{eq-tau-box}) for $x=m_j$. This implies the relation
\begin{equation}\label{eq-dec-tau-dec}
\decor{\H}(T_l)\,=\, \sum_j [D',a\otimes n_j]\circ\Bbox^{(2g)}(m_j)\;.
\end{equation}
For an element $v\in\H^{\otimes g}$ as assumed in the lemma we thus compute
\begin{equation}\label{eq-tau-invar-calc}
 \begin{split}
\eval_{\H}( \decor{\H}(T_l)).v\,& =\, \sum_j \eval_{\H}\left([D',a\otimes n_j]\right)\eval_{\H}(\Bbox^{(2g)}(m_j)).v\\
&\byeq{eq-adver}\, \sum_j \eval_{\H}\left([D',a\otimes n_j]\right) .(ad(m_j)^{\otimes g}(v))\\
& = \, \sum_j \eval_{\H}\left([D',a\otimes n_j]\right).(\varepsilon(m_j)v)\\
& = \, \eval_{\H}\left([D',a\otimes 1]\right).v\, = \, \eval_{\H}\left([D,a]\right).v\, = \, \eval_{\H}( \decor{\H}(T_r)).v\\
\end{split}
\end{equation}
where we also used $\epsilon\otimes id(\M)=1$. 
\end{proof}

Thus the restrictions of the  morphisms $\eval_{\H}\circ\decor{\H}(T)$ 
 to the invariance subspace $\inv_{\H}(\H^{\otimes g})$, given by all 
elements $v\in\H^{\otimes g}$ with $ad(x)^{\otimes g}(v)=\epsilon(x)v$ 
implements the $\tau$-Move. Regarding integrality this restriction will 
not necessarily preserve freeness of $\bbd$-modules but only 
projectiveness as described next.
 
\begin{lemma}\label{lm-defInv}
Suppose $\bbd$ is Noetherian. Then the restriction to invariance subspaces 
$\inv_{\H}(M)
=\{v\in M: x.v=\epsilon(x)v \;\forall x\in\H\}$ of modules
yields a functor
 $$\inv_{\H}\,:\,\HpDcat{\H}{\bbd}\to \pDcat{\bbd}\qquad M\mapsto \inv_{\H}(M) \;.$$
\end{lemma}
\begin{proof}
The definition of $\inv_{\H}$ extends to morphisms since invariance spaces are
mapped to each other by equivariance. Moreover, as
a finitely generated module of a Noetherian ring $\bbd$ also $M$ is  Noetherian
so that the subspace $\inv_{\H}(M)$ is Noetherian as well and thus finitely generated.
Finally, the fact that finitely generated modules are projective precisely 
when they are torsion-free (e.g., \parag 9 of \cite{lang}) implies that $\inv_{\H}(M)$
has to be projective as well since submodules of torsion free modules are obviously 
again torsion free.
 \end{proof}

\subsection{Hennings TQFT  for Closed Surfaces}\label{sss-ClosedSurfTQFT}
In order to organize and summarize the TQFT constructions thus far let us consider 
 the quotient functor  $\mathscr F_{\tau}:\Tang\to\Tangn$, which quotients 
morphism sets by the additional $\tau$-Move, as well as the filling functor
$\fillf:\cobc\to\cobn$ that was introduced in (\ref{eq-fillfunct}). 
Observe that an isotopy class of a tangle 
in $D_+^2\times[0,1]$ modulo the $\tau$-Move is the same as its isotopy class
in $S^2\times [0,1]$ after a complementary $D_-^2\times[0,1]$ is glued in.
The move really describes the transverse passage of a piece of a strand $\mathscr L$ 
through $D_-^2\times[0,1]$.

Clearly we obtain the same cobordism in $\cobn$ whether we first surger along the tangle 
and then glue in $D_-^2\times[0,1]$ or whether we first add 
$D_-^2\times[0,1]$ and then carry out the surgery. Formally this can be expressed by the
following relation of functors
\begin{equation}\label{eq-filltau}
 \Surg^{\emptyset}\circ\mathscr F_{\tau}\,=\,\fillf\circ\Surg\,,
\end{equation}
where $\Surg^{\emptyset}:\,\Tangn\to\cobn\,$ is the isomorphism
functor providing a presentation of the category of cobordisms between closed surfaces analogous to 
$\Surg$ defined in (\ref{eq-tanglefunct}).
We summarize the existence and corresponding relations of TQFT functors in the following theorem.

\begin{thm}\label{thm-Htqft} Suppose $\H$ is a \nicehopf
Hopf algebra over a Noetherian ring $\bbd$, which is,
as a $\bbd$-module, projective and finitely generated over $\bbd$. 

Then there is TQFT functor {\rm $\PH:\cobn\to\pDcat{\bbd}$} and 
a commutative diagram of functors as given in (\ref{eq-Htqft}),
where $\PHc$ is as in Corollary~\ref{cor-eqvTQFT}, $\fillf$ as in  (\ref{eq-fillfunct}),
and $\invf{\H}$ as in Lemma~\ref{lm-defInv}.
\begin{equation}\label{eq-Htqft}
\bfig
\square<1150,500>[\cobc`\HpDcat{\H}{\bbd}`\cobn`\pDcat{\bbd};\mbox{$\PHc$}`
\mbox{$\fillf$}`\mbox{$\invf{\H}$}`\mbox{$\PH$}] 
\efig
\end{equation}
\end{thm}
\begin{proof} Combining Lemma~\ref{lm-tau-invar} and Corollary~\ref{cor-fact} we find
that  the composite $\inv_{\H}\circ \eval_{\H}\circ\decor{\H}$ is invariant under moves
(T0)-(T4). Consequently it factors through the functor
$\Tmvn:\TangNM\to\Tangn$ quotienting out all these moves, which, in turn, is given by 
 $\Tmvn=\mathscr F_{\tau}\circ\Tmv$ with $\mathscr F_{\tau}$ as in (\ref{eq-filltau}).

Analogous to the construction of $\PHc$ and using the representation isomorphism 
$\Surg^{\emptyset}$ we thus infer a TQFT functor $\PH$ with  
$\inv_{\H}\circ \eval_{\H}\circ\decor{\H}=\PH \circ \Surg^{\emptyset}\circ\Tmvn$. 
Substituting the above expression for $\Tmvn$, and using $\eval_{\H}\circ\decor{\H}=\PHc \circ \Surg\circ\Tmv$, 
which is implied by Corollary~\ref{cor-fact} ,
we obtain $\inv_{\H}\circ\PHc \circ \Surg\circ\Tmv=\PH \circ \Surg^{\emptyset}\circ
\mathscr F_{\tau}\circ\Tmv$. Since $\Tmv$ is surjective we obtain
$\inv_{\H}\circ\PHc \circ \Surg =\PH \circ \Surg^{\emptyset}\circ
\mathscr F_{\tau}=\PH \circ \fillf\circ\Surg$ using also (\ref{eq-filltau}).
Given that $\Surg$ is an isomorphism we conclude (\ref{eq-Htqft}).
\end{proof}

The problem remains how to obtain \nicehopf Hopf algebras that fulfill the prerequisites of
Theorem~\ref{thm-Htqft} and Definition~\ref{def-cobmor}. We will show in Section~\ref{s3} that a large
naturally constructed  family of  \nicehopf Hopf algebras is given by quantum doubles of double balanced Hopf algebras.

\subsection{The Hennings invariant for closed 3-manifolds}\label{s2.3}

The TQFT functors can be used to associate an invariant for a closed, connected, compact,
oriented  3-manifold  $M$ as follows. 
We remove a 3-ball from $M$ with boundary $S^2\cong D^2\cyconC D^2$.
This defines a relative cobordism  $M^*:D^2\to D^2$. It is easy to see that different choices of  3-balls 
and parametrizations of the $S^2$ boundary will yield equivalent cobordisms due to connectedness of $M$ and
the fact that the oriented mapping class group of $S^2$ is trivial. If we endow $M^*$ with a 2-framing we
obtain a morphism in $\cobc$ on which we can evaluate $\PHc$. By construction we have 
$\PHc(D^2)=\H^{\otimes 0}\cong\bbd$  so that $\mathrm{End}(\PHc(D^2))=\{x\cdot id: \,x\in \bbd \}$ is the free rank one
$\bbd$-module with canonical generator given by the identity on $\PHc(D^2)$. We thus have
\begin{equation}\label{eq-Henn-idnorm}
\PHc(M^*)=\widetilde\pH(M^*)\cdot id
\end{equation}
for a unique $\widetilde\pH(M^*)\in\bbd$. Although this element of $\bbd$ does not depend on the choice of the 
removed ball and the parametrization of the bounding sphere, it still depends on the choice of the 2-framing
on $M^*$ or signature of
a 4-manifold . It does so, however, in an easily described manner.

The signature of a bounding 4-manifold is given by the signature $\sigma(L)$ of the linking matrix of a framed link $L$ representing 
$M^*$ or $M$ by surgery. It follows from elementary matrix algebra that $\sigma(L)$ is invariant under the moves (T0), (T1), and (T2) 
from Section~\ref{sec-tang} and, hence, depends only on $M^*$ so that we may write $\sigma(L)=\sigma(M^*)\,$.

For a given framed link $L$ in $S^3$ let us denote $L_{\downarrow}=L\sqcup \mbox{\textunknot{}{-1}}$ the link with an isolated, additional
(-1)-framed unknot. The respectively represented morphisms $M^*$ and $M^*_{\downarrow}$ differ only by a shift in framing
so that we have $M=M_{\downarrow}$  for
the underlying (unframed) 3-manifolds. By definition, $\sigma(L_{\downarrow})=\sigma(L)-1$ so that also 
$\sigma(M^*_{\downarrow})=\sigma(M^*)-1$.

We also observe from the constructions in Section~\ref{sec-tqft} that the evaluation of $\eval_{\H}\circ \decor{\H}$
on disjoint unions of links is multiplicative. The value of this functor on \textunknot{}{-1} is given in (\ref{eq-EZ-frshift})
as $\lambda(\r)\,$. It follows that $\PHc(M^*_{\downarrow})=\lambda(\r)\PHc(M^*)$ and the analogous relation for 
$\widetilde\pH$. Combining these relations we find that $\pH(M^*_{\downarrow})=\pH(M^*)$ where
\begin{equation}\label{eq-HennInv}
\pH(M) \;=\;\lambda(\r)^{\sigma(M^*)}\cdot \widetilde\pH(M^*)\,.
\end{equation}
Since $\pH(M)$ is invariant under shifts of framings or signatures it does indeed 
depend only on the underlying topological 3-manifold. We will refer to $\pH(M)$ as the Hennings invariant 
of $M$ for the algebra $\H$. Note that in this definition the normalization is such that $\pH(S^3)=1$,
which also allows $\pH(S^1\times S^2)=\lambda(1)$ to be non-invertible \cite{kerler2}.

\section{Gauge Transformations of Hopf algebras}\label{s2.4}

The concept of gauge transformations of Hopf algebra structures naturally arose in Drinfeld's discussion of
{\em quasi Hopf algebras} in \cite{dr89}. In this section we investigate the effect of such transformations
on the resulting TQFTs. Our study is motivated by the proof of
Theorem~\ref{thm2} in Section~\ref{s4} where we use the fact that the double of the quantum Borel algebra of
$\fsl{2}$ is equivalent to the tensor product of quantum $\fsl{2}$ and  a cyclic group algebra -- but only up to
a non-trivial gauge transformation.

For this purpose we will focus on  the situation 
of gauge transformations between strictly associative Hopf algebras, which will require an additional cocycle
condition but, in return, avoids the use of associators. A generalization of the following
discussion to quasi Hopf algebras with non-trivial associators $\Phi\in\H^{\otimes 3}$ is expected to 
generalize to respective TQFT constructions for quasi Hopf algebra as in \cite{jg13}. 

\subsection{Cocycle Condition, Special Elements, and Relations of Gauge Transformations }\label{ss-Fprop}

For a Hopf algebra $\H$ we say that an element $\F\in\H\otimes \H$ is a {\em cocycle} if it satisfies the conditions
\begin{equation}\label{eq-Fcocyle}
\begin{split}
  (1\otimes \F)(id \otimes \Delta)(\F)\;=&\; (\F\otimes 1)(\Delta \otimes id)(\F)\\
\epsilon\otimes id(\F)\;=&\;1\;=\;id\otimes\epsilon(\F)\rule{0mm}{6mm}
\end{split}
\end{equation}
We say that a cocycle $\F\in\H^{\otimes 2}$ is a {\em gauge transformation} of Hopf algebras (as opposed to
quasi Hopf algebras) if it also has an inverse $\F^{-1}\in\H^{\otimes 2}$. Note that one particular class of gauge 
transformations is given by {\em coboundaries} which are defined
for any invertible element $\,c\in\H\,$ by $\,\F_c=(c^{-1}\otimes c^{-1})\Delta(c)\,$ .

Before discussing the modifications imposed on the Hopf algebra structure by a gauge transformation let us
explore several useful relations implied by the cocycle condition. We use the following notations for the cocycle tensor expressions.
$$
\textstyle
\F=\sum_iA_i\otimes B_i \qquad \mbox{  and } \qquad \F^{-1}=\sum_iC_i\otimes D_i\;. 
$$
We also denote the maps $\Jr=(id\otimes m)(\Delta\otimes S)$ and $\Jl=(m\otimes id)(S\otimes \Delta)$ from $\H\otimes\H$ to itself,
where $m:\H^{\otimes 2}\to\H$ is the multiplication. 
For convenience we record the explicit action of $\Jr$, $\Jl$, and their inverses on elements of $\H^{\otimes 2}$. 
\begin{equation}\label{eq-Jrl}
\begin{array}{rlrl}
 \Jr(a\otimes b)&=\sum a'\otimes a''S(b) \qquad &\qquad  \Jl(a\otimes b)&=\sum S(a)b'\otimes b''\\
\Jr^{-1}(a\otimes b)&=\sum a'\otimes S^{-1}(b)a'' & \Jl^{-1}(a\otimes b)&=\sum b'S^{-1}(a)\otimes b''\rule{0mm}{6mm}\\
\end{array}
\end{equation}
The two transformations are in fact conjugate by 
\begin{equation}\label{eq-KJ}
\Jl\,=\,\JJ \circ \Jr\circ \JJ^{-1}\qquad \mbox{with} \qquad \JJ:=(S\otimes S)\circ\tau
\end{equation}
where $\tau(a\otimes b)=b\otimes a\,$. The following additional relations are readily verified as well.
\begin{equation}\label{eq-KJrels}
 \Jl\circ\Jr\;=\;\JJ\circ \Jl^{-1} \qquad \mbox{and} \qquad \JJ^2\circ \Jr=\Jr\circ \JJ^2 \quad \mbox{with}\;\;\JJ^2=S^2\otimes S^2\;.
\end{equation}

We also associate the following special elements associated to a gauge transformation.
\begin{equation}\label{eq-def-xt}
\textstyle
\mbox{and } \qquad
\begin{array}{rclcl}
\xt_{\F}& =& m(id\otimes S)(\F) &=&\sum_i A_iS(B_i)\\
\overline{\xt}_{\F} & =& m(S\otimes id)(\F^{-1}) & = & \sum_i S(C_i)D_i\rule{0mm}{7mm}.
\end{array}
\end{equation}
They will play an important role in later formulae and duality consideration. For a coboundary $\F_c$ is readily computed 
as $\xt_{\F_c}=c^{-1}S(c)^{-1}\,$ and $\overline{\xt}_{\F_c}=S(c)c\,$. Note also that if $\F$ is a cocycle then also 
\begin{equation}\label{eq-Fdual}
 \F^{\dagger}\;:=\;\JJ(\F^{-1})
\end{equation}
is a gauge transformation for $\H\,$.  This can be iterated to obtain more 
gauge transformations $\F^{\dagger\dagger}=(S^2\otimes S^2)(\F)$, $\F^{\dagger\dagger\dagger}$ etc. for the
same $\H\,$.
With the above convention we may now list several
useful relations.
\begin{lemma}\label{lm-xFrels}
Suppose $\H$ is a Hopf algebra and {\rm $\F$} a gauge transformation with {\rm $\xt_{\F}$} as above. Then the following hold.
{\rm \begin{gather}
\textstyle
\sum_iC_i {\xt}_{\F}  S(D_i)\;=\;1\;=\;\sum_iS(A_i) \overline{\xt}_{\F}  B_i \label{eq-xFrel-wrap}\\
\overline{\xt}_{\F}\;=\;\xt_{\F}^{-1}\label{eq-xFrel-xinv}\\
\F^{-1} \, = \; \Jr(\F) (1\otimes \xt_{\F}^{-1}) \; =\; \Jl^{-1}(\F) ( S^{-1}(\xt_{\F}^{-1})\otimes1)  \rule{0mm}{5mm} \label{eq-xFrel-FF-1}\\
\;\;\F \; = \;  (\xt_{\F} \otimes 1)\Jl(\F^{-1}) \; =\; (1\otimes S^{-1}(\xt_{\F}))\Jr^{-1}(\F^{-1}) \rule{0mm}{5mm}\label{eq-xFrel-F-1F}\\
\;\;\F^{\dagger}\;=\;S\otimes S(\tau(\F^{-1}))\;=\;(\xt_{\F}^{-1}\otimes\xt_{\F}^{-1})\F\Delta(\xt_{\F})\rule{0mm}{5mm}\label{eq-xFrel-dual}
\end{gather}}
\end{lemma}
\begin{proof}
The relations in (\ref{eq-xFrel-wrap}) immediately follow by applying $m(id\otimes S)$ and $m(S\otimes id)$  to the equation $\F^{-1}\F=1\,$
respectively. If we apply $m_{23}(id^{\otimes 2}\otimes S)$ and  $m_{21}(S^{-1}\otimes id^{\otimes 2})$  to the cocycle equation  (\ref{eq-Fcocyle}) we find 
$1\otimes \xt_{\F}\,=\,\F \left(\sum_iA_i'\otimes A_i''S(B_i)\right)$ as well as $S^{-1}(\xt_{\F})\otimes 1\,=\,\F \left(\sum_iB_i' S^{-1}(A_i)\otimes B_i''\right)$
respectively. The latter imply $\F^{-1}(1\otimes \xt_{\F} )=\Jr(\F)$ and $\F^{-1}(S^{-1}(\xt_{\F})\otimes 1 )=\Jl^{-1}(\F)\,$. Applying $m(S\otimes id)$ to
both of these equations we obtain $\overline{\xt}_{\F}\xt_{\F}=1$ and  $\xt_{\F}\overline{\xt}_{\F}=1$ respectively, which proves (\ref{eq-xFrel-xinv}).

Given invertibility of $\xt_{\F}$ the previous equations can be solved for $\F^{-1}$ yielding (\ref{eq-xFrel-FF-1}). The second set of equations
in (\ref{eq-xFrel-F-1F}) is obtained similarly by applying $m_{12}(S\otimes id^{\otimes 2})$ and $m_{32}(id^{\otimes 2}\otimes S^{-1})$
to the inverse of the cocycle equation given by
$$(id \otimes \Delta)(\F^{-1})(1\otimes \F^{-1})= (\Delta \otimes id)(\F^{-1})(\F^{-1}\otimes 1)$$
respectively. For the last identity we compute
\begin{equation*}
 \begin{split}
(1\otimes \xt_{\F})\JJ(\F^{-1})
&\byeq{eq-KJ}
\JJ((S^{-1}(\xt_{\F})\otimes 1)\F^{-1})
\byeq{eq-KJrels}
\Jl(\Jr(\Jl((S^{-1}(\xt_{\F})\otimes 1)\F^{-1})))
\\
&\byeq{eq-Jrl}
\Jl(\Jr((\xt_{\F}\otimes 1)\Jl(\F^{-1})))
\byeq{eq-xFrel-F-1F}
\Jl(\Jr(\F))
\byeq{eq-xFrel-FF-1}
\Jl(\F^{-1}(1\otimes \xt_{\F}))
\\
&\byeq{eq-Jrl}
\Jl(\F^{-1})\Delta(\xt_{\F})
\byeq{eq-xFrel-F-1F}
(\xt_{\F}^{-1}\otimes 1)\F\Delta(\xt_{\F})
\end{split}
\end{equation*}
which readily implies (\ref{eq-xFrel-dual}).
\end{proof}

\subsection{Gauge Transformed Quasi-Triangular Structure }\label{ss-GaugeQT}

Let us now turn to defining the gauge transformed Hopf algebra structures. The following is well known and implied, for example,
by specializing computations in \cite{dr89} to the strictly associative case with trivial associators.

\begin{lemma}\label{lm-GT-Hopf}
Suppose $\H$ is a Hopf algebra and {\rm $\F$} a gauge transformation of $\H$ as defined above.
Then a Hopf algebra {\rm $\H_{\F}$} can be defined with the same algebra structure as $\H$ but with a Hopf algebra
structure as given follows: 
{\rm \begin{eqnarray}
 \Delta_{\F}(a)  & = & \F\Delta(a)\F^{-1}\label{eq-GT-Delta}\\
\epsilon_{\F}  & = & \epsilon\label{eq-GT-counit}\\
S_{\F}(a) & = & \xt_{\F}S(a)\xt_{\F}^{-1}\label{eq-GT-S}
\end{eqnarray}}
If $\H$ is a quasi-triangular with R-matrix $\R$. Then {\rm $\H_{\F}$} is also quasi-triangular with
{\rm \begin{equation}\label{eq-Rgauge}
 \R_{\F}\,=\, \F_{21}\R\F^{-1}
\end{equation}}
where we use the usual notation {\rm $\F_{21}=\tau(\F)$} with $\tau(a\otimes b)=b\otimes a$. 
\end{lemma}
\begin{proof} The counit and coassociativity axioms immediately follow from (\ref{eq-Fcocyle}). Also $\Delta_{\F}(ab)=\Delta_{\F}(a)\Delta_{\F}(b)$
is clear from (\ref{eq-GT-Delta}). For the antipode we compute
\begin{flalign*}
&\textstyle\qquad m(id\otimes S_{\F})\Delta_{\F}(a)=\sum_{ij}A_ia'C_j\xt_{\F}S(D_j)S(a'')S(B_i)\xt_{\F}^{-1}=&\\
\mbox{by (\ref{eq-xFrel-wrap})} &\textstyle\qquad \quad 
=\sum_{i}A_ia'S(a'')S(B_i)\xt_{\F}^{-1}= \epsilon(a)\sum_{i}A_iS(B_i)\xt_{\F}^{-1}=\epsilon(a)\xt_{\F}\xt_{\F}^{-1} =\epsilon_{\F}(a)\,.&
\end{flalign*}
The other antipode equation $m(S_{\F}\otimes id)\Delta_{\F}(a)=\epsilon_{\F}(a)$ follows similarly from the second relation in
(\ref{eq-xFrel-wrap}). For the quasi-triangular structure the identity $\Delta_{\F}'(a)\R_{\F}=\R_{\F}\Delta_{\F}(a)$, where $\Delta'_{\F}$
denotes the opposite coproduct, is immediate by conjugation.

Although already implied by computations in \cite{dr89} for quasi Hopf algebra let us also give a derivation of the remaining 
quasi-triangularity axioms as a warm up for later uses of the cocycle equation (\ref{eq-Fcocyle}). To this end denote the 
left and right side of (\ref{eq-Fcocyle}) by  $\Fa 3=\F_{12}(\Delta\otimes id)(\F)=\F_{23}(id \otimes \Delta)(\F)\in\H^{\otimes 3}$.
For $\pi$ a permutation on $n$ letters let $\sigma_{\pi}$ denote the automorphism on $\H^{\otimes n}$ given by 
$\sigma_{\pi}(a_1\otimes\ldots\otimes a_n)= a_{\pi(1)}\otimes\ldots\otimes a_{\pi(n)}\,$. The following two relations in
$\H^{\otimes 3}$ are then easily verified. 
$$
(\R_{\F})_{12}\cdot \Fa 3 = \sigma_{(12)}(\Fa 3)\cdot \R_{12} \qquad \mbox{ and } \qquad (\R_{\F})_{23}\cdot \Fa 3 = \sigma_{(23)}(\Fa 3)\cdot \R_{23}\,.
$$
Combining these two identities we find
$$
(\R_{\F})_{13}\cdot (\R_{\F})_{23}\cdot \Fa 3 = \sigma_{(23)}\left( (\R_{\F})_{12}\cdot \Fa 3  \right) \cdot \R_{23} 
= \sigma_{(123)}\left(  \Fa 3  \right) \cdot \R_{13} \cdot \R_{23} \,.
$$
We also compute
\begin{equation*}
\begin{split}
(\Delta_{\F}\otimes id)(\R_{\F})\cdot \Fa 3 &= \F_{12}\cdot (\Delta\otimes id)(\R_{\F})\cdot (\Delta \otimes id)(\F)\\
&= \F_{12}(\Delta \otimes id)(\F_{21}) \cdot  (\Delta\otimes id)(\R) \, = \, \sigma_{(123)}\left(  \Fa 3  \right) \cdot (\Delta\otimes id)(\R) \,.
\end{split}
\end{equation*}
Now since $(\Delta\otimes id)(\R) =  \R_{13} \cdot \R_{23}$ by assumption and $\Fa 3$ is invertible the identities above
can be combined to yield the desired axiom 
$(\Delta_{\F}\otimes id)(\R_{\F}) =  (\R_{\F})_{13} \cdot (\R_{\F})_{23}\,$ for quasi-triangular Hopf algebras. 
The second axiom 
$(id \otimes \Delta_{\F})(\R_{\F}) =  (\R_{\F})_{13} \cdot (\R_{\F})_{12}\,$ follows entirely analogously. 
\end{proof}

\subsection{Gauge Transformed Ribbon and Balancing Elements}\label{ss-GaugeBal}

We next discuss the effect of a gauge transformation on balancing and ribbon elements. To this end it is useful to introduce the
following element.
\begin{equation}\label{eq-defzF}
\zt_{\F}\,:=\,\xt_{\F} S(\xt_{\F})^{-1}
\end{equation}

\begin{lemma}\label{lm-GT-bal}
Suppose $\H$ is balanced with ribbon element $\r\in\H$ and balancing element $\kappa\in\H$, and {\rm $\F$} is a gauge 
transformation as above. Then {\rm $\H_{\F}$} is also balanced with elements as follows:
{\rm \begin{equation}\label{eq-GT-bal}
 u_{\F}=\zt_{\F}u \qquad\qquad \r_{\F}=\r  \qquad\qquad \kappa_{\F}=\zt_{\F}\kappa
\end{equation}}
\end{lemma}

\begin{proof} 
Given that $\R_{\F}=\sum_{ikj}B_ie_kC_j\otimes A_if_kD_j$ we compute $u_{\F}$ as follows.
\begin{equation*}
\begin{split}
 u_{\F}\,&\textstyle 
=\,\sum_{ikj}S_{\F}(A_if_kD_j)B_ie_kC_j\,\byeq{eq-GT-S}\,\sum_{ikj}\xt_{\F}S(D_j)S(f_k)S(A_i)\xt_{\F}^{-1}B_ie_kC_j\,=\\
&\textstyle 
\byeq{eq-xFrel-wrap}\,\sum_{kj}\xt_{\F}S(D_j)S(f_k)e_kC_j\,=\,\sum_{j}\xt_{\F}S(D_j)uC_j
\,=\,\xt_{\F}\sum_{j}S(D_j)S^2(C_j)u\\
&\textstyle 
\,=\,\xt_{\F}S(\sum_{j}S(C_j)D_j)u\,\,\byeq{eq-xFrel-xinv}\,\xt_{\F}S(\xt_{\F}^{-1})u\,=\,\zt_{\F}u\;.
\end{split}
\end{equation*}
We verify the ribbon element using the characterization in (\ref{eq-unimod}).
If we set $\r_{\F}=\r$ centrality is obvious since multiplication is the same in $\H_{\F}$. This also implies $S_{\F}(\r_{\F})=
\xt_{\F}S(\r)\xt_{\F}^{-1}=\xt_{\F}\r\xt_{\F}^{-1}=\r=\r_{\F}\,$. Moreover, we find 
$$
(\R_{\F})_{21}\R_{\F}=
\F \R_{21}\R\F^{-1}=
\F (\r\otimes \r)\Delta(\r^{-1}) \F^{-1}=
(\r\otimes \r)\Delta_{\F}(\r^{-1})\,,
$$ 
where we use again that $\r$ and hence $\r\otimes \r$  is central. Finally, the equivalence of a ribbon structure and 
a balancing (see \cite{ke94}) implies that $\kappa_{\F}=u_{\F}(\r_{\F})^{-1}=\zt_{\F}u\r^{-1}=\zt_{\F}\kappa$ as desired.
\end{proof}
Note that Lemma~\ref{lm-GT-bal} implies that $\kappa_{\F}$ as given in (\ref{eq-GT-bal}) is group like with respect to $\Delta_{\F}$.
This is indeed true even without an underlying quasi-triangular structure as verified in the following lemma.

\begin{lemma}\label{lm-GT-balnr}
Suppose $\H$ is a Hopf algebra with gauge transformation {\rm $\F$} and $\kappa$ is a group like
element with $S^2(a)=\kappa a \kappa^{-1}$. Then {\rm $\kappa_{\F}=\zt_{\F}\kappa$} is group like in {\rm $\H_{\F}$} and 
satisfies {\rm $S^2_{\F}(a)=\kappa_{\F}a\kappa_{\F}^{-1}$}.
\end{lemma}

\begin{proof}
Iterating (\ref{eq-GT-S}) we find that $S_{\F}^2(a)=\zt_{\F}S^2(a)\zt_{\F}^{-1}$ which implies immediately
that $S_{\F}^2$ is given by conjugation with $\kappa_{\F}$. Moreover, using (\ref{eq-xFrel-dual}) we compute
\begin{equation*}
 \begin{split}
(\kappa\otimes\kappa)\F\Delta(\kappa^{-1})&=
(\kappa\otimes\kappa)\F(\kappa^{-1}\otimes\kappa^{-1})=
S^2\otimes S^2(\F)=\JJ( (\JJ(\F^{-1})^{-1})\\
&=
\JJ(
\Delta(\xt_{\F}^{-1})
\F^{-1}
(\xt_{\F}\otimes\xt_{\F})
)
=(S(\xt_{\F})\otimes S(\xt_{\F}))\JJ(\F^{-1})\Delta(S(\xt_{\F})^{-1})\\
&=(S(\xt_{\F})\xt_{\F}^{-1}\otimes S(\xt_{\F})\xt_{\F}^{-1})\F\Delta(\xt_{\F})\Delta(S(\xt_{\F})^{-1})
=(\zt_{\F}^{-1}\otimes \zt_{\F}^{-1})\F\Delta(\zt_{\F})\,.
 \end{split}
\end{equation*}
which implies $\kappa_{\F}\otimes\kappa_{\F}=(\zt_{\F}\otimes \zt_{\F})(\kappa\otimes\kappa)=\F\Delta(\zt_{\F})\Delta(\kappa)\F^{-1}=\Delta_{\F}(\kappa_{\F})$
as claimed.  
\end{proof}
Similar calculations to the ones above, using a different formalism, occur in \cite{aegn}.

\subsection{Gauge Transformed Integrals}\label{ss-GaugeInt}

We finally consider the effect of gauge transformations on integrals of $\H$. In our discussion we confine ourselves
to the unimodular case (that is, when $\Lambda\in\H$ is a two-sided integral) since this is the only relevant case for 
TQFT constructions. Generalizations to the non-unimodular case follow the same lines with additional elements such 
as $(\alpha\otimes id)(\F)\in\H$ where $\alpha$ is the comodulus. Details are left to the interested reader. 

In the unimodular case $\Lambda_{\F}=\Lambda$ is clearly also a two-sided integral in $\H_{\F}$ since the algebra
structure remains the same. This simple observation allows us to determine $\lambda_{\F}$ using non-degenerate
forms on $\H$ and $\H^*$ obtained from integrals as in \cite{larson}.

\begin{lemma}\label{lm-integralF}
Let $\H$ be a unimodular Hopf algebra with right integrals $\lambda\in\H^*$ and $\Lambda\in\H$ with $\lambda(\Lambda)=1$.
Then 
{\rm
\begin{equation}\label{eq-integralF}
\lambda_{\F}\;=\;\lambda\lar\zt_{\F}^{-1}
\end{equation}\label{eq-lF}}
is the unique right integral for {\rm $\H_{\F}$} with {\rm $\lambda_{\F}(\Lambda_{\F})=\lambda_{\F}(\Lambda)=1$}. 
\end{lemma}

\begin{proof} 
The strategy is to determine $\lambda_{\F}$ from (\ref{eq-bbS}) by computing $\beta_{\F}$ from formulae for  $\overline\beta_{\F}$ and $S_{\F}$ in $\H_{\F}$. 
In order to determine  $\overline\beta_{\F}$ observe that 
$\Jl(a\otimes b)\Delta(\Lambda)=\Jl(a\otimes b\Lambda)=\epsilon(b)\Jl(a\otimes \Lambda)=\epsilon(b)(S(a)\otimes 1)\Delta(\Lambda)$
and, by similar calculation,
$\Delta(\Lambda)\Jl^{-1}(a\otimes b)=\Delta(\Lambda) \epsilon(b)(S^{-1}(a)\otimes 1)$. This implies by  (\ref{eq-Fcocyle}) that
$\Jl(\F^{-1})\Delta(\Lambda)=\Delta(\Lambda)=\Delta(\Lambda)\Jl^{-1}(\F)$ and, hence,
\begin{equation*}
\begin{split}
\Delta_{\F}(\Lambda_{\F})&=\F\Delta(\Lambda)\F^{-1}=(\xt_{\F}\otimes 1)\Jl(\F^{-1})\Delta(\Lambda)\Jl^{-1}(\F)(S^{-1}(\xt_{\F}^{-1})\otimes 1)\\
&=(\xt_{\F}\otimes 1) \Delta(\Lambda) (S^{-1}(\xt_{\F}^{-1})\otimes 1)
\end{split}
\end{equation*}
The value for $\overline\beta_{\F}(f)$ is now obtained by applying $f\otimes id$ to $\Delta_{\F}(\Lambda_{\F})$. 
Expressing the multiplication by the elements in the first tensor factor by the actions in (\ref{e:co-arrows}) we find 
$$
\overline\beta_{\F}(f)\;=\;\overline\beta\left(S^{-1}(\xt_{\F}^{-1})\rar f  \lar  \xt_{\F}  \right) \quad\Rightarrow\quad
\overline\beta_{\F}^{\,-1}(a)\;=\; S^{-1}(\xt_{\F})\rar \overline\beta^{\,-1}(a)  \lar  \xt_{\F}^{-1} 
$$
Now $\lambda=\beta(1)=(\overline\beta)^{-1}(S(1))=(\overline\beta)^{-1}(1)$ and similarly $\lambda_{\F}=\overline\beta^{-1}_{\F}(1)\,$
so that
\begin{equation}\label{eq-lFalt}
\lambda_{\F}\;=\;S^{-1}(\xt_{\F})\rar \lambda\lar \xt_{\F}^{-1}\;.
\end{equation}
This can be rewritten as $\lambda_{\F}(a)=\lambda(\xt_{\F}^{-1}aS^{-1}(\xt_{\F}))=\lambda(S(\xt_{\F})\xt_{\F}^{-1}a)=\lambda(\zt_{\F}^{-1}a)\,$
where we also used (\ref{eq-unimod}).
\end{proof}

\section{Gauge Transformations for Hennings TQFTs}\label{sec-gaugeequiv}

In this section we will study how the TQFTs constructed in Section~\ref{sec-tqft} behave under gauge transformations
of the underlying Hopf algebras. Particularly, we will construct an explicit natural isomorphism between functors
 $\tqftVc {\H}$ and $\tqftVc {\H_{\F}}$ for a given \nicehopf Hopf algebra and gauge transformation $\F\,$. 
The main goal of this section is thus to prove Theorem~\ref{thm-GTtqft}.

\subsection{Gauge Elements for Higher Tensor Products}\label{s2.5}

The first ingredient in the construction of an isomorphisms between these functors is the extension of 
$\F$ as an intertwiner of coalgebra structures to higher tensor 
powers as follows. We start by defining the operation
 \begin{equation}\label{eq-defY}
\begin{split}
\Y&:\H\to \H\otimes \H:\,x\mapsto \F\Delta(x) \qquad \mbox{and} \qquad \overline\Y:\H\to \H\otimes \H:\,x\mapsto \Delta(x)\F^{-1}\\
&\Y_{i}^{(n)}=id^{\otimes (i-1)}\otimes\Y\!\otimes id^{\otimes (n-i)}:\,\H^{\otimes n}\to \H^{\otimes (n+1)}
\end{split}
 \end{equation}
and analogously $\overline\Y_{i}^{(n)}:\,\H^{\otimes n}\to \H^{\otimes (n+1)}$.
Moreover, for a sequence of integer indices $(i_1, \ldots , i_n)$ with $1\leq i_k\leq k$ for all $k$ we define 
\begin{equation}\label{eq-defFa}
\begin{split}
 \Fa {n+1} &:=\Y_{i_n}^{(n)}\circ\Y_{i_{(n-1)}}^{(n-1)}\circ\ldots \circ \Y_{i_1}^{(1)}(1)\;\in\,\H^{\otimes (n+1)}\\
\Fao {n+1} &:=\overline\Y_{i_n}^{(n)}\circ\overline\Y_{i_{(n-1)}}^{(n-1)}\circ\ldots \circ \overline\Y_{i_1}^{(1)}(1)\;\in\,\H^{\otimes (n+1)}
\end{split}
\end{equation}
For example, we have $\Fa 1=\F$, and $\Fao 1=\F^{-1}$. For $n=2$ we obtain the two sides of the cocycle condition
(or its inverse) depending on whether we choose $i_2=1$ or $i_2=2$. That is, $\Fa 2=(\F\otimes 1)(\Delta\otimes id)(\F)=
(1\otimes \F)(id\otimes \Delta)(\F)=$ and $\Fao 2=(\Delta\otimes id)(\F^{-1})(\F^{-1}\otimes 1)=(id\otimes \Delta)(\F^{-1})(1\otimes \F^{-1})$
are inverses and independent of the choice of $i_2$. This observation and other properties of $\F$ are generalized to $\Fa n$ next.

\begin{lemma}\label{lm-PsiIds} 
For  $\Fa n$ and  $\Fao n$ as above and for any $\,f\in\H^*\,$, $x\in \H$ and $i=1,\ldots,n$ we have the following 
identities.
{\rm \begin{align}
\mbox{\em $\Fa n$ does not depend on the choice} & \mbox{\em \ of indices $(i_1, \ldots, i_n)$ but only on $n$}
\label{eq-Fa-indep}
\\
id^{\otimes i-1}\otimes\epsilon\otimes id^{\otimes (n-i)}(\Fa {n})\,&=\,\Fa {n-1}
\label{eq-Fa-eps}
\\
\Fa n^{-1} & =\Fao n
\label{eq-Fao=inv}
\\
\Fa n\Delta^{(n)}(x) & =\Delta^{(n)}_{\F}(x)\Fa n
\label{eq-Fa-intertw}
\\
(1^{\otimes i-1}\otimes\R_{\F}\otimes 1^{\otimes (n-i)})\Fa {n+1}\;
& =\;\sigma_{(i,i+1)}\left(\Fa {n+1}\right)(1^{\otimes i-1}\otimes\R\otimes 1^{\otimes (n-i)})
\label{eq-Fa-R}
\\
(id^{\otimes i-1}\otimes(f\!\circ\! m \!\circ\! (id\otimes S))\otimes & 1^{\otimes (n-i)})(\Fa {n+1})\,
=\,f(\xt_{\F})\Fa {n-1}
\label{eq-Fa-dual-a}
\\
(id^{\otimes i-1}\otimes(f\!\circ\! m \!\circ\! (S\otimes id))\otimes & 1^{\otimes (n-i)})(\Fao {n+1})\,
=\,f(\xt_{\F}^{-1})\Fao {n-1}
\label{eq-Fa-dual-b}
\end{align}
}
\end{lemma}
\begin{proof}
For the statement in (\ref{eq-Fa-indep}) observe that (\ref{eq-Fcocyle}) and coassociativity imply that $\,\Y_1^{(3)}\!\circ\! \Y=\Y_2^{(3)}\!\circ\! \Y\,$
and hence $\,\Y_i^{(n+1)}\!\circ\! \Y_i^{(n)}\,=\,\Y_{i+1}^{(n+1)}\!\circ\! \Y_i^{(n)}\,$. Assuming the statement in (\ref{eq-Fa-indep}) is true for $n=m-1$ 
we may set $i_{m-1}$ to be $i_m$. We find $\,\Fa {m+1}=\Y_{i_m}(\Fa {m})=\Y_{i_m}(\Y_{i_m}(\Fa {m-1}))=\Y_{i_m+1}(\Y_{i_m}(\Fa {m-1}))=\Y_{i_m+1}(\Fa {m})\,$. That is,
$i_m$ may be replaced by $i_m+1$ implying independence of the choice of $i_m$.  (\ref{eq-Fa-indep}) thus follows by induction in $m$.

Identity (\ref{eq-Fa-eps}) follows immediately from $(\epsilon\otimes id)(\Y(x))=(id \otimes \epsilon)(\Y(x))=x\,$.
For (\ref{eq-Fao=inv}) note that $\Y(x)\overline\Y(y)=\Delta_{\F}(xy)$ so that, again by induction in $n$, we find
$\Fa {n+1} \Fao {n+1} = \Y_1(\Fa n)\overline \Y_1(\Fao n)=(\Delta_{\F}\otimes id^{\otimes n-1})(\Fa n\Fao n)
=(\Delta_{\F}\otimes id^{\otimes n-1})(1)=1\,$. Similarly, $\overline\Y(x)\Y(y)=\Delta(xy)$ implies $\Fao {n+1} \Fa {n+1}=1$.
The induction step for proving  (\ref{eq-Fa-intertw}) is as follows.
\begin{equation}\nonumber
\begin{split}
\Fa {n+1}\Delta^{(n+1)}(x)
&=(\F\otimes 1^{\otimes n-1})(\Delta\otimes id^{\otimes n-1})(\Fa n\Delta^{(n)}(x))\\
&=(\F\otimes 1^{\otimes n-1})(\Delta\otimes id^{\otimes n-1})(\Delta^{(n)}_{\F}(x)\Fa n)\\
&=(\Delta_{\F}\otimes id^{\otimes n-1})(\Delta^{(n)}_{\F}(x))(\F\otimes 1^{\otimes n-1})(\Delta\otimes id^{\otimes n-1})(\Fa n)\\
&=\Delta^{(n+1)}_{\F}(x)\Fa {n+1}\,.
\end{split}
\end{equation}
For (\ref{eq-Fa-R}) substitute $\Fa {n+1}=\Y_{i}(\Fa {n})$ and observe that $\R_F\Y(x)=\F_{21}\R \Delta(x)=\sigma_{(12)}(\F)\sigma_{(12)}(\Delta(x))\R
=\sigma_{(12)}(\Y(x))\R\,$. Substituting also $\Fa {n+1}=\Y_{i}(\Fa {n})$  but using
\begin{equation}\nonumber
\begin{split}
(f\circ m \circ (id\otimes S))(\Y(x))&=f(\sum_i A_ix'S(B_ix''))=f(\sum_iA_ix'S(x'')S(B_i))=\\
&=\epsilon(x)f(\sum_iA_iS(B_i))
=\epsilon(x)f(\xt_{\F})
\end{split}
\end{equation}
 we prove (\ref{eq-Fa-dual-a}) using also (\ref{eq-Fa-eps}). Equation (\ref{eq-Fa-dual-b}) follows analogously.
\end{proof}

\subsection{Gauge Transformations of Planar Diagrams}\label{ss-GaugeDiag}

Next relate the algebraic formalism of gauge transformations above to 
the combinatorial calculus introduced in Section~\ref{sss-CHLPC} for categories of $\H$-labeled planar diagrams. 
Consider the following map of decorated diagrams which inserts into a given
diagram an additional element $\xt_{\F}^{\pm 1}$ near each maximum and minimum as indicated
in Figure~\ref{fig-insert} but leaves the diagram unchanged otherwise. 

\begin{figure}[ht]
\centering
\psfrag{m}{\LARGE $\mapsto$}
\psfrag{x}{$\xt_{\F}$}
\psfrag{y}{$\xt_{\F}^{-1}$}
\hspace{-1ex}\includegraphics[width=.73\textwidth]{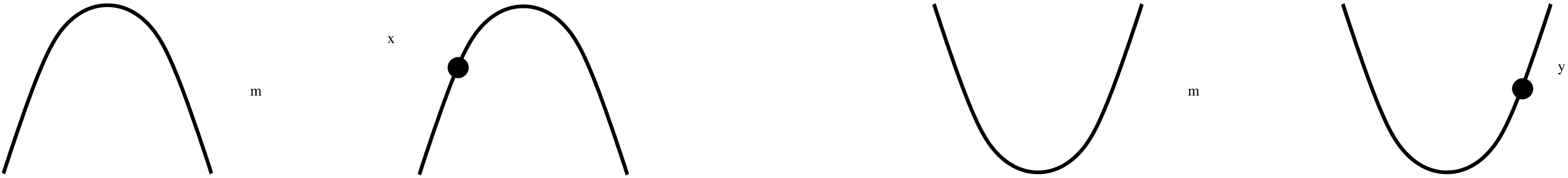}
\caption{Extrema insertions}\label{fig-insert}
\end{figure}

\begin{lemma}\label{lm-insert}
The map of decorated diagrams defined in Figure~\ref{fig-insert} factors 
into a well defined functor 
{\rm \begin{equation}\label{eq-insert}
\Ins_{\F}\,:\;\Diag(\H_{\F})\,\longrightarrow\,\Diag(\H)\,.
\end{equation}}
\end{lemma}

\begin{proof}
We need to check that equivalence classes of diagrams in $\Diag(\H_{\F})$ are mapped to 
classes for $\Diag(\H)$. The fact is clear for the planar
second and third Reidemeister moves 
since these do not involve extrema. Similarly, moving an extremum 
through a crossing follows since the extra $\xt_{\F}^{\pm 1}$ can be moved through
the crossing as in Figure~\ref{f4-1}. 

Cancellation of a maximum and a minimum  
is seen by canceling the $\xt_{\F}$ assigned to the maximum with the $\xt_{\F}^{-1}$
added near the minimum. It remains to verify the moves in Figures~\ref{f4-1} and Figure~\ref{f5-1}.
The first and third picture in  Figure~\ref{f4-1} as well as the first in Figure~\ref{f5-1}
are immediate. For the second equivalence in Figures~\ref{f4-1} we apply 
$\Ins_{\F}$ to either side of the equation (in $\Diag(\H_{\F})$) which yields 
the diagrams in $\Diag(\H)$ as depicted in Figure~\ref{fig-insertSmove1}. 
Clearly, by (\ref{eq-GT-S}), the resulting diagrams are the same.
\begin{figure}[ht]
\centering
\psfrag{m}{\large $\Ins_{\F}$}
\psfrag{x}{$x$}
\psfrag{s}{$S(x)$}
\psfrag{y}{$\xt_{\F}^{-1}$}
\psfrag{z}{$S(x)\xt_{\F}^{-1}$}
\psfrag{e}{$=$}
\psfrag{f}{$S_{\F}(x)$} 
\psfrag{w}{$\xt_{\F}^{-1}S_{\F}(x)$}
\hspace{-6ex}\includegraphics[width=.8\textwidth]{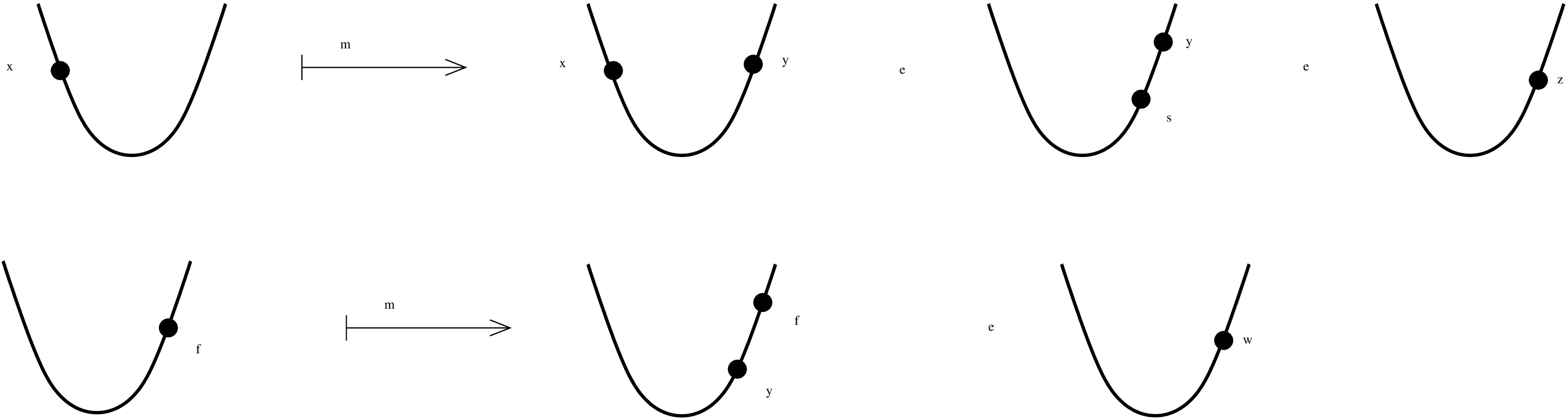}
\caption{$\Ins_{\F}$-invariance for move through minimum}\label{fig-insertSmove1}
\end{figure}

For the second picture in Figure~\ref{f5-1} we proceed similarly by applying $\Ins_{\F}$ 
to both sides of the equivalence as indicated in Figure~\ref{fig-insertkappa}. The equality
of the resulting diagrams in $\Diag(\H)$ is implied by 
(\ref{eq-defzF}), (\ref{eq-GT-bal}) and (\ref{kappa}),
as well as $\xt_{\F}^{-1}\kappa_{\F}=\xt_{\F}^{-1}\zt_{\F}\kappa=S(\xt_{\F}^{-1})\kappa
=\kappa S^{-1}(\xt_{\F}^{-1})$.
\begin{figure}[ht]
\centering
\psfrag{m}{\large $\Ins_{\F}$}
\psfrag{y}{$\xt_{\F}^{-1}$}
\psfrag{p}{$S^{-1}(\xt_{\F}^{-1})$}
\psfrag{e}{$=$}
\psfrag{k}{$\kappa$}
\psfrag{q}{$\kappa_{\F}$}
\psfrag{v}{$\xt_{\F}^{-1}\kappa_{\F}$}
\psfrag{z}{$\kappa S^{-1}(\xt_{\F}^{-1})$}
\hspace{-6ex}\includegraphics[width=.8\textwidth]{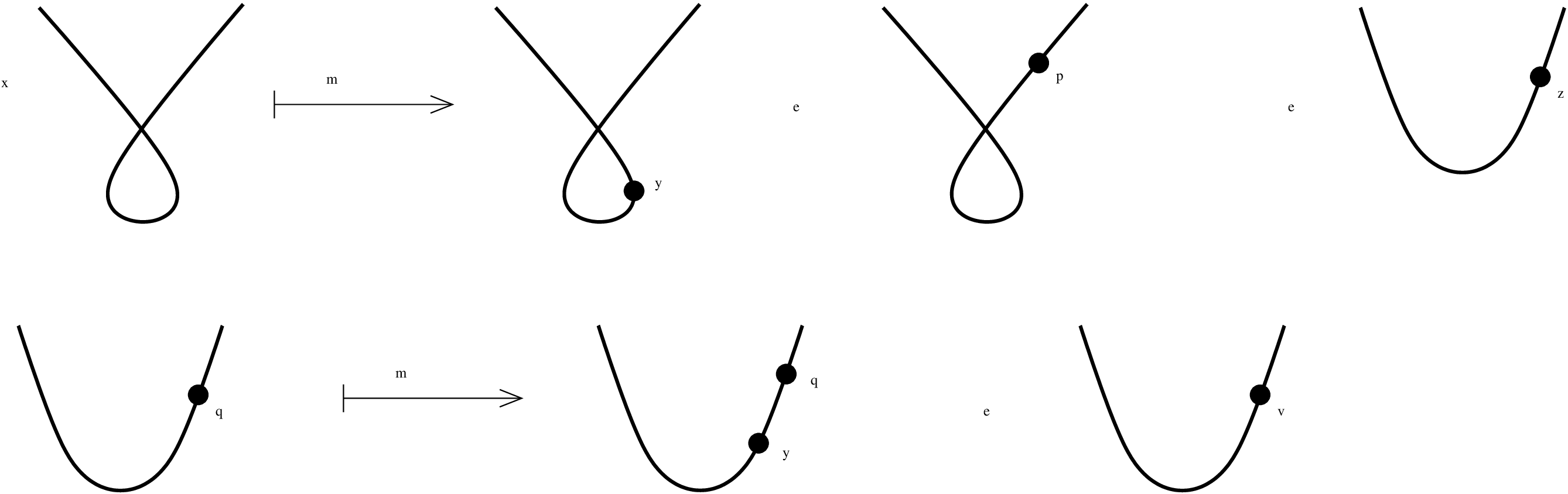}
\caption{$\Ins_{\F}$-invariance for first Reidemeister move}\label{fig-insertkappa}
\end{figure}
\end{proof}

We observe next that the elements $\Fa {n}\in\H^{\otimes g}$ introduced in
(\ref{eq-defFa}) can be thought of as a collection of morphisms in the
categories of decorated diagrams by distributing the tensor factors of
$\Fa {n}=\sum \Fa {n}^{(1)}\otimes \ldots\otimes \Fa {n}^{(n)}$ over $n$
parallel strands exactly as we did for $\Delta^{(2g-1)}(x)$
in (\ref{fig-Hact}).

In the category $\Diag^0(\H)$ of decorated diagrams without connectivity constraints 
(see beginning of Section~\ref{sec-tqft-invar}) this yields a 
collection $\Fa {*}^0=\{\Fa {n}^0:n\to n\}$ of isomorphisms in $\Diag^0(\H)$.
In the (sub) category $\Diag(\H)$ we thus obtain a collection 
$\Fa {*}=\{\Fa {2k}:k\to k\}\,$ of isomorphisms.

In the next lemma we will interpret $\Fa {*}^0$ and $\Fa {*}$ as natural
transformations between functors with target categories $\Diag^0(\H)$
and $\Diag(\H)$ respectively.

\begin{lemma}\label{lm-natisoPsi}
With conventions as above we have the following natural isomorphism of functors
from $\TangNM$ to $\Diag(\H)$:{\rm
\begin{equation}\label{eq-natisoPsi}
\Fa {*}\,:\,\decor{\H}\,\natto\,\Ins_{\F}\circ \decor{\H_{\F}}\,.
\end{equation}}
\end{lemma}

\begin{proof}
We will in fact prove that the isomorphism 
$\,\Fa {*}^0:\,\decor{\H}^0\natto\Ins_{\F}\circ \decor{\H_{\F}}^0\,$ is defined on
general tangles without connectivity constraints, which obviously implies
Lemma~\ref{lm-natisoPsi}. Explicitly we need to show that for every tangle
$T$ with $N$ end points at the top and $M$  end points at the bottom we have
\begin{equation}\label{eq-natisoExpl}
\Fa {M}^0\cdot \decor{\H}^0(T)\cdot (\Fa {N}^0)^{-1}\,=\,
\Ins_{\F}\left({\decor{\H_{\F}}^0(T)}\right)\;.
\end{equation}
By functoriality it suffices to prove (\ref{eq-natisoExpl}) for the generators of
$\TangNM\,$. The three types of generators are a single maximum, a single minimum, 
and a single positive crossing, each having some number of parallel strands to the
right and left. In the following picture we evaluate the left hand side of
(\ref{eq-natisoExpl}) by making use of the identities in Lemma~\ref{lm-PsiIds}.

In the evaluation of a crossing in Figure~\ref{fig-natiso-cross} we first apply
the rule from Figure~\ref{f1} assigning the factors of $\R\in\H^{\otimes 2}$ to the
two markings and then conjugate the diagram by $\Fa {n+1}$ assuming we have $n+1$ strands.
The equality between the last two diagrams follows by multiplying them both by
$\Fa {n+1}$ and a crossing from the bottom. Given that conjugation of $\Fa {n+1}$
by crossings is the
same as transposition of factors the resulting equality is the same as (\ref{eq-Fa-R}).
\begin{figure}[ht]
\centering
\psfrag{X}{\footnotesize $\Fa {n+1}^{-1}$}
\psfrag{Y}{\footnotesize $\Fa {n+1}$}
\psfrag{R}{\small $\R$}
\psfrag{Q}{\small $\R_{\F}$}
\psfrag{m}{$\decor{\H}^{\mbox{\footnotesize$\Fa {*}^0$}}$}
\psfrag{e}{$\byeq{eq-Fa-R}$}
\hspace{-6ex}\includegraphics[width=.9\textwidth]{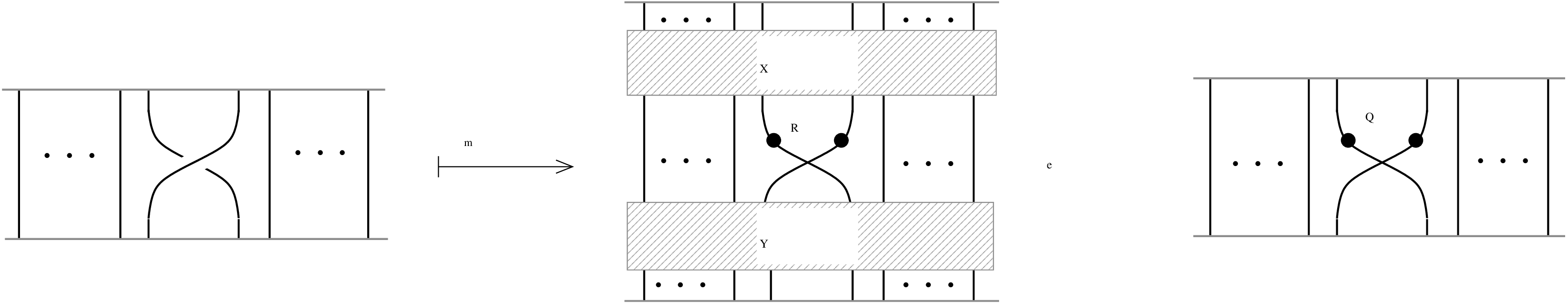}
\caption{$\decor{\H}^{\mbox{\footnotesize$\Fa {*}^0$}}$ for crossing}\label{fig-natiso-cross}
\end{figure}

Since there are no extrema involved in a crossing the resulting diagram in $\Diag^0(\H_{\F})$ 
is the same as evaluating
$\Ins_{\F}\circ\decor{\H_{\F}}^0$ on a single crossing so that (\ref{eq-natisoExpl}) holds
for this case.

The proof for maximum proceeds similarly as depicted in Figure~\ref{fig-natiso-max}.
For the last identity we can multiply the diagram from the top with $\Fa {n-1}$ and
evaluate the combined 
decoration along the arc with an arbitrary linear form $f\in\H^*$.
The identity to prove is then the same as in (\ref{eq-Fa-dual-a}) where the map
$m \!\circ\! (id\otimes S)$ accounts for combining the $i$-th and $i+1$-st factors
of $\Fa {n+1}$ along the arc using the rules in (\ref{f4-1}).

The resulting diagram in $\Diag^0(\H_{\F})$ is again clearly the same as evaluating
$\Ins_{\F}\circ\decor{\H_{\F}}^0$ on a single maximum. The case for a minimum follows 
analogously from (\ref{eq-Fa-dual-a}). Thus (\ref{eq-natisoExpl}) holds on all generators
of the tangle category.

\begin{figure}[ht]
\centering
\psfrag{X}{\footnotesize $\Fa {n-1}^{-1}$}
\psfrag{Y}{\footnotesize $\Fa {n+1}$}
\psfrag{x}{\small $\xt_{\F}$}
\psfrag{m}{$\decor{\H}^{\mbox{\footnotesize$\Fa {*}^0$}}$}
\psfrag{e}{$\byeq{eq-Fa-dual-a}$}
\hspace{-6ex}\includegraphics[width=.9\textwidth]{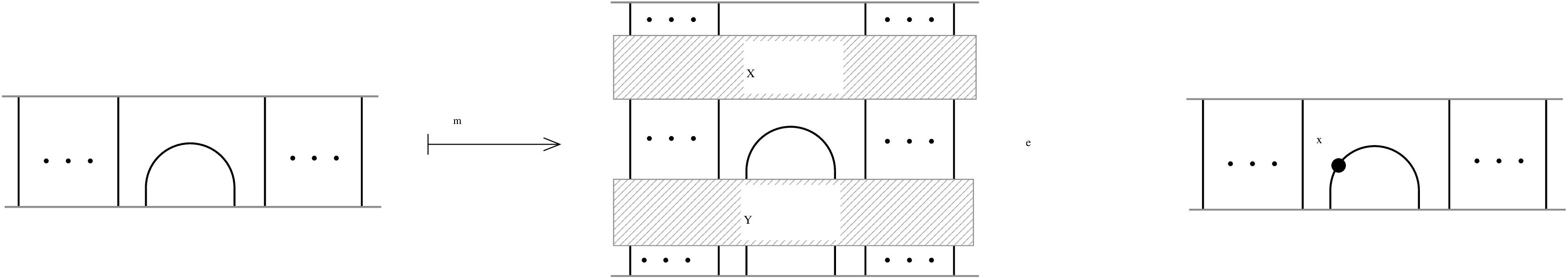}
\caption{$\decor{\H}^{\mbox{\footnotesize$\Fa {*}^0$}}$ for maximum}\label{fig-natiso-max}
\end{figure}
\end{proof}

\subsection{Gauge Transformations of Modular Evaluations}\label{ss-GaugeLin}

The relation between the evaluation functors $\eval_{\H}$ and $\eval_{\H_{\F}}$ also 
involves a natural transformation, which we defined next. Let us denote right
multiplication by $\xt_{\F}$ as
\begin{equation}\label{eq-defrho}
\rho_{\F}\,:\,\H\,\longrightarrow\,\H\,:\;z\,\mapsto\,z\xt_{\F}\,.
\end{equation}

In order to describe equivariance of a tensor power $\rho_{\F}^{\otimes g}$ that maps
$\H^{\otimes g}$ to itself (as a $\bbd$-module)
we need to identify the latter space as an $\H$-module. One $\H$-action
is given by the usual $g$-fold tensor product of adjoint action on $\H$ with
respect to the transformed coalgebra structure given by $\Delta_{\F}$ and $S_{\F}$.
We denote the respective $\H$-module by  $\H^{\otimes g}_{\F}$.

Another $\H$-action is defined similar to the adjoint action using the transformed 
coproduct $\Delta_{\F}$ but the original antipode $S$. That is, if
$\Delta^{2g-1}_{\F}(x)=\sum x^{(1)}_{\F}\otimes\ldots\otimes x^{(2g)}_{\F}\,$ then 
$x.(b_1\otimes\ldots\otimes b_g)=
\sum x^{(1)}_{\F}b_1S(x^{(2)}_{\F})\otimes\ldots\otimes x^{(2g-1)}_{\F}b_gS(x^{(2g)}_{\F})\,$. 
It is obvious that this also defines an $\H$-module which we
denote by $\widetilde\H^{\otimes g}\,$. We also readily verify that
the morphisms in 
\begin{equation}\label{eq-rhotrsf}
\rho_{\F}^*\,=\,\{\rho^{\otimes g}_{\F}:\,\H_{\F}^{\otimes g}\to \widetilde\H^{\otimes g}\}\,.
\end{equation}
indeed commute with the $\H$-action on the so defined modules. Particularly,
with $\Delta^{2g-1}_{\F}(x)=\sum x^{(1)}_{\F}\otimes\ldots\otimes x^{(2g)}_{\F}\,$ 
as before we
have 
\begin{equation}\label{eq-rhoInw}
\begin{split}
 x.(\rho^{\otimes g}_{\F}(b_1\otimes\ldots\otimes b_g))&=
x.(b_1\xt_{\F}\otimes\ldots\otimes b_g\xt_{\F})\\
&=
\sum x^{(1)}_{\F}b_1\xt_{\F}S(x^{(2)}_{\F})\otimes\ldots\otimes x^{(2g-1)}_{\F}b_g\xt_{\F}S(x^{(2g)}_{\F})\\
&
\byeq{eq-GT-S}
\sum x^{(1)}_{\F}b_1S_{\F}(x^{(2)}_{\F})\xt_{\F}\otimes\ldots\otimes x^{(2g-1)}_{\F}b_gS_{\F}(x^{(2g)}_{\F})\xt_{\F}
\\
&=
 \rho^{\otimes g}_{\F}\left({\sum ad_{\F}(x^{(1)}_{\F})(b_1)\otimes\ldots\otimes ad_{\F}(x^{(g)}_{\F})(b_g)}\right)     
\\
&=
\rho^{\otimes g}_{\F}(ad_{\F}^{\otimes g}(x)(b_1\otimes\ldots\otimes b_g))
\end{split}
\end{equation}

The maps in (\ref{eq-rhotrsf}) are thus morphisms in $\H\vDash\MM\,$. 
In the next lemma we will interpret the
collection $\rho_{\F}^*$  of these isomorphisms as a natural transformation.

\begin{lemma}\label{lm-natiso-eval}
 With conventions as above we have the following natural isomorphism of functors:
{\rm 
\begin{equation}\label{eq-natiso-eval}
 \rho_{\F}^*\,:\,\eval_{\H_{\F}}\,\natto\,\eval_{\H}\circ\Ins_{\F}\;.
\end{equation}}
\end{lemma}

\begin{proof} Expressing naturality of (\ref{eq-natiso-eval}) more explicitly we have to show for 
every class of diagrams $D: g\to h$ in $\Diag(\H_{\F})$ that
\begin{equation}\label{eq-natiso-evalexpl}
\rho_{\F}^{\otimes g}\cdot \eval_{\H_{\F}}(D)\cdot (\rho_{\F}^{\otimes h})^{-1}\,=\,
\eval_{\H}\left({\Ins_{\F}(D)}\right)\;.
\end{equation}

As before we can use functoriality to reduce the proof to generators of $\Diag(\H_{\F})$.
Moreover, all functors and natural transformations in (\ref{eq-natiso-eval}) also
respect the tensor product so that we need to verify 
(\ref{eq-natiso-evalexpl})
only for generators of $\Diag(\H_{\F})$
as a tensor category. As explained in Section~\ref{sec-tqft} all diagram morphisms are
composites of tensor products of the five types of diagrams depicted in Figures~\ref{f6-1}
and \ref{f-tt}. It thus suffices to prove (\ref{eq-natiso-evalexpl}) for each of these.

We start by evaluating $\eval_{\H}\circ\Ins_{\F}$ on the second picture in 
Figure~\ref{f6-1}. The functor $\Ins_{\F}$ inserts a $\xt_{\F}$ right above the decoration with label
$s^{\nu}_j$ which
combines to a 
single decoration $s^{\nu}_j\xt_{\F}$ in place of $s^{\nu}_j$.  
$\eval_{\H}$ assigns to this the map 
$v_j^{\nu}:\bbd\to\H\,:\,1\mapsto s^{\nu}_j\xt_{\F}$. On the other side $\eval_{\H_{\F}}$ 
assigns to the same picture the map $w_j^{\nu}:\bbd\to\H\,:\,1\mapsto s^{\nu}_j$. Clearly,
$\rho_{\F}^{\otimes 1}\circ w_j^{\nu} \circ (\rho_{\F}^{\otimes 0})^{-1}=
\rho_{\F}\circ w_j^{\nu}=v_j^{\nu}$.

Next we evaluate $\eval_{\H}\circ\Ins_{\F}$ on the third picture in the same figure.
$\Ins_{\F}$ inserts a $\xt_{\F}^{-1}$ directly below the $t_k^{\nu}$ which combines to  
$\xt_{\F}^{-1}t_k^{\nu}$ in place of $t_k^{\nu}$. $\eval_{\H}$ assigns to this the form
$l_k^{\nu}:\H\to\bbd\,:\,x\mapsto \lambda(S(x)\xt_{\F}^{-1}t_k^{\nu})\,$. If we apply 
$\eval_{\H_{\F}}$ to the same picture we obtain the form
$m_k^{\nu}:\H_{\F}\to\bbd\,:\,x\mapsto \lambda_{\F}(S_{\F}(x)t_k^{\nu})\,$. We verify (\ref{eq-natiso-evalexpl})
by computation:
\begin{equation}
 \begin{split}
  \rho_{\F}^{\otimes 0}\circ m_k^{\nu}\circ (\rho_{\F}^{\otimes 1})^{-1}(x)
&=
m_k^{\nu}(\rho_{\F}^{-1}(x))=\lambda_{\F}(S_{\F}(x\xt_{\F}^{-1})t_k^{\nu})
\byeq{eq-integralF}\lambda(\zt_{\F}^{-1}S_{\F}(x\xt_{\F}^{-1})t_k^{\nu})
\\
&\byeq{eq-GT-S}\lambda((\xt_{\F}S(\xt_{\F})^{-1})^{-1}\xt_{\F}S(x\xt_{\F}^{-1})\xt_{\F}^{-1}t_k^{\nu})
\\
&=\lambda(S(\xt_{\F})\xt_{\F}^{-1}\xt_{\F}S(\xt_{\F})^{-1}S(x)\xt_{\F}^{-1}t_k^{\nu})
\\
&=\lambda(S(x)\xt_{\F}^{-1}t_k^{\nu})=l_k^{\nu}(x)
 \end{split}
\end{equation}

The first picture is a consequence of the second and third by composing them
(setting $s_j^{\nu}=1$ and $t_k^{\nu}=r_k^{\nu}$). 

In Figure~\ref{f-tt} we start with the left
picture. 
Since no extrema are involved $\eval_{\H}\circ\Ins_{\F}$ assigns to this 
diagram the morphism $f:\H\to\H\,:\,x\mapsto axS(b)$. Similarly, 
$\eval_{\H_{\F}}$ assigns to it the map $g:\H\to\H\,:\,x\mapsto axS_{\F}(b)$.
We compute $\rho_{\F}^{\otimes 1}\circ g \circ (\rho_{\F}^{\otimes 1})^{-1}(x)
=\rho_{\F}(g(\rho_{\F}^{-1}(x)))=(a(x\xt_{\F}^{-1})S_{\F}(b))\xt_{\F}=
ax\xt_{\F}^{-1}\xt_{\F}S(b)\xt_{\F}^{-1}\xt_{\F}=
f(x)$.

Finally, $\eval_{\H}\circ\Ins_{\F}$ assigns to the picture on the right of Figure~\ref{f-tt} 
 the map $q:\H\to\H\,:\,x\mapsto \kappa^{-1}S(x)$ and $\eval_{\H_{\F}}$
assigns to it $p:\H\to\H\,:\,x\mapsto \kappa^{-1}_{\F}S_{\F}(x)\,$. As before
verification of (\ref{eq-natiso-evalexpl}) is done by computing 
$\;\rho_{\F}^{\otimes 1}\circ p \circ (\rho_{\F}^{\otimes 1})^{-1}(x)
=\rho_{\F}(p(\rho_{\F}^{-1}(x)))=\kappa^{-1}_{\F}S_{\F}(x\xt_{\F}^{-1})\xt_{\F}\byeq{eq-GT-bal}\kappa^{-1}\zt_{\F}^{-1}
\xt_{\F}S(x\xt_{\F}^{-1})=\kappa^{-1}\zt_{\F}^{-1}
\xt_{\F}S(\xt_{\F}^{-1})S(x)=q(x)\,$.
\end{proof}

\subsection{Construction of a Natural Isomorphism and Proof of Theorem~\ref{thm-GTtqft}}\label{ss-nat+THM}

In order to construct the natural transformation from Theorem~\ref{thm-GTtqft} we also
need to consider the evaluation of the transformations in Lemma~\ref{lm-natisoPsi}.
\begin{equation}\label{eq-UpsEval}
 \Fae g=\eval_{\H}(\Fa {2g})\,:\,\H^{\otimes g}\to\widetilde\H^{\otimes g}
\end{equation}
Denoting $\Fa {2g}=\sum \Fa {2g}^{(1)}\otimes\ldots\otimes\Fa {2g}^{(2g)}$
the explicit form of $\Fae{g}$ is readily derived from Figure~\ref{f-tt} as
\begin{equation}\label{eq-UpsEvalExp}
 \Fae g(b_1\otimes \ldots\otimes b_g)\,=\,\sum
\Fa {2g}^{(1)}b_1S(\Fa {2g}^{(2)})\otimes \ldots\otimes \Fa {2g}^{(2g-1)}b_1S(\Fa {2g}^{(2g)})\;.
\end{equation}
If we consider $\H^{\otimes g}$ to be a $\H$-module equipped with the tensor product of
the regular adjoint action and $\widetilde\H^{\otimes g}$ with the action defined for
(\ref{eq-rhotrsf}) above the maps $\Fae g$ also commute with the actions of $\H$. This
is follows from the following calculation.
\begin{equation}\label{eq-FaeInw}
\begin{split}
& x.\left({\Fae g(b_1\otimes \ldots\otimes b_g)}\right)=
\\
&\qquad = x.\left({\sum
\Fa {2g}^{(1)}b_1S(\Fa {2g}^{(2)})\otimes \ldots\otimes \Fa {2g}^{(2g-1)}b_gS(\Fa {2g}^{(2g)})}\right)
\\
&\qquad = {\sum
x_{\F}^{(1)}\Fa {2g}^{(1)}b_1S(\Fa {2g}^{(2)})S(x_{\F}^{(2)})
\otimes \ldots\otimes 
x_{\F}^{(2g-1)}\Fa {2g}^{(2g-1)}b_gS(\Fa {2g}^{(2g)})}S(x_{\F}^{(2g)})
\\
&\qquad = {\sum
x_{\F}^{(1)}\Fa {2g}^{(1)}b_1S(x_{\F}^{(2)}\Fa {2g}^{(2)})
\otimes \ldots\otimes 
x_{\F}^{(2g-1)}\Fa {2g}^{(2g-1)}b_gS(x_{\F}^{(2g)}\Fa {2g}^{(2g)})}
\\
&\qquad \byeq{eq-Fa-intertw} \sum
\Fa {2g}^{(1)}x^{(1)}b_1S(\Fa {2g}^{(2)}x^{(2)})
\otimes \ldots\otimes 
\Fa {2g}^{(2g-1)}x^{(2g-1)}b_gS(\Fa {2g}^{(2g)}x^{(2g)})
\\
&\qquad = \sum
\Fa {2g}^{(1)}x^{(1)}b_1S(x^{(2)})S(\Fa {2g}^{(2)})
\otimes \ldots\otimes 
\Fa {2g}^{(2g-1)}x^{(2g-1)}b_gS(x^{(2g)}) S(\Fa {2g}^{(2g)})
\\
&\qquad = \Fae g\left({\sum
x^{(1)}b_1S(x^{(2)})
\otimes \ldots\otimes 
x^{(2g-1)}b_gS(x^{(2g)}) }\right)
\\
&\qquad = \Fae g\left({ad^{\otimes g}(x)(
b_1 \otimes \ldots\otimes b_g })\right)
\\
\end{split}
\end{equation}

\medskip

\begin{proof}[Proof of Theorem \ref{thm-GTtqft}]
Applying the functor $\eval_{\H}$ to Equation~(\ref{eq-natisoPsi}) in Lemma~\ref{lm-natisoPsi} we
obtain the natural transformation
$
\eval_{\H}\Fa {*}\,:\,\eval_{\H}\circ\decor{\H}\,\natto\,\eval_{\H}\circ\Ins_{\F}\circ \decor{\H_{\F}}\,,
$
where the isomorphisms of  $\eval_{\H}\Fa {*}$ are given by the 
$\Fae g$ from (\ref{eq-UpsEval}). Similarly, (\ref{lm-natiso-eval}) implies 
a natural isomorphism
$
\rho_{\F}^*\decor{\H_{\F}}\,:\,\eval_{\H_{\F}}\circ\decor{\H_{\F}}\,\natto\,\eval_{\H}\circ\Ins_{\F}\circ\decor{\H_{\F}}\,
$ which combines to 
\begin{equation}\label{eq-GTprenatiso}
\widetilde\GTnatiso_{\F}:=
(\rho_{\F}^*\decor{\H_{\F}})^{-1}\bullet(\eval_{\H}\Fa {*})\,:\,\eval_{\H}\circ\decor{\H}\,\natto\,
\eval_{\H_{\F}}\circ\decor{\H_{\F}}\;.
\end{equation}
Given that for objects $\decor{\H_{\F}}(g)=g\,$, the morphisms for (\ref{eq-GTprenatiso}) are given by
\begin{equation}\label{eq-compnatmorph}
\bfig
\node a(100,0)[{(\widetilde\GTnatiso_{\F})_g \,:\,\H^{\otimes g}\quad}]
\node b(1030,0)[{\;\;\widetilde\H^{\otimes g}\quad}]
\node c(2000,0)[{\;\H^{\otimes g}_{\F}\;.}]
\arrow/>>->>/[a`b;\mbox{$\Fae g$}]
\arrow/>>->>/[b`c;\mbox{$(\rho^{\otimes g}_{\F})^{-1}$}]
\efig
\end{equation}
As shown in (\ref{eq-FaeInw}) and (\ref{eq-rhoInw}) above 
both $\Fae g$ and $\rho^{\otimes g}$ commute with the respective actions of
$\H$ so that also each $\widetilde\GTnatiso_{\F}(g)$ is a morphism in
$\H\vDash \MM\,$.

Substituting the functor composites in (\ref{eq-GTprenatiso}) using
 (\ref{eq-compfunct}) we thus have the natural isomorphism
$\widetilde\GTnatiso_{\F}\,:\, \PHc \circ \Surg\circ\Tmv\,\natto\,\PHFc \circ \Surg\circ\Tmv\,$.
Clearly, the functor $\Surg\circ\Tmv$ is a one-to-one correspondence on objects 
(with $\Surg\circ\Tmv(g)=g$) and maps morphism spaces surjectively onto each other. 
This ensures that
$\widetilde\GTnatiso_{\F}$ indeed gives rise to a natural isomorphism
$\GTnatiso_{\F}\,:\, \PHc \,\natto\,\PHFc \,$ defined by the same
set of morphisms given in (\ref{eq-compnatmorph}).
\end{proof}

The explicit form of the morphisms of the natural transformation is readily worked out 
from (\ref{eq-UpsEval}) and (\ref{eq-defrho}) to be
\begin{equation}\label{eq-gaugenattrsf}
\GTnatiso_{\F}(g)(b_1\otimes\ldots\otimes b_g)
\,=\,\sum\Fa {2g}^{(1)}b_1S(\Fa {2g}^{(2)})\xt^{-1}_{\F}\otimes\ldots\otimes 
\Fa {2g}^{(2g-1)}b_gS(\Fa {2g}^{(2g)})\xt^{-1}_{\F}\;.
\end{equation}

Finally, let us note that the assignment of natural transformations $\F\mapsto \GTnatiso_{\F}$
is well behaved under compositions of gauge transformations. More precisely, 
  if $\GF$ is a cocycle with respect to $\Delta_{\F}$ then $\GF\cdot\F$ is a cocycle
with respect to $\Delta$ and $(\H_{\F})_{\GF}=\H_{\GF\cdot\F}$. In this case we have
that
\begin{equation}\label{eq-natisohomo}
 \GTnatiso_{\GF\cdot\F}\,=\,\GTnatiso_{\GF}\bullet\GTnatiso_{\F}\;.
\end{equation}

\section{Integral TQFTs from Quantum Doubles}\label{s3}

In this section we specialize the previous TQFT constructions to the case in which the underlying algebra
is the Drinfeld quantum double $\H=\dh$ of a Hopf algebra $H$ over $\bbd$. The conditions that ensure 
integrality of the resulting TQFT on  $\cobn$ and $\cobc$ as postulated for the general case in 
Theorem~\ref{thm-Htqft} will reduce to only a few very mild assumptions on the algebra $H$.
Our findings will cumulate in the proof of Theorem~\ref{thm1}.

\subsection{Finite and Frobenius Hopf Algebras}\label{sec-HAovDed}
As noted in the introduction all algebras and rings are assumed to be unital
and Hopf algebras to have antipode. Following terminology used in literature 
cited here we will say that an algebra $H$ over a (unital) commutative ring $\bbd$ is {\em finite}
if it is projective and finitely generated as a $\bbd$-module. 

 Note that
for the dual
space $H^*=\mathrm{Hom}_{\bbd}(H,\bbd)$ this condition implies that $H^*$ is also 
finite, that is, projective and finitely generated. Moreover, the condition implies 
that the canonical map
\begin{equation}\label{eq-homiso}
\eta\,:\;\mathrm{Hom}_{\bbd}(H,M)\otimes_{\bbd}N\,\longrightarrow\,
\mathrm{Hom}_{\bbd}(H,M\otimes_{\bbd}N)
\end{equation}
given by $\eta(f\otimes y)(x)=f(x)\otimes y$ is an isomorphism. See, for example,
Exercise~6 on Page~155 in Section~3.10 of \cite{ja3}. Specializing (\ref{eq-homiso})
to  $M=\bbd$ and $N=H^*$  yields 
$H^*\otimes_{\bbd}H^*\cong
\mathrm{Hom}_{\bbd}(H,H^*)=\mathrm{Hom}_{\bbd}(H,\mathrm{Hom}_{\bbd}(H,\bbd))\,$. Using the adjointness relation
$\mathrm{Hom}_{\bbd}(H,\mathrm{Hom}_{\bbd}(H,\bbd))$ $\cong$ $\mathrm{Hom}_{\bbd}(H\otimes_{\bbd}H,\bbd)$
(see, for example, Proposition~3.8 in Section~3.8 of \cite{ja3}) this yields that the canonical map
\begin{equation}\label{eq-Hdualiso}
\tenhom\,:\;
\bfig
\morphism(0,0)|a|/{>}-{>>}/<750,0>[{H^*\otimes_{\bbd}H^*\;\;}`{\;(H\otimes_{\bbd} H)^*\;};] 
\efig
\text{\ \ \ with \ \ \ } \tenhom(l\otimes m)(x\otimes y)=l(x)\cdot m(y)\,
\end{equation}
is an isomorphism for $H$ as above.
The relevant implication of (\ref{eq-Hdualiso}) is that we can
define a coalgebra structure also on $H^*$ with coproduct
$\tenhom^{-1}\circ \mu^*\,:\, H^* \to H^*\otimes H^*\,$. It is readily verified that this
makes $H^*$ into a Hopf algebra
with product $\Delta^*\circ\tenhom\,:\, H^*\otimes H^*\to H^*$ and antipode $S^*:H^*\to H^*$.

The finiteness condition on $H$ also allows us to apply the {\em Dual Basis Lemma}
as in Proposition~3.11 and the following corollary in Section~3.10 of \cite{ja3}, which asserts that there is
a finite collection of pairs $(h_i,h^i)$ with $h_i\in H$ and $h^i\in H^*$ such that 
\begin{equation}\label{eq-dualbasis}
 x=\sum_ih^i(x) h_i\qquad \forall x\in H\;.
\end{equation}

In order to formulate the Frobenius condition and its equivalences we introduce notation
for spaces of integrals following, for example, \cite{kaso1}. We denote space of left integrals
as $\lint{H}=\{y\in H:\,xy=\epsilon(x)y\}$ and, analogously, the space of right integrals
$\rint{H}=\{y\in H:\,yx=\epsilon(x)y\}$ in $H$. Given that $H$ is finite, we may characterize 
the dual space as $\lint{H^*}=\{\phi\in H^*:\,(id\otimes \phi)\Delta(x)=\phi(x)\cdot 1 \;\forall \, x\in H\}$.

The Frobenius condition given in Definition~\ref{def-Frobenius}
will play a central role as well as the following definition by Pareigis that relates a
Frobenius structure to a coexisting Hopf algebra structure: 

\begin{definition}[\cite{par2}]
 A finite Hopf algebra $H$ over a commutative ring $\bbd$ is a left {\em FH-algebra} 
 if it admits a Frobenius homomorphism $\phi\in H^*$ which is also a left integral, that is,
 $\phi\in\lint{H^*}$. 
\end{definition}

Right FH-algebras are defined analogously by replacing left by right integrals.  
Next we summarize equivalent conditions for the Frobenius condition that can 
be extracted from results by Pareigis as well as Kadison and Stolin. Here we
say that a $\bbd$-submodule is a {\em $\bbd$-summand} of $H$ if it is a direct 
summand of $H$ and isomorphic to $\bbd$ as a $\bbd$-module over itself.

\begin{prop}[\cite{kaso1,par}]\label{prop-Frob-Int}
Let $H$ be a finite Hopf algebra over a commutative ring $\bbd$. 
Then  the following are equivalent:
\begin{enumerate}[label={(\roman*)}]
 \item $H$ is a Frobenius algebra.\label{it-Falg}
 \item $H$ is a left FH algebra.\label{it-FHl}
 \item $H$ is a right FH algebra.\label{it-FHr}
 \item $H^*$ is a (right or left) FH algebra.\label{it-FHdual}
 \item $\lint{H}$ $\left(\rint{H}\right)$ is a 
     direct $\bbd$-summand of $H$.\label{it-FHds}
\item $\lint{H^*}$ $\left(\rint{H^*}\right)$ is 
    a direct $\bbd$-summand of $H^*$.\label{it-FHdsd}
\end{enumerate} 
\end{prop}
\begin{proof}
As noted in Section~4 of \cite{par} Definition~\ref{def-Frobenius} is independent of
the choice of sides of the modules for finite Hopf algebras and the same Frobenius 
homomorphism that is a free generator of $H^*$ as a right $H$-module is also a free 
generator as a left $H$-module. In particular, for a given Frobenius homomorphism $\phi$ also 
$b\mapsto (b\rar \phi)=\phi(\_b)$ is also an isomorphism from $_HH$ to $_HH^*\,$.
See also Section~2 in \cite{kaso1} for an explanation of the same fact using Frobenius coordinates. 

Since the antipode is invertible (see Proposition~4 in \cite{par}) we have that 
$S^*(\phi)=\phi\circ S$ is also a Frobenius homomorphism. 
For the equivalence of \emph{\ref{it-FHl}} and \emph{\ref{it-FHr}} now follows from the observation that $\phi$ is
a left integral if and only if $S^*(\phi)$ is a right integral.

The equivalence of \emph{\ref{it-Falg}} and \emph{\ref{it-FHl}} (or \emph{\ref{it-FHr}}) is shown in Proposition~4.4 of
\cite{kaso1}, see also Proposition~3.7 in \cite{kaso9}. 
Proposition~4.3 of \cite{kaso1} ensures the equivalence of \emph{\ref{it-FHl}} (or \emph{\ref{it-FHr}}) with the
respective property stated in \emph{\ref{it-FHdual}}. 

The implication \emph{\ref{it-FHds}}$\Rightarrow$\emph{\ref{it-FHl}} is a consequence of Theorem~1 in \cite{par},
and the reverse implication follows from  Theorem~3 in \cite{par}. See also Proposition~3.1 in  \cite{kaso1}.

The fact that the invertible antipode provides a bijection between $\lint{H}$ and $\rint{H}$ shows that
the two statements in \emph{\ref{it-FHds}} are equivalent. The statement in \emph{\ref{it-FHdsd}} 
is a consequence of the equivalence of \emph{\ref{it-FHdual}} with the other statements and their 
rephrasing for $H^*$. 
\end{proof}

For later application we also introduce the following criterion. 

\begin{lemma}\label{lm-Frob-subcrit}
 Let $H$ be a finite Hopf algebra over a Dedekind domain $\bbd$.
 Suppose further that $N$ is a $\bbd$-summand of $H$ with  $N\subseteq \lint{H}$.
 
 Then $N= \lint{H}$ and, hence, $H$ is an FH algebra. 
\end{lemma}

\begin{proof}
A decomposition $H=N\oplus Q$ (given that $N$ is a $\bbd$-summand)  also implies a decomposition
 $\lint{H}=N\oplus Z$ as a $\bbd$-modules. We prove the lemma by contradiction, assuming $Z\neq 0$.
 Since $\bbd$ is a Dedekind domain it is also Noetherian as well as a Pr\"ufer ring
 (see, for example, Section~C.5. in \cite{MP02}).

 Denote by $\bbf$ the field of fractions of the domain $\bbd\,$. 
 Since $H$ is projective $H$ is also torsion-free and, hence, any 
 $\bbd$-submodule $M\subseteq H$ is also torsion-free.
 Since $\bbd$ is  Noetherian and $H$ finitely generated we have that $H$ is a 
 Noetherian module (e.g., Ch. X, \parag 1 in \cite{langA}). Thus, by definition, 
 $M$ is finitely generated. 
 
 By Theorem~C.5.5(2) in  \cite{MP02} and the fact that $\bbd$ is a Pr\"ufer ring 
 we thus have that $M$ is projective and, in particular, flat.
 Given that $\bbd \hookrightarrow \bbf$ is an injection of $\bbd$-modules  with thus 
 have that 
 $M=M\!\otimes_{\bbd}\!\bbd\,\rightarrow\,\overline M:=M\!\otimes_{\bbd}\!\bbf:\, m\mapsto m\otimes 1\,$
 is also an injection. (See, for example, Ch.XVI, \parag 2 in \cite{langA}; also Theorem~C.5.5(4) 
 in  \cite{MP02}).

 In particular, $\overline N\neq 0$ and $\overline Z\neq 0$. Moreover,  we have $\overline {\lint{H}}=
 \overline{N\oplus Z}=\overline N\oplus \overline Z\subseteq \lint{\overline H}\,$. Thus,
 $\dim(\lint{\overline H})\geq 2$ for the Hopf algebra $\overline H$ over the field $\bbf$.
 This contradicts the uniqueness results for integrals of  finite dimensional Hopf algebras 
 (over fields) as in \cite{sw69}.
\end{proof}
  
Recall that an element $x$ is said to be {\em group-like} if $\Delta(x) = x\otimes x$. 
Denote the set of group-like elements of $H$ by $G(H)$ which is clearly a group itself. 
The respective set $G(H^*)$ is, correspondingly, the
group of multiplicative forms on $H$ with values in $\bbd$.

 Following Section~3 \cite{water} we have for any finite Hopf algebra $H$ group like elements 
 $\alpha\in G(H^*)$ and $a\in G(H)$, so called {\em moduli}, with the following properties: 
\begin{eqnarray}
  \Lambda x&=\alpha(x)\Lambda \qquad \forall x\in H, \forall \Lambda\in\lint{H}\label{eq-def-modulus}\\
  (\lambda\otimes id)\Delta(x)&=\lambda(x) a \qquad\, \forall x\in H, \forall \lambda\in\lint{H^*}\rule{0mm}{7mm}\label{eq-def-comodulus}
\end{eqnarray}

The following choices of integrals with simultaneous normalizations and the existence of 
special groups-like elements are an important ingredient in the construction of TQFTs.

 \begin{prop}\label{prop-intmod-Frob}
Let  $H$ be a finite Hopf algebra over a commutative ring $\bbd$ which is Frobenius. 
Then the spaces $\lint{H}$ of left integral in $H$ 
and $\rint{H^*}$ of right integral in $H^*$ are generated freely over $\bbd$ by elements $\Lambda_l$ and 
$\lambda_r$ with the following simultaneous normalization: 
\begin{equation}\label{e3}
\lambda_r(\Lambda_l) = \lambda_r(S(\Lambda_l))=1.
\end{equation}
Moreover, there exits moduli $\alpha \in G(H^*)$ and $\lgl \in G(H)$ with the following properties:
\begin{equation}\label{eq-lm-moduli}
 \Lambda_l x = \alpha(x)\Lambda_l, \quad \text{and}\quad f \lambda_r = f(\lgl)\lambda_r, \quad \forall x\in H,\ f\in H^*.
\end{equation} 
\end{prop}

\begin{proof} By Proposition~\ref{prop-Frob-Int} we may assume that $H$ is an FH algebra and that 
$\lint{H}$ and $\lint{H^*}$ are free $\bbd$-modules. Definition~\ref{def-Frobenius} implies that
the Frobenius homomorphism $\phi=\lambda_l$ is a left integral and  
 the associated Frobenius isomorphism  implies that 
there is an element 
$\Lambda_l\in H$ such that $\Lambda_l\rar\lambda_l=\epsilon$. A standard argument (Section~3 of \cite{kaso1}
or Section~4 in \cite{par}) shows that $\Lambda_l$ needs to be a left integral. Moreover, the required
normalization follows from $1=\epsilon(\mathbf{1})=(\Lambda_l\rar\lambda_l)(\mathbf{1})=\lambda_l(\Lambda_l)\,$. 
See also arguments in Section~3 of  \cite{water} that infer this normalization.

Note, with notation as in Section~\ref{s2.2},  he have that for any $\chi\in H^*$  
\begin{equation}\label{eq-llcalc}
 \lambda_l(\Lambda_l\leftharpoonup \chi)=(\chi\cdot \lambda_l)(\Lambda_l)=\chi(\mathbf{1})\lambda_l(\Lambda_l)=\chi(\mathbf{1})\,.
\end{equation} 
 By Theorem~3 in \cite{water} and 
 Item~(iii) in Theorem~1 in \cite{water} we find that for any Frobenius Hopf algebra over a 
 general commutative ring $\bbd$ and any  $h\in H$ we have
$$
\Lambda_l\leftharpoonup (\lambda_l\leftharpoonup h)=S^{-1}(h)a\,. 
$$
Substituting $h=\Lambda_l$ and applying $\lambda_l$ on both sides we thus obtain for the left hand side
$$
\lambda_l(\Lambda_l\leftharpoonup (\lambda_l\leftharpoonup \Lambda_l))
\stackrel{\text{by (\ref{eq-llcalc})}}{=} 
(\lambda_l\leftharpoonup \Lambda_l)(\mathbf{1})=\lambda_l(\Lambda_l)=1
$$ 
and for the righ hand side, using that $a$ is group like, 
$$
\lambda_l(S^{-1}(\Lambda_l)a)=\lambda_l(S^{-1}(a^{-1}\Lambda_l))=\varepsilon(a^{-1})\lambda_l(S^{-1}(\Lambda_l))=\lambda_l(S^{-1}(\Lambda_l))\,.
$$ 
This implies the second normalization relation $\lambda_l(S^{-1}(\Lambda_l))=1$. 
Setting $\lambda_r(x)=\lambda_l(S^{-1}(x))$ for any $x\in H$ we
find that $\lambda_r\in\rint{H}$ and the respective normalizations in (\ref{e3}).

We note the first relation (\ref{eq-lm-moduli}) is identical to (\ref{eq-def-modulus}). For the second relation we note that 
$S^*$ and its inverse act as anti-automorphisms on $H^*$ which yields
\begin{eqnarray}
 f\lambda_r&=f({S^*}^{-1}(\lambda_l))={S^*}^{-1}(\lambda_l(S^*f))\stackrel{\text{by (\ref{eq-def-comodulus})}}{=}
{S^*}^{-1}(\lambda_l(S^*f)(a)) =\\
&={S^*}^{-1}(\lambda_l)f(S(a))=\lambda_rf(a^{-1})=\lambda_rf(\lgl)\,
\end{eqnarray}
where we set
\begin{equation}\label{eq-garel}
 \lgl =a^{-1}\;.
\end{equation} 
\end{proof}

In \cite{water} Waterhouse also generalizes Radford's famous formula for 
the fourth order of the antipode, which we recall here.

\begin{thm}[\cite{rad1, water}]\label{thm-S4Rad} 
Let $H$ be a finite Frobenius Hopf algebra over a commutative ring $\bbd$. 
With $\lgl\in G(H)$ and $\alpha\in G(H^*)$ as above and  any $x\in H$ we have 
\begin{equation}\label{eq-S4Rad}
 S^4(x) = \lgl(\alpha\rar x\lar\alpha^{-1})\lgl^{-1} = \alpha\rar (\lgl x\lgl^{-1})\lar\alpha^{-1}\,.
\end{equation}
\end{thm}

Note that (\ref{eq-S4Rad}) may be written more succinctly as $S^4=ad(\lgl)\circ ad^*(\alpha)$ using
notations for adjoint and coadjoint actions.
Let us also recall and summarize here the general finite order results given
in Theorem~5   of the same article:
\begin{thm}[\cite{water}]\label{thm-waterfin}
 Let $H$ be a finite Hopf algebra over a commutative ring $\bbd$. 
 Then all elements in $G(H)$ have finite order. 
\end{thm}

The ribbon or balancing property assumed in Definition~\ref{def-cobmor} applied 
to the double of a Hopf algebra $H$ is essentially equivalent to requiring that 
the moduli of $H$ admit well behaved square roots. This is formalized in the 
next definition.
\begin{definition}\label{def-doubal}
Let $H$ be a finite Hopf algebra over a commutative ring $\bbd$ with moduli $\lgl\in G(H)$
and $\alpha\in G(H^*)$ a before. Then $H$ is called {\em double balanced} if there
are group like elements $ \lbl\in G(H)$ and $\beta\in G(H^*)$ such that 
\begin{equation}\label{eq-balanced}
\beta^2=\alpha, \qquad \lbl^2=\lgl,\qquad\mbox{and}\quad S^2=ad(\lbl)\circ ad^*(\beta)\;.
\end{equation}
\end{definition} 

The moduli imply an invertible element $\alpha(\lgl)\in \bbd^{\times}$ in the units of
the ground ring and, in the case of a double balanced Hopf algebra, we define the 
respective fourth root $\theta\in\bbd^{\times}$ as 
\begin{equation}\label{eq-theta}
\theta =\beta(\lbl)\qquad \mbox{ with } \qquad \theta^4=\alpha(\lgl)\,.
\end{equation}

The finite orders of group like elements in Theorem~\ref{thm-waterfin} now imply 
the analogous statements for automorphisms and units:

\begin{cor}\label{cor-fin_ord_S+mod}
The order of the antipode $S$, the order of $\alpha(\lgl)\in\bbd^{\times}$,
as well as the order of $\theta\in \bbd^{\times}$ (in the double balanced case) are 
all finite. 
\end{cor}

Several criteria for the existence of a double balancing that require $S^2$, $\lgl$, and $\alpha$ to 
have odd orders are explored in  \cite{kr1}.

\subsection{Quantum Doubles for Projective Hopf Algebras}\label{s3.-1}

In this section we describe the construction of a quantum double $\dh$ of a Hopf algebra $H$ 
over a (unital) commutative ring $\bbd$ instead of  a field.
As usual we will denote the multiplication of $H$ by 
$\mu: H\otimes H\rightarrow H$ with unit $1\in H$, its comultiplication by $\Delta: H\rightarrow H\otimes H$ with counit 
$\epsilon:H\to \bbd$ and its antipode by $S:H\rightarrow H$, all of which are $\bbd$-module morphisms.

As a coalgebra the quantum double of $H$ is defined as $\dh\,=\,H^{*\mathrm{cop}}\otimes H\,$, where
$H^{*\mathrm{cop}}$ is 
identical to $H^*$ except that the opposite comultiplication is used.
 As an algebra
$\dh$ is defined as a bi-crossed product for which the multiplication is given as follows with notation as in
(\ref{e:co-arrows}) and (\ref{e:arrows}).
\begin{equation}\label{e7}
 (p\otimes x)(q\otimes y)  = \sum_{(q)} p q''\otimes (S^*(q')\rar x \lar q''') y\,.
\end{equation} 

Furthermore, we note that (\ref{eq-homiso}) yields in the case of $M=\bbd$ and $N=H$ an isomorphism
$\eta\,:\,H^*\otimes H\to\mathrm{End}(H)$ which shows that the dual bases can also be defined
as $\eta^{-1}(id_H)=\sum_ih^i\otimes h_i$. The existence of dual bases, thus, allows us to define
a quasi-triangular structure. Particularly, 
we can use the tensor $\sum_ih^i\otimes h_i$ to define a  canonical universal  $R$-matrix for $\dh$ by 
\begin{equation}\label{e6}
\R = \sum (\epsilon\otimes h_i)\otimes (h^i\otimes 1)\,.
\end{equation} 

These structures are next combined to defined doubles over general rings:

\begin{lemma}\label{lm-double-finite}
 Suppose $H$ is a finite Hopf algebra over a commutative ring $\bbd$. 
 
 Then $\dh$ is a finite quasi-triangular Hopf algebra over $\bbd$ with multiplication 
 as defined  in (\ref{e7}),  and $R$-matrix as in (\ref{e6}).
 \end{lemma}

 \begin{proof} The proof of the consistency bi-crossed structure as well as the 
 usual Drinfeld relations for quasi-triangularity 
 (see \cite{dr87})  for the $R$-matrix are verbatim the same as the ones
 given in, for example,  Section~IX of \cite{ka94}. Particularly, the calculations 
 there do not invoke any assumptions on the ground field and only use the 
  dual basis relation in (\ref{eq-dualbasis}). 
  
  In order to prove the existence of an antipode we use again invertibility of $S$
  as ensured by Proposition~4 of \cite{par}. 
 This allows us to define an antipode  on $\dh$
consistent with this bi-algebra structure 
by the following composition of isomorphisms:
\begin{equation}\label{eq-doubantip}
\bfig
\node a(100,0)[{\S\,:\, \dh= H^*\otimes H\;}]
\node b(1330,0)[{H^*\otimes H\;}]
\node c(2050,0)[{H\otimes H^*\;}]
\node d(2600,0)[{\dh}]
\arrow/->/[a`b;\mbox{$S^{*-1}\otimes S\;\;$}]
\arrow/->/[b`c;\mbox{$\sigma_{(12)}\;\;$}]
\arrow/->/[c`d;\mbox{$\cdot\;\;$}]
\efig
\end{equation}  
As before $\sigma_{(12)}$ denotes the transposition of  tensor factors,  and the last arrow denotes the
bi-crossed product map onto $\dh$ given by $x\otimes l\,\mapsto\,(\epsilon\otimes x)(l\otimes 1)$ using 
the specialization of (\ref{e7}) above. The verification of the antipode axiom for $\S$ as given in 
(\ref{eq-doubantip}) is exactly the same well known calculation for Hopf algebras over fields
(see again Section~IX of \cite{ka94}).

Finally, given that $H$ is finite we also have that $H^*$ is also finite. This clearly implies
that also $H\otimes H^*$ is also finitely generated as well as projective. Thus $\dh$ is also
finite over $\bbd$. 
 \end{proof}

 We conclude this section
with a couple of useful formulae assuming invertibility of $S$. The first is a
simple consequence of the fact that squares of antipodes are homomorphisms of $H$.
\begin{equation}\label{eq-sqdoubantip}
 \S^2\,=\,S^{*-2}\otimes S^2\,.
\end{equation}
Moreover, the second equivalent version of the bi-crossed product can now also be written as follows:
\begin{equation}\label{e7-1}
 (p\otimes x)(q\otimes y)  = \sum_{(x)} p(x'\rar q\lar S^{-1}(x'''))\otimes x'' y 
\end{equation}

\subsection{Ribbon Elements and Integrals for Doubles over Frobenius Hopf Algebras}\label{ss-ribbon-double} 

This section establishes the ribbon and balancing structure for Drinfeld doubles of
Frobenius Hopf algebra as constructed in the previous section. The next lemma 
recalls and generalizes the result in  \cite{kr1}. 

\begin{lemma}\label{lm-DoubBal}
Let $H$ be a finite Frobenius Hopf algebra over a commutative ring $\bbd$ with moduli
$\lgl\in G(H)$ and $\alpha\in G(H^*)$ as in (\ref{eq-lm-moduli}).

Then $\dh$ is a ribbon Hopf algebra if and only if $H$ is double balanced.

Moreover, if $\lbl\in G(H)$ and $\beta\in G(H^*)$ are the respective square roots 
of moduli as in 
Definition~\ref{def-doubal} the 
ribbon element from (\ref{eq-def-ribbon}) and balancing element from (\ref{kappa}) are 
given by
\begin{equation}\label{e5}
\r = \sum \S(f_i)e_i(\beta^{-1}\otimes \lbl^{-1})\,, \qquad\text{and}\qquad \kappa=\beta \otimes \lbl 
\end{equation}
where $\S$ is the antipode of $\dh$ and $\R = \sum e_i\otimes f_i$ is the canonical universal $R$-matrix of $\dh$.
\end{lemma}

\begin{proof} 
Then main observation is that Drinfeld's formula in Proposition~6.1 in \cite{dr89a} for
a canonical group element of a quasi-triangular Hopf algebra that implements $S^4$ in
the case of a double $\dh$ also holds when the Hopf algebra $H$ is defined over a ring
as above. In our notation and conventions it is expressed as
\begin{equation}\label{eq-DrinDoubBal}
 u S(u)^{-1} = \alpha \otimes \lgl \;\in\dh\,,
\end{equation}
where $u$ is as in (\ref{eq-def-u}) and the moduli
$\alpha$ and $\lgl$ are as in Section~\ref{sec-HAovDed}. The proof in \cite{dr89a} relies,
besides Radford's formula as given in (\ref{eq-S4Rad}) for $\bbd$, only on computations 
using Hopf algebra operations that easily extend to general commutative rings.

The first condition for a balancing element $\kappa$ as stated in (\ref{kappa}) thus becomes 
$\kappa^2=\alpha \otimes \lgl$. Note that for a Hopf algebra $A$ over $\bbd$ that is projective
and finitely generated  as a $\bbd$-module we still have $G(A^*)=\mathrm{Alg}_{\bbd}(A,\bbd)$.  
Using this and the duality properties given in Section~\ref{s3.-1} it follows that the canonical map
$G(H^*)\times G(H)\to G(\dh)$ is a group isomorphism by adapting the proof of this statement
for Hopf algebras over fields given by Radford in  Proposition~9 of \cite{rad2}. 

The existence of a group like element $\kappa$ that fulfills the first identity in (\ref{kappa}) is
thus equivalent to finding group like square roots $\beta$ and $l$ for $\alpha$ and $\lgl$ respectively
so that $\kappa=\beta\otimes \lbl\in\dh$. With this the second condition in  (\ref{kappa}) combined
with the general formula in (\ref{eq-sqdoubantip}) turns out to equivalent to the formula in
(\ref{eq-balanced}) by a straightforward calculation as given in \cite{kr1}, for which the assumption that $H$
is over a field is, again, not required.

As noted in Section~\ref{s2.2} the existence of a balancing element is equivalent to the existence 
of a ribbon element via the relation in  (\ref{eq-rel-ribbon}). In our case the latter readily translates 
into Formula~(\ref{e5}) above.
 \end{proof}
 
The previous results are now drawn together in the following proposition that summarizes the 
integral and balancing structure of the doubles in the case of finite Frobenius Hopf algebras.

\begin{prop}\label{prop-intD}
Let $H$ be a finite double balanced Frobenius Hopf algebra over a commutative ring $\bbd$.
Moreover, let $\lambda_r\in \rint{H^*}$ and $\Lambda_l\in\lint{H}$ generators  of integral spaces
with normalizations  as in (\ref{e3}). Denote also the assumed square roots of moduli as 
$\beta\in G(H^*)$ and $\lbl\in G(H)$   as in Lemma~\ref{lm-DoubBal} and the root of unity 
$\theta=\beta(\lbl)$ as in (\ref{eq-theta}). 

Then $\ld\in \dh^*$ defined by
\begin{equation}\label{eq-defdint}
\ld(f\otimes x) = \theta^{-2} f(\Lambda_l)\lambda_r(x), \quad \forall f\in H^*, \ x\in H
\end{equation}
is a generator of the ideal of right integrals of $\dh$ with
\begin{equation}\label{eq-dphase}
 \ld(\r)=\theta^3\qquad \mbox{and}  \qquad \ld(\r^{-1})=\theta^{-3}\;.
\end{equation}
Moreover, the element $\DLambda \in \dh$ defined by 
\begin{equation}\label{eq-DLambda}
 \DLambda =\theta^2S^*(\lambda_r)\otimes S(\Lambda_l)
\end{equation}
is a two-sided integral in $\dh$ which generates $\rint{\dh}=\lint{\dh}$ 
and  for which
\begin{equation}\label{eq-DInt-norm}
\ld(\DLambda) =   1\, \mbox{\ \ and \ \ } \, \S(\DLambda)=\DLambda\;.
\end{equation}
\end{prop}

\begin{proof}
In order to show that $\ld$ is a right integral for $\dh$ we need to verify that for every $f\in H^*$ and $x\in H$,
\begin{equation}\label{e8}
\sum (f\otimes x)''\ld((f\otimes x)') = (\epsilon\otimes \mathbf{1}) \ld (f\otimes x).
\end{equation}
Using that $\dh = H^{*\text{cop}}\otimes H$ as a co-algebra as well as the 
integral   conditions for $\lambda_r$ and $\Lambda_l$ we compute that
\begin{align*}
\text{LHS of (\ref{e8})} & = \sum (f'\otimes x'') \ld (f''\otimes x') 
 = \theta^{-2}\sum (f'\otimes x'') f''(\Lambda_l)\lambda_r(x') \\
& = \theta^{-2}\sum f''(\Lambda_l)f' \otimes \lambda_r(x')x''
 = \theta^{-2}\sum f''(\Lambda_l) f' \otimes \lambda_r(x)\cdot \mathbf{1}\rule{0mm}{7mm}\\
& = \theta^{-2} f(\Lambda_l)\epsilon\otimes\lambda_r(x)\cdot \mathbf{1}
 = \theta^{-2} f(\Lambda_l)\lambda_r(x)(\epsilon\otimes \mathbf{1}) \rule{0mm}{7mm}\\
&= \text{RHS of (\ref{e8})}.\rule{0mm}{7mm}
\end{align*} 
A similar calculation shows that 
\begin{equation}\label{eq-gDmod}
 g_{\mathcal D}=\alpha\otimes \lgl
\end{equation}
is the modulus for the right integral $\ld$ as in (\ref{eq-lm-moduli}). 
In order to prove that $\DLambda$ as in (\ref{eq-DLambda}) defines a integral in ${\dh}$
we  recall from Proposition~5 in
\cite{ke94} that if $\mu\in \lint{H^*}$ and $\Lambda\in \lint{H}$ are left integrals then
\begin{equation}
 \DLambda^{\flat}:=\mu\otimes S^{-1}(\Lambda)\;\in\, H^*\otimes H=\dh
\end{equation}
is a two-sided integral in $\dh$.
Although \cite{ke94} generally assumes $H$ to be over a field, the proof of Proposition~5 only
depends on calculations with Hopf algebra operations as well as the existence of moduli and dual bases,
which are given also for finite Frobenius Hopf algebras. 
The proof thus extends verbatim to finite Frobenius Hopf algebras  over rings. 

Given that $\lambda_r$  is a   {\em right} integral 
we can define a left integral as $\mu=S^{*3}(\lambda_r)$. Now by Proposition~5 in
\cite{ke94} we also have $\S( \DLambda^{\flat})= \DLambda^{\flat}$ and, using the expression
in (\ref{eq-sqdoubantip}), also $\DLambda^{\flat} =\S^2(\DLambda^{\flat})=S^*(\lambda_r)\otimes S(\Lambda_l)$. 
Clearly the multiple $\DLambda$ of this expression is then also an $\S$-invariant two-sided integral.

The normalization condition in (\ref{eq-DInt-norm}) is a straightforward calculation from 
the formulae in (\ref{eq-defdint}) and (\ref{eq-DLambda}) as 
$\ld(\DLambda)=(\theta^2\lambda_r(S(\Lambda_l))(\theta^{-2}\lambda_r(S(\Lambda_l))  =1$
using (\ref{e3}).
This in turn readily implies that both $\ld$ and $\DLambda$ are generators using arguments similar
to those in the beginning of the proof of Proposition~\ref{prop-intmod-Frob} above. 

 The formulae for the integral evaluations  in (\ref{eq-dphase}) are obtained by the following calculations. 
One can rewrite $\r$ using the canonical universal $R$-matrix $\R$ in (\ref{e6}) as follows.
\begin{align*}
\r &= \sum \S(h^i\otimes 1) (\epsilon\otimes h_i) (\beta^{-1}\otimes \lbl^{-1}) \\
&= \sum (S^{*-1}(h^i)\otimes 1) (\epsilon\otimes h_i) (\beta^{-1}\otimes \lbl^{-1}) \\
\text{by (\ref{e7})} \qquad &= \sum (S^{*-1}(h^i)\otimes h_i) (\beta^{-1}\otimes \lbl^{-1}) \\
\text{by (\ref{e7})} \qquad &= \sum S^{*-1}(h^i)\beta^{-1}\otimes(\beta\rar h_i\lar\beta^{-1})\lbl^{-1} \\
\text{by (\ref{eq-balanced})} \qquad &= \sum S^{*-1}(h^i)\beta^{-1}\otimes \lbl^{-1} S^2(h_i).
\end{align*}
Combining this expression for $\r$ with (\ref{eq-defdint}) we obtain the following calculation.
\begin{align*}
\theta^{2}\ld(\r) & = \sum (S^{*-1}(h^i)\beta^{-1})(\Lambda_l)\cdot \lambda_r(\lbl^{-1} S^2(h_i)) \\
\text{since } \lbl\in G(H) \quad & = \sum S^{*-1}(h^i)(\Lambda'_l) \beta^{-1}(\Lambda''_l) \cdot \lambda_r(S^2(\lbl^{-1} h_i)) \\
& = \sum h^i(S^{-1}(\Lambda'_l)) \beta^{-1}(\Lambda''_l) \cdot S^{*2}(\lambda_r) (\lbl^{-1}h_i) \\
&= \sum h^i(S^{-1}(\Lambda'_l)) (S^{*2}(\lambda_r)\lar \lbl^{-1})(h_i) \beta^{-1}(\Lambda''_l) \\
\text{by (\ref{eq-dualbasis})}\qquad &= \sum (S^{*2}(\lambda_r)\lar \lbl^{-1}) (S^{-1}(\Lambda'_l)) \beta^{-1}(\Lambda''_l) \\
&= \sum S^{*2}(\lambda_r)(\lbl^{-1} S^{-1}(\Lambda'_l)) \beta^{-1}(\Lambda''_l) \\
\text{since } \lbl\in G(H) \quad &= \sum S^{*2}(\lambda_r)(S^{-1} (\Lambda'_l \lbl)) \beta^{-1}(\Lambda''_l) \\
\text{since } \beta \in G(H^*) \quad &= \sum S^*(\lambda_r)(\Lambda'_l \lbl) \beta^{-1}(\Lambda''_l \lbl)\beta^{-1}(\lbl^{-1}) \\
&= (S^*(\lambda_r)\beta^{-1})(\Lambda_l \lbl) \theta \\
&= (S^*(\lambda_r)\beta^{-1})(\Lambda_l) \alpha(\lbl)\theta\\
&= S^*(\beta\lambda_r)(\Lambda_l) \alpha(\lbl)\theta\\
&= \beta(g) S^*(\lambda_r)(\Lambda_l)\beta(\lbl^2)\theta\\
&= \beta(\lbl^5)\lambda_r(S(\Lambda_l))\\
&= \theta^5\,.
\end{align*}
Hence $\ld(\r) = \theta^3$ as asserted. The equation $\ld (\r^{-1}) = \theta^{-3}$ follows
from a similar but simpler calculation starting from the following known identity for the inverse of the
ribbon element.
$$
\r^{-1} = \sum (h^i\otimes 1)\S^2(\epsilon\otimes h_i) (\beta\otimes \lbl).
$$
This concludes the proof of the proposition.
\end{proof}

Finally, combining Corollary~\ref{cor-fin_ord_S+mod} with (\ref{eq-dphase}) we observe the following:

\begin{cor}\label{cor-fin-ord_ribbon}
Let $H$ and $\bbd$ be as above and $\ld\in\dh^*$ and $\r\in\dh$ the integral and
ribbon element of $\dh$ as in Proposition~\ref{prop-intD}.

Then $\ld(\r)$ and  $\ld(\r^{-1})$ are roots of unity.
\end{cor}

\subsection{Proof of Theorem~\ref{thm1}}\label{s3.0}

This section now combines the special algebraic properties of quantum doubles
of finite Frobenius Hopf algebras as laid out in the preceding sections with the general
TQFT constructions from Section~\ref{sec-tqft}. The main observation is the
following characterization of \nicehopf doubles.

\begin{lemma}\label{lm-dh=nh}
Suppose $H$ is a double balanced finite Frobenius Hopf algebra over a (unital)  commutative ring $\bbd$. 

Then its double   $\dh$ is a \nicehopf Hopf algebra. 
\end{lemma}

\begin{proof} We will verify each of the four conditions from Definition~\ref{def-cobmor}:
The ribbon or balancing property (1) follows immediately from Lemma~\ref{lm-DoubBal}. 

In order to prove the modularity condition (2)  of Definition~\ref{def-cobmor}  we note that by standard arguments  
$\tenhom^{\#}: H^*\otimes H\to (H^*\otimes H)^*=\dh$
given by $\tenhom^{\#}(l\otimes x)(k\otimes y)=k(x)l(y)$ is an isomorphism if $H$ is finite (that is, projective and finitely generated) over $\bbd$.
We can derive an explicit expression for the composition of $\tenhom$ with the map in 
(\ref{eq-def-Omega-map}) using duality relation in (\ref{eq-dualbasis}),  the form of $\R$ in (\ref{e6}), as well as
the bi-crossed product in (\ref{e7-1}) as follows:
\begin{equation}
 \begin{split}
\overline\M (\tenhom^{\#} (l\otimes x))&=\tenhom^{\#}(l\otimes x)\otimes id(\R_{21}\R)\\
&= \tenhom^{\#} (l\otimes x)\sum_{ij}(h^j\otimes 1)(\epsilon\otimes h_i)\otimes (\epsilon\otimes h_j)(h^i\otimes 1)\\
&= \sum_{ij}\tenhom^{\#} (l\otimes x)(h^j\otimes   h_i)  (\epsilon\otimes h_j)(h^i\otimes 1)\\
&= \sum_{ij} h^j(x) l( h_i)  (\epsilon\otimes h_j)(h^i\otimes 1)\\
&=   (\epsilon\otimes x)(l\otimes 1) =  \sum_{(x)}  x'\rar l\lar S^{-1}(x''')\otimes x''   
\end{split}
\end{equation}
It is now easily verified from Hopf algebra axioms that the map
\begin{equation}
H^*\otimes H\to H^*\otimes H\,:\; k\otimes y \,\mapsto \, \sum_{(y)}  S^{-1}(y')\rar k\lar y'''\otimes y''   
\end{equation}
is a two-sided inverse for $\overline \M\circ\tenhom^{\#}$. Thus also $\overline \M$ is invertible as
required in (2) of Definition~\ref{def-cobmor}. 

The normalization required in 
 (3) of Definition~\ref{def-cobmor} is immediate from (\ref{eq-dphase})  in  Proposition~\ref{prop-intD}.
Finally,   condition (4)  is also implied by Proposition~\ref{prop-intD}  using both
parts of (\ref{eq-DInt-norm}).
\end{proof}

Part (a) of Theorem~\ref{thm1} now follows by specialization of Corollary~\ref{cor-eqvTQFT} and 
Theorem~\ref{thm-Htqft} to the case where $\H=\dh$ and using Lemma~\ref{lm-dh=nh} above.
The specialization 
to the Hennings invariant $\ph$ in Part (b) of Theorem \ref{thm1} follows from
(\ref{eq-HennInv}) and (\ref{eq-dphase}) which determine the signature
phase as $\ld(\r^{\pm 1})=\theta^{\pm 3}\,$.

Finally, to see Part (c)  we note that by (\ref{eq-eval-obj}) the space associated to a surface 
of genus $n$ (and one boundary component) is given by $\dh^{\otimes n}$. As a $\bbd$-module
this is a tensor product of copies of $H$ and $H^*$. Thus if $H$ is a free $\bbd$-module is also 
$H^*$ and hence also $\dh^{\otimes n}$.

\section{The Hennings TQFT for the quantum double of  $\Bz$}\label{s4}

In this section we illustrate the previously developed techniques, as
summarized in Theorem~\ref{thm1} and Theorem~\ref{thm-GTtqft}, with explicit
computations of all details in the case of the (na\"ively) truncated quantum 
Borel algebra $H=\Bz$ at an   $\ell$-th root of unity $\zeta$.  
This also provides a warm-up for Section~\ref{s9} where we verify the basic 
properties for doubles of Borel algebra of Lusztig's small quantum groups for
general Lie types under slightly different conventions, but omit the discussion 
of factorizations up to gauge twists given in this section.

The main results of this section include Theorem~\ref{thm-TQFT-DB=UA} below on the factorization of
the associated TQFT
and the proof of Theorem~\ref{thm2}. In Section~\ref{s9} below we will also treat the case of general 
Lie types in a more abstract fashion and using Lusztig's divided power generators for the positive 
Borel  algebra instead of the truncated algebra. 

The ground ring for the examples in this section as well as in Section~\ref{s9} is the domain of cyclotomic
integers $\bbd=\mathbb Z[\zeta]$, which may be considered either as subring of $\mathbb C$  or as the polynomial 
ring $\mathbb Z[x]$ modulo the respective cyclotomic polynomial.   
 
The first example given by Drinfeld in \cite{dr87} for 
his quantum double construction is that for the quantum universal enveloping algebra 
$U_{\hbar}\mathfrak b$ where $\mathfrak b$ is the Borel algebra associated to a simple
Lie algebra $\mathfrak g$.  He shows in Section~13 that $\mathcal D(U_{\hbar}\mathfrak b)=U_{\hbar}\mathfrak g \otimes
U\mathfrak h$, where $\mathfrak h$ is a second copy of the Cartan algebra of $\mathfrak g$.

In Proposition~\ref{lm-DBFfact} of Section~\ref{s4.2} we will establish an analogous factorization  in 
the case of the algebra $\Bz$ which is of finite rank over $\mathbb Z[\zeta]$.  The generating set of the dual 
algebra is essentially the opposite Borel part of the divided powers introduced by Lusztig in \cite{lu}.
The analogous product relation will hold na\"ively
only on the level of associate algebras but requires an additional gauge twist, as discussed in Section~\ref{s2.4},
to yield a factorization of quasi-triangular Hopf algebras. 

In Section~\ref{s:relation} this is used to establish the respective factorization of TQFTs and Hennings invariants 
and infer the formula in Theorem~\ref{thm2}. In Section~\ref{MOO} we discuss the Hennings invariant associated
to the Cartan algebra, and discuss in detail the double construction over $\Bz$ and choices of generators in
Section~\ref{s4.1}. 

Throughout this section we assume that $\ell$ is an {\em odd} integer and $\zeta$ is a primitive $\ell$-th root of unity.
In numerous calculation we will need the multiplicative inverse of 2 in $\mathbb Z/\ell$, for which we thus 
introduce the following notation:
\begin{equation}\label{e:h}
\hlk = \frac{\ell+1}2.
\end{equation}

\subsection{The MOO invariant}\label{MOO}
The MOO invariant was introduced by  Murakami,  Ohtsuki and  Okada in \cite{mh3}. 
Its construction generalizes Kirby and Melvin's formula for
the WRT $SU(2)$ invariant at the third root of unity in \cite{km}.
The construction of the MOO invariant as described in Section~7 of \cite{mh3} follows the standard WRT process 
starting from a particular ribbon Hopf algebra, which is defined as follows.

For $\ell$ an odd  integer, 
let $\A:=\A_\ell$ be the $\mathbb Q[\zeta, \sqrt{-1}]$-algebra generated by $z$ 
with the relation $z^\ell=1$. Then $\A$ is a Hopf algebra with
$$
S(z)=z^{-1},\quad \Delta(z)=z\otimes z,\quad \epsilon(z) = 1.
$$
$\A$ is endowed with a ribbon Hopf algebra structure as follows: The universal $R$-matrix is given by 
\begin{equation}\label{eq-RmatA}
 \R_\A = \frac1\ell\sum_{i,j=0}^{\ell-1} \zeta^{-2ij} z^i\otimes z^j\;,
\end{equation}
where $\zeta$ is again a primitive $\ell$-th root of unity.
We compute for the canonical element from (\ref{eq-def-u})  that
\begin{equation}\label{eq-u_A}
u_\A = \frac1\ell\sum_{i,j=0}^{\ell-1} \zeta^{-2ij}z^{i-j} = \frac{\gamma_\ell}\ell \sum_{n=0}^{\ell-1}\zeta^{\hlk \cdot n^2}z^n\,.
\end{equation}
where $\hlk$ as in (\ref{e:h}) and we denote the Gauss sum
\begin{equation}\label{eq-gamma-l}
 \gamma_\ell = \sum_{m=0}^{\ell-1}\zeta^{-\hlk\cdot m^2} =\varepsilon_{\ell}\sqrt \ell \leg p \ell\;,
\qquad\mbox{so that}\quad |\gamma_{\ell}|^2=\ell\;.
\end{equation}
Here $\varepsilon_{\ell}=1$ if $\ell\equiv 1 \mod 4$ and  $\varepsilon_{\ell}=\sqrt{-1}$ if $\ell\equiv 3 \mod 4\,$.
Moreover, $p$ is defined by $\zeta^{-\hlk}=e^{2\pi \sqrt{-1}\frac p \ell}$ and $\leg ..$ is the Jacobi symbol.
Since $S(u_\A)=u_\A$ and $S^2=id$ for the above definitions the balancing and ribbon structure for $\A$ are trivial in the sense
that
\begin{equation}\label{eq-ribbon-A}
 \r_{\A}=u_{\A} \qquad \mbox{and} \qquad \kappa_{\A}=1\;.
\end{equation}

Since $\A$ is semisimple it follows from Lemma~1 in \cite{kerler3} that $\Zz$ can also be computed using the
Hennings algorithm as described in  Section~\ref{s2.3} of this article. That is, we have 
\begin{equation}\label{eq-Z=Phi}
 \Zz\,=\,\HI{\A}\;.
\end{equation}
Let $\ia$ be the element in $\A^*$ defined by 
\begin{equation}\label{eq-lambda-A}
 \ia(z^a)=\sqrt\ell\delta_{a,0}\;.
\end{equation}
Note that (\ref{eq-gamma-l}) implies that $\sqrt\ell\in \mathbb Z[\zeta, \sqrt{-1}]\,$.
It is easy to check that $\ia$ is a right integral for $\A$ and $\ia(\r_\A)\ia(\r_\A^{-1})=1$. 

The results in \cite{mh3} imply an explicit formulae for the values of $\Zz$. If the order of the 
first homology $\h(M)$ as defined in (\ref{eq-defhM}) is coprime to $\ell$ they simplify to
Jacobi symbols as follows.

\begin{lemma}[\cite{mh3}]\label{lm-MOO=Jacobi}

Suppose $M$ is a rational homology sphere, and let $\ell$ be an odd integer with \mbox{$(\h(M),\ell)=1$}. 
Then $\Zz(M)=\leg{\h(M)}\ell$ where $\leg\cdot\cdot$ is the Jacobi symbol. 
\end{lemma}

\begin{proof} Suppose $\,\ell=p_1^{n_1}\ldots p_k^{n_k}\,$ is the prime factorization of $\ell$ with each $n_j>0\,$. It follows
by iteration of Proposition 2.3 of \cite{mh3} that
\begin{equation}\label{eq-ZZfact}
\Zz(M)\,=\,\mathscr Z_{\zeta_1}(M)\mathscr Z_{\zeta_2}(M)\ldots \mathscr Z_{\zeta_k}(M)\;,
\end{equation}
where $\zeta_j$ is a primitive $p_j^{n_j}$-th root of unity. Since $M$ is a homology sphere with $(\h(M),\ell)=1$ we have
 $p_j \ndiv \,|H_1(M,\mathbb Z)|<\infty$ so that $H_1(M,\mathbb Z)$ cannot have any $p_j$-torsion or free parts.
Consequently, $H_1(M,\mathbb Z/p_j)=H_1(M,\mathbb Z)\otimes\mathbb Z/p_j=0\,$ with $j=1,\ldots,k\,$.

Corollary 4.8 of \cite{mh3} now asserts that \smash{$\mathscr Z_{\zeta_j}(M)={\leg{\h(M)}{p_j}}^{n_j}$} where $\leg\cdot\cdot$
 is the Legendre symbol.
Combined with (\ref{eq-ZZfact}) this yields 
$$\textstyle \Zz(M)={\leg{\h(M)}{p_1}}^{n_1}\ldots{\leg{\h(M)}{p_k}}^{n_k}={\leg{\h(M)}{\ell}}\,$$ as claimed.
\end{proof}

\subsection{The Borel subalgebra and its quantum double}\label{s4.1}

The quantum double described here is the same as the one in \cite{kerler2} 
except that ours has a different ground ring.
In order to simplify notation we will use $B=\Bz$ to denote the Borel subalgebra 
of quantum $\fsl2$ at the root of unity $\zeta$.
It is defined as the $\zz$-algebra generated by $e$ and $k$ with relations
\begin{equation}\label{rb}
k^\ell=1, \qquad e^\ell=0, \quad\mbox{and}\quad  kek^{-1} = \zeta e\;.
\end{equation}
It is a Hopf algebra with structural maps:
\begin{equation}\label{crb}
\begin{split}
\Delta(k) = k\otimes k, \qquad S(k) = k^{-1}, \qquad \epsilon(k) = 1,\\
\Delta(e) = e\otimes 1 + k^2\otimes e, \qquad S(e) = -k^{-2} e, \qquad \epsilon(e) = 0.
\end{split}
\end{equation}
Obviously $B$ is a free $\zz$-module with basis $\{e^ik^j\st 0\le i, j\le \ell-1\}$. 
A left integral in $B$ is 
\begin{equation}\label{eq-sl2-Lambda}
\Lambda = (\sum_{j=0}^{\ell-1}k^j) e^{\ell-1},
\end{equation}
and a right integral $\lambda$ for $B^*$ is given by
\begin{equation}\label{eq-sl2-lambda}
\lambda(e^n k^j) = \delta_{n,\ell-1} \delta_{j,0}.
\end{equation}
These are readily checked to fulfill the normalizations in (\ref{e3}) that are required 
in Proposition~\ref{prop-intD}. The moduli $\alpha$ and $\lgl$ defined in  
Proposition~\ref{prop-intmod-Frob} are
\begin{equation}\label{eq-alpha-l}
\alpha(e^ik^j) =\delta_{i,0}\zeta^j \qquad\text{and}\qquad \lgl = k^{-2}\,.
\end{equation}
They have group like square roots
\begin{equation}\label{eq-beta-l}
\beta(e^ik^j) = \delta_{i,0}\zeta^{\frac{j(1-\ell)}2} \qquad\text{and}\qquad \lbl=k^{-1}.
\end{equation}
One can easily check that (\ref{eq-balanced}) holds in $B$ for these choices.
Proposition~\ref{prop-intD} thus implies that the quantum double $\db$ is a ribbon Hopf algebra.

We next describe this ribbon algebra in terms of generators and relations 
starting with explicit formulae for the dual algebra $B^*$. 
They will involve the so called q-number expressions in $\zz$ denoted as follows:
\begin{equation}\label{eq-def-qnums}
 [i]=\frac{\zeta^i-\zeta^{-i}}{\zeta-\zeta^{-1}}\,,\qquad 
[n]! = \prod_{i=1}^n[i]\quad\text{and}\quad 
\qchoose a b  =\frac {[a]!}{[a-b]![b]!}\;.
\end{equation}
For $k, s\in\{0,1,\ldots, \ell-1\}$ 
we define special elements $f^{(k)}$ and $\omega_s$ with in $B^*$ by 
\begin{equation}\label{eq-defdualgen}
 f^{(k)}(e^n k^j) = \delta_{n,k}\qquad \mbox{and} \qquad \omega_s(e^n k^j) = \delta_{n,0}\delta_{s,j}\;.
\end{equation}
We note that $\alpha$ and the $\omega_j$ are related by transformations
\begin{equation}\label{eq-alpha-omega}
 \alpha^k=\sum_{j=0}^{\ell-1}\zeta^{kj}\omega_j \mbox{\ \ \ and \ \ \ } \omega_j = \frac1\ell \sum_{i=0}^{\ell-1}\zeta^{-ij}\alpha^i\;,
\end{equation}
where the second relation is to be used with caution as it is strictly only defined over $\mathbb Q(\zeta)$. 
The next lemma describes the dual algebra $B^*$ and readily follows from the relations for $B$ above.
 
\begin{lemma}\label{lm-dual-basis} Let $B^*=\mathrm{Hom}_{\zz}(B,\zz)$  be the algebra over $\zz$ dual to $B$
with coproduct denoted by $\BDelta$. Then $\{f^{(n)}\omega_j\}_{0\le n, j\le \ell-1}\,$ is a basis of $B^*$ dual
to $\{e^nk^j\}_{0\le n, j\le \ell-1}$.

Moreover, $B^*$ is isomorphic to the bi-algebra over $\zz$ given by generators 
$\{f^{(n)}\}$ and $\{\omega_j\}$ subject to the relations 
\begin{equation}\label{eq-Bstar-rel}
  f^{(n)}f^{(m)}=\zeta^{mn}\qchoose{n+m}{n}f^{(n+m)}\,,  \quad
\omega_{j+2n}f^{(n)}=f^{(n)}\omega_j\,,  \mbox{\ \ \ and \ \ \ }  \omega_i\omega_j=\delta_{i,j}\omega_j 
\end{equation}
as well as  co-relations  
\begin{equation}\label{eq-Bstar-corel}
 \BDelta(f^{(n)})=\sum_{q=0}^{n}f^{(n-q)}\alpha^q\otimes f^{(q)} \mbox{\ \ \ and \ \ \ } 
\BDelta(\omega_j)=\sum_{s=0}^{\ell-1}\omega_{j-s}\otimes \omega_s\,.
\end{equation}
\end{lemma}
Note that the first relation in (\ref{eq-Bstar-rel})  implies that $f^{(n)}f^{(m)}=0$ whenever $n+m\geq \ell$.
We also imply $\sum_i\omega_i=\alpha^0=\epsilon=1$.
The evaluation of the bi-crossing formluae in 
(\ref{e7}) and (\ref{e7-1}) on these generators yields the relations in $\db$ as follows
\begin{equation}\label{eq-Bcrossrel}
\begin{split}
kf^{(n)}=\zeta^{-n}f^{(n)}k\,, \qquad &   e\omega_j = \omega_{j-2} e\,, \qquad k\omega_j=\omega_jk \\\ 
 \mbox{\ \ and \ \ } \quad ef^{(n)}=f^{(n)}e & + f^{(n-1)}(\alpha - \zeta^{-2(n-1)}k^2)\,.\rule{0mm}{8mm}
\end{split}
\end{equation}
Give that a set of free generators over $\zz$ for $\db= B^*\otimes B$ is readily given by such generators over $B$ and $B^*$
and that the above relations can be used to write any expression in terms of these we make the following observation. 
\begin{lemma}
The double  $\db$ is freely generated as a $\zz$-module by the basis 
\begin{equation}\label{eb1}
\{f^{(m)}\omega_i \otimes e^nk^j\st 0\le m,n,i,j\le \ell-1\}\,.
\end{equation}
It is, as an algebra over $\zz$, isomorphic to the algebra defined in terms of generators $\{e, k^{\pm1}, \omega_j, f^{(n)}\}_{0\leq j,n\leq l-1}\,$
and relations (\ref{rb}), (\ref{eq-Bstar-rel}), and (\ref{eq-Bcrossrel}). The coalgebra structure is given by (\ref{crb}) and 
$\BDelta^{opp}$ as in (\ref{eq-Bstar-corel}) for these generators.
\end{lemma}

From now on we will use the fact that for $f\in B^*$ and $x\in B$ we have $(f\otimes 1)(\epsilon\otimes x)=f\otimes x$ in $\db$ 
to omit the tensor symbol and simply write $fx$ for the same expression.
We next list the remaining ingredients of $\db$ relevant to the TQFT  construction. The antipode $S_{\db}$ of this
double is identical to the one in (\ref{crb}) for the generators of $B$ and on the
remaining generators it is given by 
\begin{equation}
 S_{\db}(f^{(n)})= (-1)^n \zeta^{-n(n-1)}f^{(n)}\alpha^{-n} \quad \mbox{and}\quad S_{\db}(\omega_j)=\omega_{-j}\,.
\end{equation}

The quasi-triangular structure of a double is given by the canonical universal $R$-matrix as in (\ref{e6}). For $\db$ this 
can be factored as follows.
\begin{equation}\label{e:Ruud}
\Rd = \sum_{0\le m,i\le \ell-1}  e^mk^i \otimes f^{(m)}\omega_i= \left(\sum_{m=0}^{\ell-1} e^m\otimes f^{(m)}\right) \Dd,
\end{equation}
where we denote
\begin{equation}\label{e:Dd}
\Dd = \sum_{i=0}^{\ell-1}  k^i\otimes\omega_i
\end{equation}
which is sometimes called the diagonal part of $\Rd$.  
The special group like element defined in (\ref{kappa}) is readily found from the
group like square roots given in (\ref{eq-beta-l}) to be 
\begin{equation}\label{e:kappad}
\kd \, = \,\beta \cdot \lbl \,= \, \alpha^{\hlk} k^{-1}\,,
\end{equation}
and the root of unity defined by the contraction of these elements is
\begin{equation}\label{eq-sl2theta}
\theta \, = \,\beta(\lbl) \,= \, \zeta^{-\hlk} \,.
\end{equation}
In order to determine the special integrals from Proposition~\ref{prop-intD} 
 note first that we can express the integral from (\ref{eq-sl2-lambda}) as $\lambda=f^{(\ell-1)}\omega_0$ so that
also $S^*(\lambda)=S_{\db}(\lambda)=\zeta^2f^{(\ell-1)}\omega_2$. 
Using also (\ref{eq-sl2-Lambda}) this implies  
\begin{equation}
\Lambda_{\db}=\zeta f^{(l-1)}\omega_2(\sum_{i=0}^{\ell-1}\zeta^ik^i)e^{(l-1)}=\zeta \omega_0(\sum_{i=0}^{\ell-1} k^i)f^{(\ell-1)}e^{(\ell-1)}\,.
\end{equation} 
 Combining (\ref{eq-defdint}) of Proposition~\ref{prop-intD} with (\ref{eq-sl2theta}), (\ref{eq-sl2-lambda}), (\ref{eq-sl2-Lambda}), as well as 
$$
f^{(m)}\omega_i(\Lambda)=f^{(m)}\omega_i(k^ie^{\ell-1})=f^{(m)}\omega_i(\zeta^{-i}e^{\ell-1}k^i)=\zeta^{-i}\delta_{m,\ell-1}
$$
we obtain the following formula for the normalized right integral  $\db$  of Proposition~\ref{prop-intD}.
\begin{equation}\label{e:int}
\into (f^{(m)}\omega_i e^nk^j) = \zeta^{1-i}\delta_{m,\ell-1}\delta_{n,\ell-1}\delta_{j,0},
\end{equation}
Furthermore, we compute the ribbon element from 
(\ref{e:Ruud}), (\ref{eq-def-u}), 
(\ref{e:kappad}), (\ref{eq-rel-ribbon}), and the relations of $\db$:
\begin{equation}\label{eq-sl2ribbon}
 \r_{\db} = \sum_{n,j=0}^{\ell-1}(-1)^n\zeta^{\hlk j + n(n+j+1)}f^{(n)}\omega_{-j-2n}e^nk^{j+1}
\end{equation}
The evaluation of the integral on the ribbon element is given by (\ref{eq-dphase}) in Proposition~\ref{prop-intD}
using (\ref{eq-sl2theta}) but can also be obtained by applying (\ref{e:int}) to (\ref{eq-sl2ribbon}) directly. 
\begin{equation}\label{eq-sl2-phase}
\into (\rd^{\pm 1}) = \theta^{\pm 3} = \zeta^{\pm \frac{(\ell-3)}2}
\end{equation}
 
Since $B=\Bz$ is a free module over the Dedekind domain $\bbd=\mathbb Z[\zeta]$ 
and double balanced  by (\ref{eq-beta-l}) we can apply  Theorem \ref{thm1} to construct TQFTs.

\begin{cor}\label{cor-sl2-TQFT}With $B=\Bz$ as above there is a  TQFT functor 
 \begin{equation}\label{eq-sl2-TQFT}
\PBz : \;\cobc \,\to \, \HfDcat{\db}{\zz}\;,
\end{equation}
that assigns to a surface of genus $n$ the module $\db^{\otimes n}\,$. 
Its evaluation
on the framed $S^3$ represented by \textunknot{}{-1} (or framing anomaly) is given by 
$\theta^3=\zeta^{\frac {\ell-3}2}$. In particular the associated Hennings invariant is integral in
the sense that
\begin{equation}
 \HI{\dBz}\in\zz\,.
\end{equation} 
\end{cor}

The bi-algebra $B^{*cop}\subset \db$ with $B^*$ as described in Lemma~\ref{lm-dual-basis} above
has a simpler subalgebra $B^{\dagger}$ that is generated by the group like element $\alpha$ from
(\ref{eq-alpha-l}) as well as $f=f^{(1)}$. Their powers are related to the original generators through
(\ref{eq-alpha-omega}) as well as 
\begin{equation}\label{eq-fnfn}
 f^n =\,\zeta^{\frac{n(n-1)}2} [n]! \,f^{(n)}\,.
\end{equation} 
The relations and co-relations for $\db$ imply the following for these generators:
\begin{equation}\label{e:rd} 
\begin{array}{lll}
\alpha f\alpha^{-1} = \zeta^2 f, \qquad & f^\ell = 0, \qquad & e f - f e = \alpha - k^2, \\
\alpha e\alpha^{-1} = \zeta^{-2} e, &  \alpha^\ell = 1, & \\
k f k^{-1} = \zeta^{-1} f, & \alpha k = k \alpha & 
\end{array}
\end{equation}

The relations in (\ref{e:rd}) imply that the bicrossing closes in the subalgebra so
that $\db^{\dagger}=B^{\dagger}\otimes B\subseteq B^{*cop}\otimes B=\db$ is 
indeed a subalgebra over $\zz$ with an analogous PBW type basis $\{f^m\alpha^ie^nk^j\}\,$. It 
follows that $\db$ may be equivalently defined in terms of generators $\{f,\alpha,e,k\}$ and relations (\ref{rb})
and (\ref{e:rd}). 

Note that the transformation formulae in (\ref{eq-alpha-l}) and (\ref{eq-fnfn}) also imply that the field completion
as the same, that is, $B^{\dagger}\otimes \mathbb Q(\zeta)=B\otimes \mathbb Q(\zeta)$ as well as
$\db^{\dagger}\otimes \mathbb Q(\zeta)=\db\otimes \mathbb Q(\zeta)$.
Finally, $B^{\dagger}$  inherits a well defined   Hopf algebra structure given by 
\begin{equation}\label{aaa} 
\begin{array}{ll}
\Delta(\alpha) = \alpha\otimes\alpha, & \Delta(f) = f\otimes\alpha + 1\otimes f, \\
S(\alpha)=\alpha^{-1}, &  S(f)=-f\alpha^{-1}. 
\end{array}
\end{equation}
This identifies $\db^{\dagger}$ as a Hopf algebra. Note, however, that the quasi-triangular structure will
generally not extend over $\zz$ to this subalgebra so that one has to consider the field completions.

\subsection{Rational Factorization of a Gauge Twisted $\db$}\label{s4.2}

 Theorem~\ref{thm2} is based on the fact that the Hennings invariant for $\dBz$ can be written
 as the product of the Hennings invariant for the standard quantum-$\fsl2$ and the MOO invariant, which is also
the Hennings invariants for the algebra $\A$ as discussed in Section~\ref{MOO}.

The aim of this section is to prove this fact by establishing a
respective factorization for the associated Hopf algebras. More precisely,  our goal is to identify the 
double $\db$ as a product of the standard quantum $\fsl2$ and the group algebra of the cyclic 
group of order $\ell$. This will not only imply the identities between invariants but also factorizations 
of the associated Hennings TQFTs.

The factorization of Hopf algebras will, however, not hold in the na\"\i ve sense. The first caveat is
that the ground ring needs to be extended since, in particular, $\uz$ is not naturally defined over
$\zz$. The second subtlety is that the factorization is only true up to a gauge twist transformation
of the coalgebra structure, for which we developed the general theory in 
Section~\ref{sec-gaugeequiv}.

The sought identities of invariants will not depend on these modification since an equality in an extended 
ring will obviously imply equality in the original ring and since the associated gauge transformations of TQFTs 
are trivial on genus zero surfaces.

\medskip

We begin by defining the following change of generators for  $\hdb = \db\otimes \mathbb Q(\zeta)$:
\begin{equation}\label{e:EFKefk}
E := \frac{\alpha^{-\hlk} k^{-1} e}{\zeta-\zeta^{-1}}, \quad F := -f, \quad K:= \alpha^{-\hlk} k, \quad Z := \alpha^{\hlk} k\;.
\end{equation}
The inverse relations are as follows:
\begin{equation}\label{e:efkEFK}
e = (\zeta-\zeta^{-1})ZE, \quad f = -F, \quad k = K^{\hlk}Z^{\hlk}, \quad \alpha = K^{-1}Z.
\end{equation}

The relations (\ref{rb}) and (\ref{e:rd}) for $\hdb$ are reexpressed in the new generators
by the following two sets of relations:
\begin{equation}\label{sl2Z}
Z^\ell = 1, \qquad ZK=KZ, \qquad ZE=EZ, \qquad ZF=FZ
\end{equation}
and
\begin{equation}\label{sl2EFK}
 \begin{array}{ll}
K E K^{-1} = \zeta^2 E, & K^\ell = 1, \\
K F K^{-1} = \zeta^{-2} F, \quad \quad  & (\zeta-\zeta^{-1}) (EF - FE) = K - K^{-1}\;.\rule{0mm}{7mm}  \\ 
\end{array}
\end{equation}

Note, relations (\ref{sl2Z}) imply that $Z$ generates a central subalgebra $\A\cong \mathbb Q(\zeta)[\mathbb Z/\ell\mathbb Z]$
in $\hdb$. Moreover, the relations in (\ref{sl2EFK}) show that the set $\{E,F,K\}$ generates a subalgebra $\uz$ 
isomorphic to the standard quantum $\fsl2$ over $\mathbb Q(\zeta)$. Thus as algebras we have indeed a factorization
\begin{equation}\label{e:iso}
\hdb \cong \uz\otimes\A\;.
\end{equation}

The factorization as algebras indicated above, however, does not extend to a factorization of Hopf-algebras. 
Particularly, the co-relations (\ref{crb}) and (\ref{aaa}) for $\hdb$ yield mixed terms as follows:
\begin{equation*}\label{bb}
\begin{array}{lll}
\Delta(Z)=Z\otimes Z, &\epsilon(Z) = 1, & S(Z)=Z^{-1},\\
\Delta(K)=K\otimes K, & \epsilon(K) = 1, & S(K)=K^{-1},\\
\Delta(E)=E\otimes Z^{-1} + K\otimes E, & \epsilon(E) = 0, & S(E) = -ZK^{-1}E,\\
\Delta(F)=F\otimes K^{-1}Z + 1\otimes F,  \quad  & \epsilon(F) = 0, & S(F) = -FKZ^{-1}.
\end{array}
\end{equation*}

Similarly, we can express  the canonical R-matrix from (\ref{e:Rd}) and (\ref{e:Dd}) in terms of the
new generators. The resulting expression is not a product of R-matrices but contains mixed terms of
$K$ and $Z$ generators.
\begin{equation}\label{e:Rd}
\Rd = \left(\sum_{m=0}^{\ell-1} \tau_m E^mZ^m\otimes F^m\right) \Dd,
\end{equation}
with diagonal part
\begin{equation}
\Dd=\frac 1 {\ell} \sum_{i,j=0}^{\ell-1}\zeta^{2ij}K^iZ^i\otimes K^{j}Z^{-j}
\end{equation}
and coefficients
\begin{equation}
 \tau_m=\frac{(\zeta^{-1}-\zeta)^m}{\zeta^{\frac{m(m-1)}2} [m]!}= (-1)^m \frac {\zeta^{\frac{m(1-m)}2}(\zeta-\zeta^{-1})^m}{[m]!}\,.
\end{equation}
 
It is useful to introduce the following idempotents of $\A$.
\begin{equation}\label{eq-Pidem}
P_i=\frac 1 {\ell}\sum_{i=0}^{\ell-1}\zeta^{-2ij}Z^j\qquad \mbox{so that} \quad \Delta(P_i)=\sum_{s=0}^{\ell-1}P_s\otimes P_{j-s}\,.
\end{equation}

With these conventions now we define the relevant gauge transformation.

\begin{lemma} The element $\F\in \db^{\otimes 2}$ given by 
{\rm 
\begin{equation}\label{eq-gaugeFdef}
\F=\frac 1 {\ell}\sum_{i,j=0}^{\ell-1}\zeta^{-2ij}K^i\otimes Z^j=\sum_{i=0}^{\ell-1}K^i\otimes P_i
\end{equation}}
is a gauge transformation.
\end{lemma}
\begin{proof}
It is readily computed that each side of (\ref{eq-Fcocyle}) equals $\sum_{mj}K^m\otimes K^jP_{m-j}\otimes P_j$.
Moreover, the counit condition is easily verified and the element has an inverse
\begin{equation}\label{eq-Finvers}
 \F^{-1}=\frac 1 {\ell}\sum_{i,j=0}^{\ell-1}\zeta^{2ij}K^i\otimes Z^j=\sum_{i=0}^{\ell-1}K^{-i}\otimes P_i\;.
\end{equation}
\end{proof}

The element $\F$ can thus be used to gauge transform the coalgebra structure of $\db$. In order to compute 
the coproduct and R-matrix of $\hdb_{\F}$ it is useful to record the following, easily verified identities:
\begin{equation}\label{eq-gauge-EF}
\F(E\otimes 1)\F^{-1}=E\otimes Z \qquad \mbox{and}\qquad \F(F\otimes 1)\F^{-1}=F\otimes Z^{-1}\,.
\end{equation}

Moreover, we have commutation relations 
\begin{equation}\label{eq-gauge-comm}
 [\F , 1\otimes E ]\;=\;  [\F , 1\otimes F ]\;=\;  [\F , g\otimes h]\;=\;   0  \qquad \mbox{ with } g,h\in\{K^iZ^j\}\;.
\end{equation}

Using equations (\ref{eq-gauge-EF}) and (\ref{eq-gauge-comm}) we find for the gauge transformed coproduct
$\Delta_{\F}(X)=\F\Delta(X)\F^{-1}$ by straightforward computation that
\begin{equation}\label{eq-gaugeDelta}
\begin{split}
\Delta_{\F}(E)= E\otimes 1+K\otimes E  \qquad\quad &\Delta_{\F}(K) = K\otimes K \\
\Delta_{\F}(F)= F\otimes K^{-1}+1\otimes F  \qquad\quad  &  \Delta_{\F}(Z) = Z\otimes Z \,.
\end{split}
\end{equation}

From the relations in (\ref{eq-gauge-EF}) and (\ref{eq-gauge-comm}) we also 
 compute for the gauge twisted R-matrix in the sense of (\ref{eq-Rgauge}) for this example:
\begin{equation}\label{eq-RgaugeFact}
 (\Rd)_{\F}\;=\;\F_{21}\Rd\F^{-1}\;=\;\Ru\cdot \R_\A\;.
\end{equation}
Here $\R_\A$ is as in (\ref{eq-RmatA}) with $z$ substituted by $Z$, and $\Ru$ is defined as 
\begin{equation}\label{e:Ru}
\Ru = \left(\sum_{m=0}^{\ell-1} \tau_m E^m\otimes F^m\right) \Du,
\end{equation}
where
\begin{equation}\label{e:Du}
\Du = \frac1\ell\sum_{0\le i, j \le \ell-1} \zeta^{2ij} K^i\otimes K^j\;.
\end{equation}

The main steps in the calculation for (\ref{eq-RgaugeFact}) is the identity
$\F_{21} (E^mZ^m\otimes F^m)= (E^m\otimes F^m)\F_{21}$ readily implied by  (\ref{eq-gauge-EF}) and (\ref{eq-gauge-comm}) 
as well as  $\F_{21}\F^{-1} \Dd = \Du\R_\A$, which is an exercise in resummations.
We summarize our findings as follows.

\begin{prop}\label{lm-DBFfact}
 We have the  factorization of quasi-triangular Hopf algebras given by the canonical isomorphism
{\rm 
\begin{equation}\label{eq-DBFfact}
 \faciso:\,(\hdb)_{\F} \,\longrightarrow\, \uz\otimes\A\;.
\end{equation}}
which assigns a PBW basis element $F^a  E^b K^c Z^d$ to $F^a  E^b K^c \otimes Z^d$. It is an
isomorphism of quasi-triangular algebras in the sense that
{\rm 
\begin{equation}\label{eq-DBFfact-R}
 \faciso^{\otimes 2}(\R_{\hdb_{\F}})\,=\,\faciso^{\otimes 2}((\R_{\hdb})_{\F})\,=\,(\Ru)_{13}(\R_A)_{24}\,=\,\R_{\uz\otimes\A}\;\in\;
(\uz\otimes\A)^{\otimes 2}\,,
\end{equation}}
where the R-matrix of $\uz$ is given by $\Ru$ and the one of $\A$ is given by $\R_{\A}$ as above.
\end{prop}

\begin{proof}
The factorization as associative algebras was already noted in (\ref{e:iso}).
The coproduct in (\ref{eq-gaugeDelta}) clearly restricts to coproducts on $\uz$ and $\A$. 
Moreover, $\Ru$ and $\R_{\A}$ are readily verified to be the R-matrices of these respective 
algebras and (\ref{eq-RgaugeFact}) show that their product is the R-matrix of $(\Rd)_{\F}$.
\end{proof} 

In the remainder of this section let us compute the other special elements associated to this gauge 
transformation as defined in Section~\ref{sec-gaugeequiv}. To begin with the elements defined
in (\ref{eq-def-xt}) are given by 
\begin{equation}\label{eq-xtZK}
 \xt_{\F}=\frac 1 \ell \sum_{i,j=0}^{\ell-1}\zeta^{2ji} K^iZ^j = \sum_{j=0}^{\ell-1}K^{-j}P_j 
\quad \mbox{and}\quad
\xt_{\F}^{-1}=\frac 1 \ell \sum_{i,j=0}^{\ell-1}\zeta^{-2ji} K^iZ^j = \sum_{j=0}^{\ell-1}K^{j}P_j 
\end{equation}
where we use (\ref{eq-xFrel-xinv}) as well as (\ref{eq-Finvers}). From (\ref{eq-xtZK}) we see that
$S(\xt_{\F})=\xt_{\F}$ so that
\begin{equation}\label{eq-zt=1}
 \zt_{\F}=1
\end{equation}
for the element defined in (\ref{eq-defzF}). By Lemma~\ref{lm-GT-bal} and Lemma~\ref{lm-integralF} this implies that
all elements related to the ribbon structure and integrals remain unchanged under this gauge twist. Particularly, we have
\begin{equation}\label{eq-specelemunch}
 u_{\F}=u\,,\quad \kappa_{\F}=\kappa\,,\quad \r_{\F}=\r\,,\quad \Lambda_{\F}=\Lambda\,,\quad\mbox{and}\quad \lambda_{\F}=\lambda\;.
\end{equation}

Combining (\ref{eq-specelemunch}) with Proposition~\ref{lm-DBFfact} implies that in the basis chosen in 
(\ref{e:EFKefk}) and (\ref{e:efkEFK}) the special elements of the untwisted $\hdb$ indeed factor 
accordingly into elements in or on $\uz$ and $\A$. For later use let us list here in more detail the explicit
formulae for these factorized elements, starting with the canonical element $u$ as defined in (\ref{eq-def-u}).
Particularly, we find
\begin{equation}\label{eq-dhfac-u}
\begin{split}
 u_{\dh}& =u_{\uz}u_{\A} \qquad \mbox{ with }\\
u_{\uz} & =  \left(\frac {\overline{\gamma_{\ell}}}{\ell}\sum_{n=0}^{\ell-1}\zeta^{-\hlk n^2}K^n\right)
\left(\sum_{m=0}^{\ell-1}(-1)^m\zeta^{m(m+1)}\tau_m F^mE^mK^m\right)\;.
\end{split}
\end{equation}
Here $\gamma_{\ell}$ is as in (\ref{eq-gamma-l})  and $u_{\A}$ as in   (\ref{eq-u_A}) with $z$ replaced by $Z\,$.
Note further that (\ref{eq-ribbon-A}) and  (\ref{e:kappad}) imply
\begin{equation}\label{eq-dhfac-kappa}
\kappa_{\dh}=\kappa_{\uz}=K^{-1} \qquad \mbox{and}\quad \kappa_\A=1\,.
\end{equation}
Consequently we also have 
\begin{equation}\label{eq-dhfac-ribbon}
 \r_{\dh}=\r_{\A}\r_{\uz}\quad \mbox{with}\quad \r_{\A}=u_{\A}\quad\mbox{and}\quad \r_{\uz}=u_{\uz}K\;.
\end{equation}

Also the integral from (\ref{e:int}) is reexpressed in the basis $\{F^a  E^b K^c Z^d\}$ for $\hdb$ as defined in 
(\ref{e:EFKefk}) and (\ref{e:efkEFK})  in the following factorizable form.
\begin{align}\label{e:iab}
\into (F^a  E^b K^c Z^d) &= \delta_{a,\ell-1}\delta_{b,\ell-1}\delta_{c,\ell-1}\delta_{d,0} \frac{[\ell-1]!\ell}{(\zeta-\zeta^{-1})^{\ell-1}}\nonumber\\
&=\intp(F^aE^bK^c)\cdot\ia(Z^d).
\end{align}
Here $\lambda_{\A}$ is as in (\ref{eq-lambda-A}) and $\intp$ is the normalized right integral for $\uz$ given by 
\begin{equation}\label{e:ib}
\intp(F^aE^bK^c) = \delta_{a,\ell-1}\delta_{b,\ell-1}\delta_{c,\ell-1}\frac{[\ell-1]!\sqrt\ell}{(\zeta-\zeta^{-1})^{\ell-1}}\;.
\end{equation}
As already mentioned in Section~\ref{MOO} , if $\ell\equiv 3 \mod 4$ we actually need to extend the ground ring further
to $\mathbb Q(\zeta)[\sqrt{-1}]$ in order to ensure that $\sqrt\ell$ lies in that ring so that indeed $\lambda_{\A}\in \A^*$
and $\lambda_{\uz}\in\uz^*$. This technicality can also be circumvented at the level of TQFTs by restricting to 
the index 2 subcategory of evenly 2-framed cobordisms as defined in Lemma~10 of \cite{kerler4}.

\subsection{Factorization of TQFTs and Proof of Theorem~\ref{thm2}}\label{s:relation}

\

We begin with a fairly straightforward observation about the factorization of a general
Hennings TQFT if the  underlying quasi-triangular Hopf algebra is the direct product 
of two quasi-triangular Hopf algebras. 

\begin{lemma}\label{lm-HennTQFTfact} 
 For two quasi-triangular ribbon Hopf algebras $\mathcal H$ and $\mathcal K$ satisfy the
prerequisites of Theorem~\ref{thm-Htqft} over the same domain $\bbd$  then so does  $\mathcal H\otimes_{\bbd} \mathcal K$
for the canonical product ribbon structure. Moreover, there is a canonical natural  isomorphism of TQFT functors
\begin{equation}\label{eq-HennTQFTfact} 
 \mathbf{f}:\,\tqftVc{\mathcal H\otimes\mathcal K}
\;\stackrel{\bullet} {-\!\!\!\longrightarrow}\; 
\tqftVc{\mathcal H}\otimes\tqftVc{\mathcal K}\;,
\end{equation}
given by permutation of tensor factors.
\end{lemma}

\begin{proof} The fact that the product of two \nicehopf Hopf algebras is again \nicehopf 
is obvious from the product form of the various special elements. Also tensor products of 
projective finitely generated modules are again projective and finitely generated.

For a particular genus $g$ the natural isomorphism in (\ref{eq-HennTQFTfact})  is given by the obvious
permutation of tensor factors $(\mathcal H\otimes\mathcal K)^{\otimes g}\cong (\mathcal H^{\otimes g})\otimes(\mathcal K^{\otimes g})$
with 
\begin{equation}
\mathbf{f}_g\left(
(h_1\otimes k_1)\otimes\ldots \otimes (h_g\otimes k_g)\right)\;=\;(h_1\otimes \ldots \otimes h_g)\otimes (k_1\otimes\ldots \otimes k_g).
\end{equation}

For an $\mathcal H\otimes \mathcal K$-labeled planar curve $(D,c)$ in the sense of Section~\ref{sss-CHLPC} , where $D$ has
$N$ markings we similarly identify 
$c\in (\mathcal H\otimes\mathcal K)^{\otimes N}$with an element $c'=\sum_{\mu}h_{\mu}\otimes k_{\mu}$
with $h_{\mu}\in \mathcal H^{\otimes N}$ and $k_{\mu}\in \mathcal K^{\otimes N}$. 
It follows from the constructions in Section~\ref{sss-FunctDH} that $\eval_{\mathcal  H\otimes \mathcal K}([D,c])$ is the same
morphism as $\sum_{\mu}\eval_{\mathcal  H}([D,h_{\mu}])\otimes \eval_{\mathcal K}([D,k_{\mu}])$ conjugated by the respective
permutations of tensor factors $\mathbf{f}_n$ and  $\mathbf{f}_m$. 

We note that the images of $\decor{\H\otimes\mathcal K}:\TangNM\to\Diag_{\mathcal H\otimes\mathcal K} $ have are all pure tensors in
the sense that if $\decor{\H\otimes\mathcal K}(T)=[D,c]$ for a tangle $T:n\to m$
we have that the respective element $c'=h\otimes k$ with $h\in\mathcal H^{\otimes N}$
and $k\in\mathcal K^{\otimes N}$. Moreover, we have $\decor{\mathcal H}(T)=[D,h]$ and $\decor{\mathcal K}(T)=[D,k]$. This follows from the
fact that a tangle $T$ can be broken into crossings, maxima, and minima, and  the $R$-matrix, integrals, and balancing elements 
assigned to these pictures for $\mathcal H\otimes\mathcal K$ are the pure tensors of the respective elements in the assignments for
 $\mathcal H$ and  $\mathcal K$.

Combining the properties of $\eval_{\mathcal  H\otimes \mathcal K}$ and $\decor{\mathcal  H\otimes \mathcal K}$ above we thus find 
that their composite maps a tangle $T$ to the tensor product of the morphism  $\eval_{\mathcal  H}\circ\decor{\mathcal  H}(T)$
and $\eval_{\mathcal  K}\circ\decor{\mathcal  K}(T)$ conjugated by  $\mathbf{f}_n$ and  $\mathbf{f}_m$. The respective statement
for the TQFT  functors defined on cobordisms instead of the representing tangles is  immediate.
\end{proof}

In Section~\ref{sec-gaugeequiv} we constructed the natural isomorphism between Hennings TQFTs whose underlying 
Hopf algebras are related by gauge twisting. For the particular gauge transformation $\F$ from (\ref{eq-gaugeFdef})
this isomorphism takes on a special form which we compute next.

\begin{lemma} \label{lm-FaeDA}
For the gauge transformation {\rm $\F$} as given in (\ref{eq-gaugeFdef}) the natural
transformation $\Fae{g}$ from (\ref{eq-UpsEval}) is given by factor-wise left multiplication by 
{\rm $\xt_{\F}$} as in (\ref{eq-xtZK}). That is, we have 
{\rm \begin{equation}\label{eq-FaeDA}
 \Fae{g}(a_g\otimes\ldots \otimes a_1)\,=\,\xt_{\F}a_g\otimes \ldots \otimes \xt_{\F}a_1\;.
\end{equation}}
\end{lemma}

\begin{proof}
We start by computing the special elements $\Fa n \in \dh^{\otimes n}$ defined in (\ref{eq-defFa}).
It follows by induction that 
\begin{equation}
 \Fa{n+1}\,=\,\sum_{j_1,\ldots, j_n=0}^{\ell-1}  K^{j_n}\otimes P_{j_n-j_{n-1}}K^{j_{n-1}}\otimes \ldots \otimes P_{j_2-j_1}K^{j_1}\otimes P_{j_1}
\end{equation}
where we use the iteration $\Fa{n+1}=\Y_{1}(\Fa n )$ as defined in (\ref{eq-defY}) and the identity
$\Y(K^j)=\F(K^j\otimes K^j)=\sum_{J^*}K^{j^*}\otimes P_{j^*-j}K^j$. The action of $\Fae{g}$ is now obtained
from $\Fa {2g}$ via the expression from (\ref{eq-UpsEvalExp}). We compute
\begin{equation}
\begin{split}
 \Fae{g}(a_g\otimes\ldots \otimes a_1)\,& =\,
\sum_{j_1,\ldots, j_{2g-1}}
K^{j_{2g-1}}a_gS(P_{j_{2g-1}-j_{2g-2}}K^{j_{2g-2}})\otimes \ldots \qquad \quad \\
& \qquad \qquad \ldots \otimes
P_{j_{2k}-j_{2k-1}}K^{j_{2k-1}}a_kS(P_{j_{2k-1}-j_{2k-2}}K^{j_{2k-2}}) \otimes \ldots \\
& \qquad \qquad \qquad \ldots \otimes
P_{j_{2}-j_{1}}K^{j_{1}}a_1S(P_{j_1})   \\
& =\,
\sum_{j_1,\ldots, j_{2g-1}}
P_{j_{2g-2}-j_{2g-1}} K^{j_{2g-1}}a_gK^{-j_{2g-2}}\otimes \ldots \qquad \quad \\
& \qquad \qquad \ldots \otimes
P_{j_{2k}-j_{2k-1}}P_{j_{2k-2}-j_{2k-1}} K^{j_{2k-1}}a_k K^{-j_{2k-2}} \otimes \ldots \\
& \qquad \qquad \qquad \ldots \otimes
P_{j_{2}-j_{1}}P_{-j_1} K^{j_{1}}a_1   \\
\end{split}
\end{equation}
where we used that $S(P_j)=P_{-j}$ and that the $P_j$ are central. Since $P_aP_b=\delta_{a,b}P_a$ 
we only consider terms for which $j_{2k}-j_{2k-1}=j_{2k-2}-j_{2k-1}$ and $j_{2}-j_{1}= -j_{1}\,$. 
That is,
we have contributions only when $0=j_2=\ldots = j_{2k-2} = j_{2k}=\ldots =j_{2g-2}$ are all zero.
Relabeling the remaining summation indices as $n_k=j_{2k-1}$ we obtain
\begin{equation} 
 \Fae{g}(a_g\otimes\ldots \otimes a_1)\,   =\,
\sum_{n_1,\ldots, n_g}
P_{-n_g} K^{n_g}a_g \otimes \ldots  \otimes P_{-n_k} K^{n_k}a_k \otimes \ldots  \otimes P_{-n_1} K^{n_1}a_1
\end{equation}
Clearly, the summation can now be distributed over the factors to yield the desired form (\ref{eq-FaeDA})
using the expression in (\ref{eq-xtZK}) for $\xt_{\F}$.
\end{proof}

We note that the natural isomorphism is $\GTnatiso_{\F}$ from Theorem \ref{thm-GTtqft} as a map on $\H^{\otimes g}$ 
is by (\ref{eq-compnatmorph}) actually the composite of $\Fae{g}$ as computed above and 
 $(\rho^{\otimes g}_{\F})^{-1}$ which multiplies $\xt_{\F}^{-1}$ from the right to each tensor factor, see also (\ref{eq-gaugenattrsf}).
Thus, for the gauge transformation from  (\ref{eq-gaugeFdef}) the natural isomorphism $\GTnatiso_{\F}:\tqftVc{\hdb}\overset{\bullet}{\to}\tqftVc{\hdb_{\F}}$ 
of TQFTs is given by
\begin{equation}\label{eq-gaugeiso-chi}
(\GTnatiso_{\F})_g\,=\,\chi_{\F}^{\otimes g} \qquad\mbox{ where }\quad  \chi_{\F}(x)=\xt_{\F}x\xt_{\F}^{-1}\,.
\end{equation}
The latter inner automorphism can be reexpressed by multiplication of an element by $Z^{-d}$ 
where $d$ is the $K$-degree of the element. More precisely, the following can be verified by a straightforward computation. 
\begin{lemma}\label{lm-chi-formula}
 Suppose for $X\in\db$ we have $KXK^{-1}=\zeta^{2d}$. Then 
{\rm 
\begin{equation}\label{eq-chi-formula}
 \chi_{\F}(X)=XZ^{-d}\,.
\end{equation}}
\end{lemma}
For example, we have $\chi_{\F}(E)=EZ^{-1}$, $\chi_{\F}(F)=FZ$, and that $\chi_{\F}$ is identity on $K$ and $Z$. 
In applications it is useful to think of (\ref{eq-chi-formula}) as the definition of $\chi_{\F}$. We denote the composite 
with the isomorphism from (\ref{eq-DBFfact}) as follows:
\begin{equation}\label{eq-defcombchi}
\begin{split} 
 \widehat\chi_{\F}\,:&\,\hdb \;\stackrel{\chi_{\F}}{\longrightarrow}\;\hdb_{\F} \;\stackrel{\faciso}{\longrightarrow}\;\uz\otimes \A\\
\; :&\;F^aE^bK^cZ^d\quad\mapsto\quad F^aE^bK^c\otimes Z^{d+a-b}\;.\rule{0mm}{7mm}
\end{split}
\end{equation}
Finally we note that there is an obvious natural isomorphism 
$\facisoNat:\tqftVc{\hdb}\stackrel{\bullet}{\to}\tqftVc{\uz\otimes \A}$ given by 
$\facisoNat_g=\faciso^{\otimes g}$. Composing the natural isomorphisms $\GTnatiso_{\F}$ from (\ref{eq-gaugeiso-chi}) , $\facisoNat$,
and $\mathbf{f}$ described in Lemma~\ref{lm-HennTQFTfact}  we obtain the following isomorphism of TQFTs:
 
\begin{thm}\label{thm-TQFT-DB=UA}

Let $l\geq 3$ be an odd integer as before. If $\ell\equiv 1\mod 4$ 
or if we restrict TQFT-functors to evenly framed cobordisms
assume $\bbd=\mathbb Q(\zeta)$. If $\ell\equiv 3\mod 4$ and
all framings of cobordisms are included let $\bbd=\mathbb Q(\zeta)[\sqrt{-1}]\,$.

Then there is a natural isomorphism of TQFT functors over $\bbd$ given by 
{\rm \begin{flalign}
 \;  &&\VcDBUA\,&:\; \tqftVc{\hdb}\,\stackrel{\bullet \cong}{-\!\!\!-\!\!\!-\!\!\!\longrightarrow}\,\tqftVc{\uz}\otimes\tqftVc{\A}&&\\
\mbox{with}&&\VcDBUA_g\,=\,\mathbf{f}_g\circ (\widehat\chi_{\F})^{\otimes g}\,&:
\;\hdb^{\otimes g}\,\longrightarrow\,(\uz)^{\otimes g}\otimes \A^{\otimes g}\;.&&\nonumber\rule[-1mm]{0mm}{8mm}
\end{flalign}}
\end{thm}

As an application of Theorem~\ref{thm-TQFT-DB=UA} 
we consider the special case $g=0$ which leads to the respective invariants for closed 3-manifolds
as described in  Section~\ref{s2.3}. In all cases the  associated module is free of rank one. It is obvious that in this
case $\VcDBUA_0\,=\,id$ so that we have strict equality $\tqftVc{\hdb}(M^*)=\tqftVc{\uz}(M^*)\tqftVc{\A}(M^*)$
for $M^*:D^2\to D^2$ as in  Section~\ref{s2.3}. It follows immediately from the factorizations
in (\ref{eq-dhfac-ribbon}) and (\ref{e:iab}) 
that the extra factor $\lambda(\r)$ occurring in (\ref{eq-HennInv}) also factors as
\begin{equation}\label{eq-factor-factor}
 \into(\r_{\dh})\,=\,\intp(\r_{\uz})\ia(\r_{\A})\;.
\end{equation}

Combining our results with the definitions in (\ref{eq-HennInv}) we thus find the following factorization of the associated
Hennings invariants. 

\begin{prop}\label{pp}
Let $M$ be a closed oriented 3-manifold and $\zeta$ be a root of unity of odd order.
Then
\begin{equation}\label{pp1} 
\HI{\dBz}(M)=\HI{\uz}(M)\HI{\A}(M)\;.
\end{equation}
\end{prop}

\begin{proof}
Following the description of the construction of Hennings invariants in Section~\ref{s2.3}
we first construct a cobordism $M^*:D^2\to D^2$ for $M$ (with choice of some even framing).
We note that in the case of $g=0$ the associated modules are all free of rank one. 
Thus, the isomorphism 
$$
\mathrm{Ad}(\VcDBUA_0)\,:\;\mathrm{End}(\tqftVc{\db}(D^2))
\,\to\,  
\mathrm{End}(\tqftVc{\uz}(D^2))\otimes \mathrm{End}(\tqftVc{\A}(D^2))
$$
is between rank one spaces. Indeed it is the canonical one mapping identities to each other.
Since, by Theorem~\ref{thm-TQFT-DB=UA} , it also maps the values of $M^*$ assigned
by the TQFTs to each other it follows from (\ref{eq-Henn-idnorm})
that $\widetilde{\HI{\db}}(M^*)=\widetilde{\HI{\uz}}(M^*)\widetilde{\HI{\A}}(M^*)\,$.

Given (\ref{eq-factor-factor}) above we see thus that 
 all terms for $\db$ in (\ref{eq-HennInv}) factor into the respective terms for 
$\uz$ and $\A$, implying (\ref{pp1}).
\end{proof}

Note that by choosing even framings for $M^*$ we avoid working over ring
extensions by $\sqrt{-1}$ and all invariants are in $\mathbb Q(\zeta)$. In fact,
other results   imply that they have values in $\mathbb Z[\zeta]$.
The remaining ingredient in the proof of Theorem~\ref{thm2} is the following relation
established and proved in \cite{cks} for a closed, connected, compact, oriented 3-manifold $M$.  
\begin{equation}\label{wrth}
\HI{\uz}(M) = \h(M) \tz(M).
\end{equation}
Here $\tz(M)$ denotes the  Witten-Reshetikhin-Turaev $SO(3)$ invariant associated to the root of
unity $\zeta\,$, and $\h(M)$ is the order of first homology group as defined in (\ref{eq-defhM}). 

Combining (\ref{wrth}), (\ref{pp1}), and  (\ref{eq-Z=Phi}) we infer (\ref{e1}) and thus Theorem~\ref{thm2}.

\section{Hennings TQFTs for General Lie Types $\mathfrak g$}\label{s9}

Although we chose to focus in the previous section on the rank one case for the purpose of 
illustration of the general construction, many of the results there generalize without too 
much difficulty to higher rank Lie algebras. In this section we will review the basic 
construction of the ``small'' quantum as defined by Lusztig in 
\cite{lu90b} (see also \cite{lu90a}) and provide the relations and ingredients
required for the respective TQFT constructions.

\subsection{Lusztig's small quantum Borel algebra $\mathbf u^{0,+}$}\label{ss-LusBor}

For a simple Lie algebra $\mathfrak g$ of rank $n$ let $\roots$ be a root system and denote 
by $\roots^+$ the set of positive roots and by
$\sroots=\{\alpha_i: 1\leq i\leq n\}$ the set of simple roots. The ordering of simple roots is 
as in  $\parag 4.1$ of \cite{lu90b} (also $\parag 2.1$ of \cite{lu90a}).  
Further let $\rootl$ be the associated root lattice with the standard (coefficient wise) 
partial order $\preccurlyeq$ with respect to $\sroots$ as a lattice basis.

Denote further by $a_{ij}=\langle \alpha_i,\alpha_j\rangle=2\frac {(\alpha_i,\alpha_j)}{(\alpha_i,\alpha_i)}$
the Cartan matrix of $\roots$ as well as the integers $d_i=\frac 12(\alpha_i,\alpha_i)\in\{1,2,3\}$. 
As before set $[n]_{q}=\frac {q^n-q^{-n}}{q-q^{-1}}$,
$[n]_q!=[n]_q\cdot [n-1]_q\cdot\ldots [2]_q\cdot [1]_q$, and $\qchoose{n}{m}_q=\frac{[n]_q!}{[n-m]_q[m]_q}$.
Following $\parag 1$ in \cite{lu90b} and using $q_i=q^{d_i}$ let $\mathbf U^{0,+}\cong\mathbf U^0\otimes \mathbf U^+$ 
be the algebra over 
$\mathbb Q(q)$ with generators $\{E_i, K_i: i=1,\ldots, n\}$ and relations
\begin{equation}\label{eq-uqb-rels}
 \begin{split}
  K_iK_j=K_jK_i\,, \qquad K_i^{-1}K_i=K_iK_i^{-1}=1\,, \qquad K_iE_jK_i^{-1}=q^{d_ia_{ij}}E_j\,,\\
  \mbox{and \ \ }\sum_{r+s=1-a_{ij}}(-1)^s\qchoose{1-a_{ij}}{s}_{q_i}E_i^rE_jE^s_i=0 \mbox{\ \ for \ \ }i\neq j\,.
 \end{split}
\end{equation} 
Moreover, a coproduct on $\mathbf U^{0,+}$ is defined by 
\begin{equation}\label{eq-uqb-corels}
 \Delta(K_i)=K_i\otimes K_i\,, \qquad \mbox{and} \qquad \Delta(E_i)=E_i\otimes 1 + K_i\otimes E_i\,.
\end{equation}

Lusztig introduces the divided power generators $E^{(n)}_i=([n]_{q_i}!)^{-1}E^n_i$ and we considers 
analogous to $\parag 1.3$ of \cite{lu90b} the subalgebra $U^{0,+}$ of $\mathbf U^{0,+}$ generated
by the $E_i^{(n)}$, $K_i^{\pm 1}$, and additional elements ${\qchoose {K_j;k}n}\in \mathbf U^0$ 
as an algebra over the ground ring $\mathbb Z[q,q^{-1}]$.

In order to explain the dependence on the integers $d_i$, we use in this paragraph the more explicit 
notation $\mathbf U^{0,+}=\mathbf U(q,\underline d)$ to indicate the dependence on the indeterminant 
$q$ and on the choice 
of the tuple $\underline d=(d_1,\ldots,d_n)$ of these integers. For a number $c\in\mathbb N$ consider 
the ground ring extension 
$\mathbf U(q^c,\underline d):=\mathbf U(\overline q,\underline d)\otimes_{\overline q=q^c}\mathbb Q(q)$.
The structural statements in \cite{lu90b} and the following sections for $\mathbf U^{0,+}$  thus readily
extend to $\mathbf U(q^c,\underline d)$ as well. At the same time a presentation of $\mathbf U(q^c,\underline d)$
over $\mathbb Q(q)$
is obtained by substituting $q$ by $\overline q=q^c$ in the relations in  (\ref{eq-uqb-rels}).
Considering the parameters appearing in these relations we note that $\overline q^{d_ia_{ij}}= q^{(cd_i)a_{ij}}$
and $\overline q_i=\overline q^{d_i}=q^{(cd_i)}$. The same relations are thus obtained by substituting each $d_i$
by $cd_i$ so that $\mathbf U(q^c,\underline d)\cong \mathbf U(q,c\underline d)$.

Instead of a ground ring extension we may thus, equivalently, consider the $d_i$ as additional integer 
variables (not necessarily in $\{1,2,3\}$) subject only to the constraint that $d_ia_{ij}$ is symmetric. 
The added degree of freedom of rescaling $d_i\mapsto cd_i$ will be useful to ensure balancing structures
in some case. In the language
of root systems the rescaling can also be understood as a rescaling of inner form $(\,,\,)\mapsto c(\,,\,)$ 
that is representing the same Cartan data.

For an element $\beta=\sum_ib_i\alpha_i\in\rootl$ denote by $K^{\beta}=K_1^{b_1}\ldots K_n^{b_n}$.  
Assign to a monomial $X=K^{\beta}E_{i_1}\ldots E_{i_N}$ a root lattice valued degree $\ndeg(X)\in \rootl$ given
by $\ndeg(X)=\sum_k\alpha_{i_k}=\sum_ia_i\alpha_i$ where $a_i=|\{k:i_k=i\}|$. Since 
the relations and co-relations in 
(\ref{eq-uqb-rels}) and (\ref{eq-uqb-corels}) are homogeneous with respect to this degree function
we have thus a graded Hopf algebra
$\mathbf U^{0,+}=\bigoplus_{\nu}\mathbf U_{\nu}^{0,+}$ where the direct sum runs over elements 
$\nu\in \rootl^{\succcurlyeq 0}$ with non-negative coefficients. The grading clearly extends with
$\ndeg(F_i)=-\alpha_i$ to the entire
quantum group $\mathbf U$ as noted in $\parag 1.6$ of \cite{lu90b} as well as to the subalgebras 
$U$ or $U^{0,+}$ over $\mathbb Z[q,q^{-1}]$.

Considering the algebra autormophisms $T_i$ on the full quantum group
introduced in Theorem~3.1 of \cite{lu90b} and the form of 
the elementary Weyl reflection $s_i$ at $\alpha_i$ we notice that
\begin{equation}
 \ndeg(T_i(X))=\ndeg(X)-\langle\alpha_i,\ndeg(X)\rangle\alpha_i=s_i\left(\ndeg(X)\right)\,.
\end{equation}
holds on all generators and thus on the entire algebra. Similarly, the extension of the 
automorphism to general elements of the Weyl group $W$ from Theorem~3.1 of \cite{lu90b}
obey $\ndeg(T_w(X))=w(\ndeg(X))$ for all elements $w\in W$. In Section $\parag 4.1$ of \cite{lu90b}
elements $E^{(N)}_\beta=T_{w_{\beta}}(E_{i_{\beta}}^{(N)})$ are introduced with
$\beta=w_{\beta}(\alpha_{i_{\beta}})$, so that 
$
 \ndeg(E^{(N)}_\beta)=\ndeg(T_{w_{\beta}}(E_{i_{\beta}}^{N}))=\ndeg(T_{w_{\beta}}(E_{i_{\beta}})^{N})=N\cdot 
 \ndeg(T_{w_{\beta}}(E_{i_{\beta}}))=N\cdot w_{\beta}(\ndeg(E_{i_{\beta}}))=
 N\cdot w_{\beta}(\alpha_{i_{\beta}})=N\cdot\beta$.
 For a function $\psi:\Phi^{+}\to\mathbb N\cup\{0\}$ and using the ordering on $\Phi^+$ we thus have that for
 \begin{equation}
  E^{\psi}:=\prod_{\alpha\in \Phi^+}E^{(\psi(\alpha))}_{\alpha} \qquad \mbox{we have}\qquad 
  \widehat\psi:=\ndeg(E^{\psi})=\sum_{\alpha\in\Phi^+}\psi(\alpha)\alpha\,\in\rootl\;,
 \end{equation}
considering the $\alpha$ as combinations of $\alpha_i$ in the root lattice $\rootl$. Extending
Assertion~B in \cite{lu90b} we have the following property of the coproduct. 
\begin{lemma}\label{lm-copEpsi}Assume notation as above. Then 
\begin{equation} 
 \Delta(E^{\psi})=E^{\psi}\otimes 1 + K^{\widehat\psi}\otimes E^{\psi}+ M
\end{equation}
where $M$ is a sum of terms with bigrading $(\nu,\mu)$ where $\nu,\mu\neq 0$ and $\nu,\mu\prec \hat\psi$.
\end{lemma}
\begin{proof}
We show more generally that for $X\in \mathbf U^+_{\eta}$ with $\eta\in \rootl^{\succcurlyeq 0}$ we have 
$\Delta(X)=X\otimes 1+K^{\eta}\otimes X+M$ where $M$ consists of terms with bigrading $(\nu,\mu)$ such
that $\nu+\mu=\eta$ but $\nu,\mu\neq 0$ (and hence $\nu,\mu\prec \eta$). To this end note that any such
$X$ is a linear combination over $\mathbb Q(q)$ of monomials $E_{i_1}\ldots E_{i_p}$  
with $\sum_k\alpha_{i_k}=\eta$. By linearity of the statement it is thus enough to proof it for such 
monomials. Considering the form of the coproduct we see that the $(\eta,0)$ and $(0,\eta)$ graded 
components of $\Delta(E_{i_1}\ldots E_{i_p})=\prod_k(E_{i_k}\otimes 1+K_{i_k}\otimes E_{i_k})$ are
$\prod_k(E_{i_k}\otimes 1)=(E_{i_1}\ldots E_{i_k})\otimes 1$ and 
$\prod_k(K_{i_k}\otimes E_{i_k})=(K_{i_1}\ldots K_{i_k})\otimes 
(E_{i_1}\ldots E_{i_k})=K^{\eta}\otimes (E_{i_1}\ldots E_{i_k})$, which is what we needed to show.
\end{proof}

 In Section $\parag 8$ of  \cite{lu90b} the algebra $U^{0,+}$ the indeterminant $q$ is 
 specialized to an $\ell$-th primitive root of unity $\zeta$ for any $\ell\in\mathbb N$, yielding an algebra
 $U^{0,+}_{\zeta}=U^{0,+}\otimes_{q=\zeta} \mathbb Z[\zeta]$ with ground ring $\mathbb Z[\zeta]$. 
 Denote by $\ell_i$ the 
 order of $\zeta^{2d_i}$ (i.e., $\ell =gcd(\ell,2d_i)\cdot\ell_i$) and $\ell_{\alpha}=\ell_i$
 if $\alpha$ is in the same $W$-orbit as $\alpha_i$ (that is, $(\alpha,\alpha)=(\alpha_i,\alpha_i)=2d_i$).
 The specialization now implies relations $E_{\alpha}^{(i)}E_{\alpha}^{(j)}=0$ for $i,j<\ell_{\alpha}$ but
 $i+j\geq \ell_{\alpha}$
 as well as $K_i^{2\ell_i}=1$. 
 
 As in \cite{lu90b}  $\mathbf u^+$ and $\mathbf u^0$ are the subalgebras $U^{0,+}_{\zeta}$ generated
 by the sets of elements $\{E_{\alpha}^{(N)}\,:\,\alpha\in\Phi^+,\,0\leq N<\ell_{\alpha}\}$ and 
 $\{K_i^{\pm 1}\,:\,i=1,\ldots,n\}$
 respectively. Let $\rootl_{\ell}\subseteq\rootl $ be the sublattice generated by $\{2\ell_i\alpha_i\}$ and write
 $\overline \rootl =\rootl/\rootl_{\ell}$ for the respective finite abelian group. Denoting the restricted 
 exponent set $\mathcal X_{\ell}=\{\psi:\Phi^+\to \mathbb N\cup\{0\} | \,0\leq \psi(\alpha)<\ell_{\alpha}\}$
 Theorem~8.3 of \cite{lu90b} states that
\begin{equation}\label{eq-PBW-zeta}
\{ E^{\psi}\,:\;\psi\in \mathcal X_{\ell}\}\qquad\mbox{\ and \ } \qquad
\{ K^{\beta}\,:\;\beta\in  \overline \rootl\}
\end{equation}
are PBW bases  of $\mathbf u^+$ and $\mathbf u^0$ over $\mathbb Z[\zeta]$ respectively. The grading $\ndeg$ is inherited by 
these subalgebra 
so that $\mathbf u^+=\bigoplus_{\nu}\mathbf u^+_{\nu}$ and $\ndeg(\mathbf u^0)=0$. The basis 
 (\ref{eq-PBW-zeta}) implies now that for
\begin{equation}\label{eq-gradBz}
 \ttau=\sum_{\alpha\in \Phi^+}(\ell_{\alpha}-1)\alpha\;\in\rootl\,,\mbox{\ \ \ \ we have \ \ }
 \mathbf u^+_{\nu}\neq 0 \;\Rightarrow\;0\preccurlyeq\nu\preccurlyeq\ttau\,.
\end{equation}
Let $\psi_{\tops}\in \mathcal X_{\ell}$ be given by $\psi_{\tops}(\alpha)=\ell_{\alpha}-1$ for all $\alpha\in\Phi^+$
and denote
$E^{\tops}=E^{\psi_{\tops}}\neq 0$. The basis in (\ref{eq-PBW-zeta}) and the grading restriction in (\ref{eq-gradBz})
now imply
\begin{equation}\label{eq-topEprop}
 \mathbf u^+_{\ttau}=\mathbb Z[\zeta]E^{\tops}\qquad \mbox{and}\qquad E_iE^{\tops}=0=E^{\tops}E_i\,
\end{equation}
since $\ndeg(E_iE^{\tops})=\ttau+\alpha_i\not\preccurlyeq\ttau$.

Lusztig's small quantum Borel algebra $\mathbf u^{0,+}$ is now the subalgebra of $U^{0,+}_{\zeta}$ generated
by the $E_{\alpha}^{(N)}$ and $K_i^{\pm 1}$ as above so that $\mathbf u^{0,+}=\mathbf u^0\mathbf u^+$. The graded
components of the induced 
grading are given by   $\mathbf u^{0,+}_{\nu}=\mathbf u^0\mathbf u^+_{\nu}$ with $\nu\in \rootl$.

The extension of the coproduct in (\ref{eq-uqb-corels}) to elements $E^{(N)}_i$ as in equation (a) of 
$\parag 1.3$ of \cite{lu90b} implies that $\mathbf u^{0,+}$ is indeed a $\rootl$-graded Hopf algebra over $\mathbb Z[\zeta]$
with $\mathbf u^{0,+}_{\nu}=0$ if $\nu\not\preccurlyeq\ttau$. Following Theorem~8.3 of \cite{lu90b} and (\ref{eq-PBW-zeta}) it has a
PBW basis $\{K^{\beta}E^{\psi}:\beta\in  \overline \rootl\,,\,\psi\in\mathcal X_{\ell}\}$ so that 
$\mathbf u^{0,+}\cong\mathbf u^0\otimes\mathbf u^+$ as a free $\mathbb Z[\zeta]$-module. 
Moreover, 
Lemma~\ref{lm-copEpsi} also holds as a co-relation in $\mathbf u^{0,+}$ and implies that
\begin{equation}\label{eq-coprodgrad}
\Delta(K^{\beta}E^{\psi})=K^{\beta}E^{\psi}\otimes K^{\beta} + K^{\beta+\widehat\psi}\otimes K^{\beta}E^{\psi}+ M
\end{equation}
with $M$ of mixed grading as before. 

In the following section it will be useful to view the small Borel algebra as a semi-direct product
$\mathbf u^{0,+}\cong\mathbf u^0\ltimes\mathbf u^+\,$. Particularly, note that 
$\mathbf u^0\cong\mathbb Z[\zeta][\overline \rootl]$. The structure of the semi-direct product
is thus given by a homomorphism $\crass:\,\overline Q\to \mathrm{HAut}(\mathbf u^+):\,\beta\mapsto \crass_{\beta}\,$,
where $\mathrm{HAut}$ denotes the group of Hopf-algebra automorphisms. Given that all elements of $\overline Q$
are group-like this also implies a consistent Hopf-algebra structure for the semi-direct product.
The homomorphism is 
defined on an 
element $X\in \mathbf u^{+}_{\nu}$ (in fact, more generally, for $X\in \mathbf u^{0,+}_{\nu}$) by 
\begin{equation}\label{eq-XKgradrel}
\crass_{\beta}(X)= K^{\beta}XK^{-\beta}=\zeta^{(\beta,\nu)}X\,.
\end{equation}
The latter equality follows readily by induction, noting for the exponent in (\ref{eq-uqb-rels})  that 
$d_ia_{ij}=\frac 12 (\alpha_i,\alpha_i)\langle \alpha_i,\alpha_j\rangle=(\alpha_i,\alpha_j)$.

\subsection{The extended small quantum Borel algebra $\smallBz$}\label{ss-defSBZ}
In this section we discuss balancing structures on small quantum Borel algebras which may require an index 2 extension
of the algebra $\mathbf u^{0,+}$. We begin with an identity for the highest root lattice vector $\ttau$
given in (\ref{eq-gradBz}).
\begin{lemma} \label{lm-taualphi}
Using conventions as in Section~\ref{ss-LusBor} above we have 
 \begin{equation}\label{eq-taualpha}
  (\alpha_i,\ttau)=lcm(\ell,\|\alpha_i\|^2)-\|\alpha_i\|^2=lcm(\ell,2d_i)-2d_i=2d_i(\ell_i-1)\,.
 \end{equation} 
\end{lemma}
\begin{proof}
 Recall that an elementary Weyl reflection $s_i$ about a simple root $\alpha_i$ is given by
 $s_i(\alpha)=\alpha-\langle \alpha_i,\alpha\rangle\alpha_i$ and permutes
 the set of positive roots in $\Phi^+-\{\alpha_i\}$ among each other and maps $\alpha_i$ to $-\alpha_i$. 
 Moreover, we note that $\ell_{\alpha}=\ell_{s_i(\alpha)}$ as these number were determined by their 
 $W$-orbits. Consequently, we find that for the root lattice element $\ttau'=\sum_{\alpha\in\Phi^+-\{\alpha_i\}}(\ell_{\alpha}-1)\alpha$
 we have  $s_i(\ttau')=\ttau'$. Since  $\ttau=\ttau'+(\ell_i-1)\alpha_i$ this implies 
  $s_i(\ttau)=\ttau-2(\ell_i-1)\alpha_i$ so that $\langle \alpha_i,\ttau\rangle=2(\ell_i-1)$.
  Thus $(\alpha_i,\ttau)=\frac 12 (\alpha_i,\alpha_i)\langle \alpha_i,\ttau\rangle
  =(\alpha_i,\alpha_i)(\ell_i-1)=2d_i(\ell_i-1)$. We further note 
  $2d_i\ell_i=2d_i\frac {\ell}{gcd(2d_i\ell)}=lcm(2d_i,\ell)\,$ which implies (\ref{eq-taualpha}).
\end{proof}
The argument above is analogous to that for the sum of positive roots $2\rho$ as in
Proposition~29 in Ch.VI$\parag 1.10$ of \cite{bour}, which shows that $(2\rho,\alpha_i)=(\alpha_i,\alpha_i)$.
A consequence of (\ref{eq-taualpha}) is thus that
\begin{equation}\label{eq-taurho}
 (\alpha_i,\ttau)\,\equiv\,-2d_i\,=-\,2(\alpha_i,\rho)\;\mod \ell\;.
\end{equation}
For the purpose of constructing a balancing structure we would like to ensure that $\ttau$ is a multiple of 2
in $\overline \rootl$. Equation~(\ref{eq-taurho}) suggests that $-\rho$ may be a choice for
representing half of $\ttau$. However, $\rho\in \rootl$ only for about half of the simple Lie types, namely,
$A_{2k}$, $C_{4k}$, $C_{4k+3}$, $D_{4k}$, $D_{4k+1}$, $E_6$, $E_8$, and $G_2$. This is apparent from the 
explicit forms for $2\rho$ for all Lie types that can be found in Plaches I-IX in \cite{bour}.

For the general case we need to consider  an index 2 extension of the abelian group $\overline \rootl$ and 
the respective group algebra $\mathbf u^{0}$ over $\mathbb Z[\zeta]$. The group extension  is given by 
$\widehat\rootl=\overline \rootl\oplus\mathbb Z\gamma/\langle 2\gamma-\ttau\rangle$. The respective extension $\widehat{\mathbf u}^0$
of $\mathbf u^{0}$ is given by introducing an additional group-like generator $J$ (alternatively denoted $K^{\gamma}$)
with relation and corelation
\begin{equation}\label{eq-Jrel}
 J^2=K^{\ttau}\qquad \mbox{and}\qquad \Delta(J)=J\otimes J\,.
\end{equation}
The inclusion of $J$ into the small Borel algebra still requires an extension of the homomorphism 
$\crass:\overline Q\to \mathrm{HAut}(\mathbf  u^{+})$ to $\widehat\rootl$, which is given on an element
$X\in\mathbf u^+_{\nu}$ with $\nu=\sum_ic_i\alpha_i$ by 
\begin{equation}\label{eq-Jact}
 \crass_{\gamma}(X)=JXJ^{-1}=\zeta^{-(\rho,\nu)}X=\zeta^{-\sum_ic_id_i}X\,,
\end{equation}
where we use  $(\rho,\alpha_i)=\frac 12 (\alpha_i,\alpha_i)=d_i\in\mathbb Z$. It follows from 
(\ref{eq-taurho}) that $\crass$ respects the relation $\ttau=2\gamma$ and is thus well defined 
on $\widehat\rootl$. The following proposition introduces the extended version of Lusztig's
small quantum Borel algebra that we want to consider and summarizes the applicable observations
of this and the previous section.
\begin{prop}\label{prop-defextsqb}
Let $\smallBz$ be defined as $\mathbf u^{0,+}$ with an additional group-like
generator $J$ and additional relations as in (\ref{eq-Jrel}) as well as $JE_iJ^{-1}=\zeta^{-d_i}E_i$.
Then $\smallBz$ is a $\rootl$-graded Hopf algebra with grading $\smallBz_{\nu}=\widehat{\mathbf u}^0{\mathbf u}^+_{\nu}$.

As in (\ref{eq-gradBz}) we have $\smallBz_{\nu}\neq 0$ only if $0\preccurlyeq\nu\preccurlyeq\ttau$ and 
$\smallBz_{\ttau}=\widehat{\mathbf u}^0E^{\tops}$ so that $E_i\smallBz_{\ttau}=0=\smallBz_{\ttau}E_i$.
Moreover, using the notation $K^{\gamma}=J$ with $\gamma\in\widehat\rootl$ as above, we have that 
$\{K^{\beta}E^{\psi}:\beta\in  \widehat \rootl\,,\,\psi\in\mathcal X_{\ell}\}$ is a PBW basis of
$\smallBz$, and that the coproduct relations from  
 (\ref{eq-coprodgrad}) holds for all $\beta\in \widehat \rootl$. 
\end{prop}

As noted previously the statements in Proposition~\ref{prop-defextsqb} and the 
following discussion will also be valid if omitted the additional generators $J$
for the Lie types listed above for which $\rho\in\rootl$. In this case the Hopf algebra
splits as a $\smallBz=\mathbf u^{0,+}\otimes \mathcal Z$, where $\mathcal Z$ is the group algebra 
generated by the central element $Z=JK^{\rho}$ and which may thus be discarded without affecting
the balancing structure. 

\subsection{Double Balancing of $\smallBz$}\label{ss-balSBZ}
In this section we discuss double balancings on
$\smallBz$ as first described in (\ref{eq-balanced}) of the introduction and defined following 
Lemma~\ref{lm-DoubBal} of Section~\ref{ss-ribbon-double}. This requires the computation of 
integral and moduli. We start by denoting the standard two-sided integral of the group
algebra $\widehat{\mathbf u}^{0}\cong\mathbb Z[\zeta][\widehat\rootl]$ given by the 
summation of all of the group element.
\begin{equation}
 \grint=\sum_{\beta\in\widehat \rootl}K^{\beta}=(1+J)\prod_i(\sum_{j=0}^{2\ell_i-1}K_i^j)\,.
\end{equation}

\begin{lemma}\label{lm-sbz-cointcomod}
A non-zero left integral in $\smallBz$ is given by 
\begin{equation}\label{eq-sbz-coint}
 \Lambda=\grint E^{\tops} 
\end{equation}
and the respective comodulus is the algebra homomorphism $\comod:\smallBz\to \mathbb Z[\zeta]$ 
defined on generators by 
\begin{equation}\label{eq-sbz-comod}
 \comod(E_i)=0\,,\quad \comod(K_i)=\zeta^{2d_i}\,, \mbox{ \ \ \ and \ \ \ } \comod(J)=\zeta^{(\rho,\ttau)}\,.
\end{equation}
\end{lemma}

\begin{proof}
 We note that $K^{\beta}\Lambda=K^{\beta}\grint E^{\tops}=\grint E^{\tops}=\Lambda=\varepsilon(K^{\beta})\Lambda$ 
 for all $\beta\in\widehat \rootl$ since $\grint$ is a two-sided 
 integral for the subalgebra of group like elements. Also $E_i\Lambda=0=\varepsilon(E_i)\Lambda$ by 
 Proposition~\ref{prop-defextsqb} above so that $\Lambda$ is indeed a left integral. 
 
 For the comodulus note that (\ref{eq-XKgradrel}) specializes to $K^{\beta}E^{\tops}K^{-\beta}=\zeta^{(\beta,\ttau)}E^{\tops}$ 
 so that $\Lambda K^{\beta}=\grint E^{\tops} K^{\beta}=\zeta^{-(\beta,\ttau)}\grint K^{\beta} E^{\tops} 
 =\zeta^{-(\beta,\ttau)}\grint E^{\tops} 
 =\zeta^{-(\beta,\ttau)}\Lambda$, and hence $\comod(K^{\beta})=\zeta^{-(\beta,\ttau)}$. For $K_i=K^{\alpha_i}$ 
 the expression in (\ref{eq-sbz-comod}) now follows from Lemma~\ref{lm-taualphi}. The expression for
 $J$ follows analogously from (\ref{eq-Jact}). Finally, Proposition~\ref{prop-defextsqb} also implies
 $\Lambda E_i=0$.
\end{proof}

\begin{lemma}\label{lm-sbz-intmod}
A right integral on $\smallBz$ is given on the PBW basis described in Proposition~\ref{prop-defextsqb}  by
\begin{equation}\label{eq-sbz-int}
 \lambda(K^{\beta}E^{\psi})=\left\{
 \begin{array}{cl}
 1 & \mbox{\ if \ } \beta=0 \mbox{\ and \ } \psi=\psi_{\tops}\\
 0 & \mbox{\ otherwise}
 \end{array}\right.
\end{equation}
and the respective modulus is  
\begin{equation}\label{eq-sbz-mod}
\lgl =K^{\ttau}\;.
\end{equation}
\end{lemma}

\begin{proof} Note that (\ref{eq-sbz-int}) implies that $\lambda(X)=0$ if $X\in \smallBz_{\nu}$ with $\nu\prec\ttau$.
Hence also $(\lambda\otimes id)(\Delta(X))=(id\otimes\lambda )(\Delta(X))=0=\lambda(X)1$ since $\Delta(X)$ contains only terms with 
bigrading strictly less than $\ttau$ in each component. We are thus left checking the integral identity on $\smallBz_{\ttau}$, that is,
terms of the form $K^{\beta}E^{\tops}$. To this end consider (\ref{eq-coprodgrad}) for $\psi=\psi^{\tops}$. Since, again, 
components of the bidegrees of the terms of $M$ are less than $\ttau$ we have $(\lambda\otimes id)(M)=0=(id\otimes\lambda )(M)$,
and clearly also $\lambda(K^{\beta})=0$. Using $\ttau=\hat\psi^{\tops}$ we thus find 
$$
( \lambda \otimes id)(\Delta(K^{\beta}E^{\tops}))=\lambda(K^{\beta}E^{\tops})K^{\beta} 
\mbox{\ \ and \ \ } 
( id \otimes \lambda)(\Delta(K^{\beta}E^{\tops}))=K^{\beta+\ttau} \lambda(K^{\beta}E^{\tops})\;.
$$
The first equation reduces the integral identity to $\lambda(K^{\beta}E^{\tops})(K^{\beta}-1)=0$ which is
fulfilled with the definition in (\ref{eq-sbz-int}). The comodulus identity and the
second equation imply for $\beta=0$ that
$\lambda(E^{\tops})(K^{\ttau}-\lgl)=0$ and hence (\ref{eq-sbz-mod}).
 
\end{proof}

\begin{lemma}\label{lm-Hautsqs} We have the following equality of automorphisms in $\mathrm{HAut}(\smallBz)$:
 \begin{equation}\label{eq-Hautsqs}
  \crass_{\ttau}=\crass_{\gamma}^2=S^2=ad(\lgl)=ad^*(\comod)\,.
 \end{equation} 
\end{lemma}

\begin{proof} The relation $\crass_{\ttau}=\crass_{\gamma}^2$ is immediate from the fact that $\crass$ is a homomorphism.
For the remaining automorphisms it is readily checked that they are all identity on the group like generators so that
it suffices to verify that their value on $E_i$ is indeed $\crass_{\ttau}(E_i)=\zeta^{(\ttau,\alpha_i)}E_i=\zeta^{-2d_i}E_i$
by (\ref{eq-XKgradrel}) and (\ref{eq-taurho}). 

For the next identity note that (\ref{eq-uqb-corels}) implies that $S(K_i)=K_i^{-1}$ and $S(E_i)=-K_i^{-1}E_i$ so that 
$S^2(E_i)=K_i^{-1}E_iK_i=\zeta^{-2d_i}E_i$. Also $ad(g)(E_i)=ad(K^{\ttau})(E_i)=K^{\ttau} E_iK^{-\ttau}=\crass_{\ttau}(E_i)$
by definition. Finally, we note that an algebra homomorphism $\chi:\smallBz\to \mathbb{Z}[\zeta]$ (that is, a
group-like element $\chi\in \smallBz^*$) is zero on nilpotent generators $E_i$ and give by a character in 
$\mathrm{Hom}(\widehat \rootl,\mathbb Z[\zeta]^{\times})$. Using notation from (\ref{e:arrows}) and again (\ref{eq-uqb-corels}) 
we find that the coadjoint action of $\chi$ on $E_i$ is given as
\begin{equation}\label{eq-coadj-ubz}
\begin{split}
 ad^*(\chi)(E_i)&=\chi\rar E_i\lar \chi^{-1}=\chi^{-1}\otimes id\otimes \chi(\Delta^{2}(E_i))\\
& =\chi^{-1}(K_i)E_i\chi(1)=\chi(K_i)^{-1}E_i\;.
\end{split}
\end{equation}
The last identity now follows by combining (\ref{eq-coadj-ubz}) and (\ref{eq-sbz-comod}).
\end{proof}

Note that (\ref{eq-Hautsqs}) immediately implies Radford's expression for $S^4$ as discussed in the introduction. The
double balancing elements described there are given as follows.

\begin{lemma}\label{lm-sbz-bal} Suppose there is an integer $\sqw\in\mathbb Z$ such that $(\rho,\ttau)\equiv 2\sqw \mod \ell$.
Then $\smallBz$ admits a double balancing in the sense of (\ref{eq-balanced}) with $\lblz=J$ and $\hcomod$ given
on generators by
\begin{equation}\label{eq-sbz-bal}
 \hcomod(E_i)=0\,,\quad \hcomod(K_i)=\zeta^{d_i}=\zeta^{(\alpha_i,\rho)}\,, \mbox{ \ \ \ and \ \ \ } \ztheta=\hcomod(J)=\zeta^{\sqw}\,.
\end{equation}
\end{lemma}

\begin{proof} 
We first note that $\hcomod$ is well defined on $\widehat\rootl$ (and hence on $\smallBz$). In particular the second
identity implies $\hcomod(K_i^{2l_i})=\zeta^{2d_i\ell_i}=1$ (by definition of $\ell_i$) as well as 
$\hcomod(K^{\beta})=\zeta^{(\beta,\rho)}$ for $\beta\in\overline{\rootl}$ so that
 also $\hcomod(J^2)=\zeta^{2\sqw}=\zeta^{(\ttau,\rho)}=\hcomod(K^{\ttau})$.
The relations $\lblz^2=J^2=K^{\ttau}=\lgl$ and $\hcomod^2=\comod$ are immediate by construction.

Further, (\ref{eq-Jact}) implies that $Ad(\lblz)=Ad(J)=\crass_{\gamma}$. By (\ref{eq-coadj-ubz}) and (\ref{eq-Jact}) 
we have $ad^*(\hcomod)(E_i)=\hcomod(K_i)^{-1}E_i=\zeta^{-d_i}E_i=\crass_{\gamma}(E_i)$ and hence $ad^*(\hcomod)=\crass_{\gamma}$.
Lemma~\ref{lm-Hautsqs} we now have $ad(\lblz)\circ ad^*(\hcomod)=\crass_{\gamma}^2=S^2$, as required in (\ref{eq-balanced}).
\end{proof}

We end this section by compiling in Table~\ref{tb-lie} formulae for the expression  $(\rho,\ttau)$ occurring in Lemma~\ref{lm-sbz-bal} as this
determines balancing as well as the framing anomaly $\ztheta$. To this end 
note that (\ref{eq-taualpha}) implies
\begin{equation}\label{eq-rhotauform}
 (\rho,\ttau)=\sum_id_i(\ell_i-1)s_i \qquad \mbox{\ \ for \ \ }\qquad  2\rho=\sum_is_i\alpha_i\;.
\end{equation}
The coefficients $s_i$ for $2\rho$ can be found again in the Plaches I-IX in \cite{bour}. Below we
further denote $\dOne=\min\{d_i:i=1,\ldots, n\}$ as well as $\ellOne$, $\ellTwo$, and $\ellThree$
the orders of $\zeta^{2\dOne}$, $\zeta^{4\dOne}$, and $\zeta^{6\dOne}$ respectively.
\begin{table}
\begin{center} 
\begin{tabular}{ccl}
 Lie Type & & $(\rho,\ttau)$ \\
 \hline
$A_n$ &&  $\dOne(\ellOne -1){\binom {n+2} 3}$\rule{0mm}{7mm}\\ 
$B_n$ &&  $2\dOne(\ellTwo -1)\left[3{\binom {n+1}  3}+{\binom {n}  3}\right]+\dOne(\ellOne-1)n^2$\rule{0mm}{7mm}\\ 
$C_n$ &&  $4\dOne(\ellOne -1){\binom {n+1} 3}+2\dOne(\ellTwo-1){\binom {n+1} 2}$\rule{0mm}{7mm}\\
$D_n$ &&  $2\dOne(\ellOne -1)\left[{\binom {n+1} 3}+{\binom {n} 3}\right]$\rule{0mm}{7mm}\\ 
$E_6$  && $156\,\dOne(\ellOne -1)$\rule{0mm}{7mm}\\ 
$E_7$  && $399\,\dOne(\ellOne -1)$\rule{0mm}{7mm}\\ 
$E_8$  && $1240\,\dOne(\ellOne -1)$\rule{0mm}{7mm}\\ 
$F_4$ &&  $4\dOne[16(\ellOne-1)+23(\ellTwo-1)]$\rule{0mm}{7mm}\\ 
$G_2$ &&  $2\dOne[5(\ellOne-1)+9(\ellThree-1)]$\rule{0mm}{7mm}\\ 
&&\\
\end{tabular} 
\caption{}
\label{tb-lie}
\end{center}
\end{table}
We next list a few situations in which the prerequisite of Lemma~\ref{lm-sbz-bal} is fulfilled
by inspection of the above formulae.
\begin{cor}\label{cor-baldiv2}
The algebra $\smallBz$ is double balanced in any of the following cases:
\begin{enumerate}
 \item The Lie Type of $\mathfrak g$ is not $A_{4k+1}$, $B_{2k+1}$, or $E_7$.
 \item $\ell$ is odd or twice an odd integer.
 \item $\dOne$ is even.
\end{enumerate}
\end{cor}

\begin{proof}
For (1) we note that the expressions for Lie types $C_n$, $D_n$, $E_6$, $E_8$, $F_4$, 
and $G_2$ are obviously even. The assertion follows by observing that 
$B_n$ the expression is even if $n$ is even, and the
binomial coefficient for $A_n$ is even if $n\not\equiv 1\mod 4$.Thus 
with the stated exceptions we obtain an even expression for any choice of $\ell$ or $d_i$'s.

Clearly if $\ell$ is odd  the prerequisite of Lemma~\ref{lm-sbz-bal} is always true. If $\ell$
is twice an odd number then $\ellOne$ is an odd number so that the factor
$(\ellOne-1)$ is even. Other factors in the above table such as 
$(\ellTwo-1)$ and $(\ellThree-1)$ are multiplied with even expressions.
Statement (3) is obvious, and may apply to all simply laced types, namely A, D, and E. 
\end{proof}

Thus there are only few situations in which $\smallBz$ does not fulfill the prerequisite of 
Lemma~\ref{lm-sbz-bal}, such as for $B_{2k+1}$ with
$\ell$ divisible by 4, and here a balancing is a-prior not provided by the methods described.

\subsection{Proof of Theorem~\ref{thm-Bzg=Frob+Bal} and Notes on 
   TQFT's Associated to $\smallBz$}\label{ss-TQFT-SBZ}

\begin{proof}[Proof of Theorem~\ref{thm-Bzg=Frob+Bal}]
 The existence of a finite PBW basis for $\smallBz$ as in Proposition~\ref{prop-defextsqb} 
  implies that$\smallBz$ is a finite algebra over $\mathbb Z[\zeta]$. It is easy to check 
  that the integrals found in Lemma~\ref{lm-sbz-cointcomod} and Lemma~\ref{lm-sbz-intmod} 
  fulfill $\lambda(\Lambda)=1$ so that $\smallBz=\mathbb Z[\zeta]\Lambda\oplus 
  \mathrm{ker}(\lambda)$ as $\mathbb Z[\zeta]$-modules.
  It is well known that $\mathbb Z[\zeta]$ is a Dedekind domain since it is the ring of
  integers of the cyclotomic field $\mathbb Q(\zeta)$ for any integer  $\ell>1$, 
  see for example Theorem~2.6 in \cite{wa97}.
  
  Lemma~\ref{lm-Frob-subcrit} now implies that $\smallBz$ is a Frobenius Hopf algebra and 
  that $N=\mathbb Z[\zeta]\Lambda$ coincides with $\lint{\smallBz}$. The assertions on 
  balancing in Theorem~\ref{thm-Bzg=Frob+Bal} follow directly from Lemma~\ref{lm-sbz-bal}.
\end{proof}

We conclude our discussion with a few more remarks and open questions on Hennings 
TQFTs constructed from the previously introduced quantum algebras. 
 
We first note that the $\mathfrak{sl}_2$-algebra $\mathcal D(B)$ constructed and discussed in 
Section~\ref{s4} is mildly different from the version $\mathcal D(B'_{\zeta}(\mathfrak{sl}_2))$.
Particularly, the generators of $B$ are (up to Cartan generators) ordinary powers $e^n$ and
divided powers $f^{(n)}$ provide the basis in the dual algebra $B^*$. Conversely, for $\smallBz$ 
we start with divided powers $E^{(n_{\alpha})}_{\alpha}$ which leads to a description of the 
dual of $\smallBz$ in terms of ordinary powers $F_{\alpha}^{n_{\alpha}}$. The two descriptions 
are thus expected to lead to equivalent TQFTs as the differ only by switching the roles of
the Borel subalgebras.

The use of the additional generator $J$ is also circumvented in the definition of $B$ in the 
$\mathfrak{sl}_2$-case since we choose $\ell$ to be odd to begin with and since the only entry 
in the Cartan matrix is even.

We also conjecture that the results from Proposition~\ref{lm-DBFfact} and 
Theorem~\ref{thm-TQFT-DB=UA} can be generalized to the case of $\smallBz$. In fact, such a 
factorization for doubles of small quantum Borel algebras in the language and context of 
representation categories has been given \cite{aegn}.

\end{document}